\documentclass[journal,onecolumn,draftcls]{IEEEtran}

\usepackage{braket}
\usepackage{array}

\usepackage{pgfplots}

\usepackage{hyperref}       
\usepackage{url}            
\usepackage{booktabs}       
\usepackage{amsfonts}       
\usepackage{nicefrac}       
\usepackage{microtype}      
\usepackage{yfonts,dsfont}
\usepackage{setspace}

\usepackage{subfigure}

\usepackage{cite}
\usepackage{pifont}

\usepackage{mathtools}
\usepackage{mathcomp}
\usepackage{euscript,bbm,amstext,wasysym,arydshln,framed}

\usepackage{tikz}
\usepackage{amssymb}
\usepackage{verbatim}

\usepackage{algorithm}
\usepackage{algorithmic}

\usetikzlibrary{positioning,shapes,arrows,calc,matrix}

\newcommand{\enma}[1]   {\ensuremath{#1}}

\newcommand{\beq}{\begin{equation}}
\newcommand{\eeq}{\end{equation}}
\newcommand{\bseq}{\begin{subequations}}
\newcommand{\eseq}{\end{subequations}}
\newcommand{\beqn}{\begin{eqnarray}}
\newcommand{\eeqn}{\end{eqnarray}}
\newcommand{\ba}{\begin{array}}
\newcommand{\ea}{\end{array}}
\newcommand{\bct}{\begin{center}}
\newcommand{\ect}{\end{center}}
\newcommand{\btmz}{\begin{itemize}}
\newcommand{\etmz}{\end{itemize}}
\newcommand{\benum}{\begin{enumerate}}
\newcommand{\eenum}{\end{enumerate}}

\newcommand{\cF}{\enma{\mathcal F}}

\newcommand{\norm}[1]{\| #1 \|}                 

\newcommand{\trace}     {\enma{\mathrm{trace}}}

\newcommand{\inner}[2]{\left\langle #1,#2 \right\rangle}

\newcommand{\matbegin}{
        \left[
}
\newcommand{\matend}{
        \right]
}

\newcommand{\tbo}[2]{
  \matbegin \begin{array}{c}
       #1 \\ #2
       \end{array} \matend }

\newcommand{\tbt}[4]{
  \matbegin \begin{array}{cc}
       #1 & #2 \\ #3 & #4
       \end{array} \matend }

\newcommand{\be}{\begin{equation}}
\newcommand{\ee}{\end{equation}}

\newcommand{\cplxs}{ C\kern -.35em \rule{0.03 em}{.7 ex}~   }

\def\complex{\hbox{C\kern -.45em \rule{0.03 em}{1.5 ex}}~}

\newcommand{\bi}{\begin{itemize}}
\newcommand{\ei}{\end{itemize}}

\DeclareMathOperator{\EX}{\mathbb{E}}

\newcommand{\R}{\mathbb{R}}

\newcommand{\DefinedAs}[0]{\mathrel{\mathop:}=}
\newcommand{\AsDefined}[0]{=\mathrel{\mathop:}}

\DeclareMathOperator*{\minimize}{minimize}

\DeclareMathOperator*{\subject}{subject~to}

\newtheorem{mythm}{Theorem}
\newtheorem{myprop}{Proposition}
\newtheorem{mylem}{Lemma}
\newtheorem{myrem}{Remark}
\newtheorem{myass}{Assumption}

\newcommand{\one}{\mathds{1}}

\newcommand{\abs}[1]{\left|#1\right|}

\newcommand{\vect}     {\enma{\mathrm{vec}}}

\renewcommand{\baselinestretch}{1.37}

\newcommand{\cA}{{\cal A}}

\newcommand{\cD}{{\cal D}}

\newcommand{\cB}{{\cal B}}
\newcommand{\cS}{{\cal S}}
\newcommand{\cM}{{\cal M}}

\newcommand{\bbR}{\mathbb{R}}

\newcommand{\bbS}{\mathbb{S}}

\newcommand{\non}{\nonumber}
\newcommand{\ds}{\displaystyle}

\newcommand{\hP}{\hat{P}}
\newcommand{\hK}{\hat{K}}
\newcommand{\mrd}{\mathrm{d}}
\newcommand{\mre}{\mathrm{e}}

\newcommand{\tY}{\tilde{Y}}
\newcommand{\tX}{\tilde{X}}
\newcommand{\tK}{\tilde{K}}
\newcommand{\tP}{\tilde{P}}

\newcommand{\cSK}{\mathcal{S}_K}
\newcommand{\cSY}{\mathcal{S}_Y}

\newcommand{\bbP}{\mathbb{P}}
\newcommand{\sfA}{\mathsf{A}}
\newcommand{\sfB}{\mathsf{B}}
\newcommand{\sfC}{\mathsf{C}}
\newcommand{\sfD}{\mathsf{D}}

\newcommand{\sfM}{\mathsf{M}}

\makeatletter
\newcommand*\rel@kern[1]{\kern#1\dimexpr\macc@kerna}
\newcommand*\widebar[1]{%
	\begingroup
	\def\mathaccent##1##2{%
		\rel@kern{0.8}%
		\overline{\rel@kern{-0.8}\macc@nucleus\rel@kern{0.2}}%
		\rel@kern{-0.2}%
	}%
	\macc@depth\@ne
	\let\math@bgroup\@empty \let\math@egroup\macc@set@skewchar
	\mathsurround\z@ \frozen@everymath{\mathgroup\macc@group\relax}%
	\macc@set@skewchar\relax
	\let\mathaccentV\macc@nested@a
	\macc@nested@a\relax111{#1}%
	\endgroup
}

\IEEEoverridecommandlockouts                              
\begin{document}

\title{\huge Convergence and sample complexity of gradient methods \\[-0.15cm] for the model-free linear quadratic regulator problem}

\author{Hesameddin Mohammadi, Armin Zare, Mahdi Soltanolkotabi, and Mihailo R.\ Jovanovi\'c
\thanks{The work of H.\ Mohammadi, A.\ Zare, and M.\ R.\ Jovanovi\'c is supported in part by the National Science Foundation (NSF) under Awards ECCS-1708906 and ECCS-1809833 and the Air Force Office of Scientific Research (AFOSR) under Award FA9550-16-1-0009. The work of M.\ Soltanolkotabi is supported in part by the Packard Fellowship in Science and Engineering, a Sloan Research Fellowship in Mathematics, a Google Faculty Research Award, as well as Awards from NSF, Darpa LwLL program, and AFOSR Young Investigator Program.}
\thanks{H.\ Mohammadi, M.\ Soltanolkotabi, and M.\ R.\ Jovanovi\'c are with the Ming Hsieh Department of Electrical and Computer Engineering, University of Southern California, Los Angeles, CA 90089. A.\ Zare is with the Department of Mechanical Engineering, University of Texas at Dallas, Richardson, TX 75080. Emails: \{hesamedm, soltanol, mihailo\}@usc.edu, armin.zare@utdallas.edu.}
}

\maketitle



	\vspace*{-4ex}
\begin{abstract} 
	Model-free reinforcement learning attempts to find an optimal control action for an unknown dynamical system by directly searching over the parameter space of controllers. The convergence behavior and statistical properties of these approaches are often poorly understood because of the nonconvex nature of the underlying optimization problems and the lack of exact gradient computation. In this paper, we take a step towards demystifying the performance and efficiency of such methods by focusing on the standard infinite-horizon linear quadratic regulator  problem for continuous-time systems with unknown state-space parameters. We establish exponential stability for the ordinary differential equation (ODE) that governs the gradient-flow dynamics over the set of stabilizing feedback gains and show that a similar result holds for the gradient descent method that arises from the forward Euler discretization of the corresponding ODE. We also provide theoretical bounds on the convergence rate and sample complexity of the random search method with two-point gradient estimates. We prove that the required simulation time for achieving $\epsilon$-accuracy in the model-free setup and the total number of function evaluations both scale as $\log \, (1/\epsilon)$.
\end{abstract}
 \vspace*{-2ex}

\vspace*{-4ex}
\section{Introduction}
\label{sec.intro}

In many emerging applications, control-oriented models are not readily available and classical approaches from optimal control may not be directly applicable. This challenge has led to the emergence of Reinforcement Learning (RL) approaches that often perform well in practice. Examples include learning complex locomotion tasks via neural network dynamics~\cite{nagkahfealev18} and playing Atari games based on images using deep-RL~\cite{mnikavsilgraantwierie13}. 

RL approaches can be broadly divided into model-based~\cite{deamanmatrectu17,simmantujorrec18} and model-free~\cite{ber11,abblazsze18}. While model-based RL uses data to obtain approximations of the underlying dynamics, its model-free counterpart prescribes control actions based on estimated values of a cost function without attempting to form a model. In spite of the empirical success of RL in a variety of domains, our mathematical understanding of it is still in its infancy and there are many open  questions surrounding convergence and sample complexity. In this paper, we take a step towards answering such questions with a focus on the infinite-horizon Linear Quadratic Regulator (LQR) for continuous-time systems.

The LQR problem is the cornerstone of control theory. The globally optimal solution can be obtained by solving the Riccati equation and efficient numerical schemes with provable convergence guarantees have been developed~\cite{andmoo90}. However, computing the optimal solution becomes challenging for large-scale problems, when prior knowledge is not available, or in the presence of structural constraints on the controller. This motivates the use of direct search methods for controller synthesis. Unfortunately, the nonconvex nature of this formulation complicates the analysis of first- and second-order optimization algorithms. To make matters worse, structural constraints on the feedback gain matrix may result in a disjoint search landscape limiting the utility of conventional descent-based methods~\cite{ack80}. Furthermore, in the model-free setting, the exact model (and hence the gradient of the objective function) is unknown so that only zeroth-order methods can be used.

In this paper,  we study convergence properties of gradient-based methods for the continuous-time LQR problem. In spite of the lack of convexity, we establish (a)  {\em exponential stability\/} of the ODE that governs the gradient-flow dynamics over the set of stabilizing feedback gains; and (b) {\em linear convergence\/} of the gradient descent algorithm with a suitable stepsize. We employ a standard convex reparameterization for the LQR problem~\cite{ferbalboyelg92,dulpag00} to establish the convergence properties of gradient-based methods for the nonconvex formulation. In the model-free setting, we also examine convergence and sample complexity of the random search method~\cite{manguyrec18} that attempts to emulate the behavior of gradient descent via gradient approximations resulting from objective function values. For the two-point gradient estimation setting, we prove linear convergence of the random search method and show that the total number of function evaluations and  the simulation time required in our results to achieve $\epsilon$-accuracy are proportional to $\log \, (1/\epsilon)$.

For the {\em discrete-time\/} LQR, global convergence guarantees were recently provided in~\cite{fazgekakmes18} for gradient decent and the random search method with one-point gradient estimates. The authors established a bound on the sample complexity for reaching the error tolerance $\epsilon$ that requires a number of function evaluations that is at least proportional to $(1/\epsilon^4) \log \, (1/\epsilon)$. If one has access to the infinite-horizon cost values, the number of function evaluations for the random search method with one-point gradient estimates can be improved to $1/\epsilon^2$~\cite{malpanbhakhabarwai19}. In contrast, we focus on the {\em continuous-time\/} LQR and  examine the two-point gradient estimation setting. The use of two-point gradient estimates reduces the required number of function evaluations to $1/\epsilon$~\cite{malpanbhakhabarwai19}. We significantly improve this result by showing that the required number of function evaluations is proportional to $\log \, (1/\epsilon)$. Similarly, the simulation time required in our results is proportional to $\log \, (1/\epsilon)$; this is in contrast to~\cite{fazgekakmes18} that requires $\mathrm{poly} \, (1/\epsilon)$ simulation time and~\cite{malpanbhakhabarwai19} that assumes an infinite simulation time. Furthermore, our convergence results hold both in terms of the error in the objective value and the optimization variable (i.e., the feedback gain matrix) whereas~\cite{fazgekakmes18} and~\cite{malpanbhakhabarwai19} only prove convergence in the objective value. We note that the literature on model-free RL is rapidly expanding and recent extensions to Markovian jump linear systems~\cite{janhudul20}, $\mathcal{H}_{\infty}$ robustness analysis through implicit regularization~\cite{zhahubas20}, learning distributed LQ problems~\cite{furzhekam20}, and output-feedback LQR~\cite{fatpol20} have been made.

Our presentation is structured as follows. In Section~\ref{sec.ProblemFormulation}, we revisit the LQR problem and present gradient-flow dynamics, gradient descent, and the random search algorithm. In Section~\ref{sec.MainResult}, we highlight the main results of the paper. In Section~\ref{sec.convex-char}, we utilize convex reparameterization of the LQR problem and establish exponential stability of the resulting gradient-flow dynamics and gradient descent method. In Section~\ref{sec.non-convex-char}, we extend our analysis to the nonconvex landscape of feedback gains. In Section~\ref{sec.BiasVariance}, we quantify the accuracy of two-point gradient estimates and, in Section~\ref{sec.rndSearch}, we discuss convergence and sample complexity of the random search method. In Section~\ref{sec.examples}, we provide an example to illustrate our theoretical developments and, in Section~\ref{sec.conclusions}, we offer concluding remarks. Most technical details are relegated to the appendices. 

\subsubsection*{Notation}
We use $\vect(M)\in\R^{mn}$ to denote the vectorized form of the matrix $M\in\R^{m \times n}$ obtained by concatenating the columns on top of each other. 
We use $\norm{M}_F^2 = \inner{M}{M}$ to denote the Frobenius norm, where $\inner{X}{Y}\DefinedAs \trace \, (X^TY)$ is the standard matricial inner product. We denote the largest singular value of linear operators and matrices by $\norm{\cdot}_2$ and the spectral induced norm of linear operators by $\norm{\cdot}_S$ 
\begin{align*}
\norm{\cM}_{2} 
\;\DefinedAs\; \sup_M\dfrac{\norm{\cM(M)}_F}{\norm{M}_F},
~~~
\norm{\cM}_S 
\;\DefinedAs\; \sup_M\dfrac{\norm{\cM(M)}_2}{\norm{M}_2}.
\end{align*}
We denote by $\bbS^n\subset\R^{n\times n}$ the set of symmetric matrices. For $ M\in\bbS^n$, $M\succ 0$ means $M$ is positive definite and $\lambda_{\min}(M)$ is the smallest eigenvalue. We use $S^{d-1}\subset\R^d$ to denote the unit sphere of dimension $d-1$. We  denote the expected value by $\EX[\cdot ]$ and probability by $\bbP(\cdot)$. To compare the asymptotic behavior of $f(\epsilon)$ and $g(\epsilon)$ as $\epsilon$ goes to $0$, we use $f=O(g)$ (or, equivalently, $g=\Omega(f)$) to denote $\limsup_{\epsilon \, \rightarrow \, 0}f(\epsilon)/g(\epsilon) < \infty$;  $f=\tilde{O}(g)$ to denote $f=O(g \log^k \! g)$ for some integer $k$; and $f=o(\epsilon)$ to signify $\lim_{\epsilon \, \rightarrow \, 0}f(\epsilon)/\epsilon = 0$.

	\vspace*{-1ex}
\section{Problem formulation}
\label{sec.ProblemFormulation}

The infinite-horizon LQR problem for continuous-time LTI systems is given by
\begin{subequations}
	\label{eq.LQR}
	\begin{align}
	\label{eq.lqrEXform}
	\minimize\limits_{x,\,u}
	& ~~
	\EX 
	\left[ \, \int_{0}^{\infty} (x^T (t) Q x(t) \, + \, u^T(t) R u (t)) \, \mrd t \, \right]
	\\[0.cm]
	\label{eq.linsys1}
	\subject 
	& ~~
	\dot{x} 
	\; = \; 
	A x \; + \; B u, 
	~~
	x(0) \, \sim \, \cD
	\end{align}
\end{subequations}
where $x(t)\in\R^n$ is the state, $u(t)\in\R^m$ is the control input, $A$ and $B$ are constant matrices of appropriate dimensions, $Q$ and $R$ are positive definite matrices, and the expectation is taken over a random initial condition $x(0)$ with distribution $\cD$. 
For a controllable pair $(A,B)$, the solution to~\eqref{eq.LQR} is given~by 
\begin{subequations}
	\label{eq.K-ARE}
	\begin{align}
	u (t) \; = \; -K^\star x (t) \; = \; - R^{-1} B^T P^\star x (t)
	\label{eq.u-star}	
	\end{align}
	where $P^\star$ is the unique positive definite solution to the Algebraic Riccati Equation (ARE)
	\begin{align}\label{eq.ARE}
	A^T P^\star\,+\, P^\star A \,+\, Q \,-\, P^\star B R^{-1}B^TP^\star \;=\; 0. 
	\end{align}
\end{subequations}

When the model is known, the LQR problem and the corresponding ARE can be solved efficiently via a variety of techniques~\cite{kle68,bitlauwil12,perger94,balvan03}. However, these methods are not directly applicable in the model-free setting, i.e., when the matrices $A$ and $B$ are unknown. Exploiting the linearity of the optimal controller, we can alternatively formulate the LQR problem as a direct search for the optimal linear feedback gain, namely
\begin{subequations}
	\label{eq.optProbK}
	\begin{align}\label{eq.optProbK1}
	\underset{K}{\minimize} ~ f(K)
	\end{align}
	where
	\begin{align}\label{eq.f}
	f(K) \, \DefinedAs \, 
	\left\{
	\ba{ll}
	\trace \left( (Q + K^TR\,K) X(K) \right), 
	& K\in\cSK
	\\[0.cm]
	\infty,
	&\text{otherwise.}
	\ea
	\right.
	\end{align}
	Here, the function $f(K)$ determines the LQR cost in~\eqref{eq.lqrEXform} associated with the linear state-feedback law
	$
	u=-Kx
	$,
	\begin{align}
	\cSK \;\DefinedAs\; \{ K\in\bbR^{m\times n} \, | \, A\,-\,BK\ \text{is Hurwitz}\}
	\end{align}
\end{subequations}
is the set of stabilizing feedback gains and, for any $K \in \cSK$,
\begin{subequations}
	\begin{align}
	\label{eq.X}
	X(K)
	\;\DefinedAs\;
		\int_{0}^{\infty}\EX\left[x(t)x^T(t)\right]
	\;=\; 
	\int_{0}^{\infty} \mre^{(A\,-\,BK) t}\,\Omega\, \mre^{(A\,-\,BK)^T t}\, \mrd t
	\end{align}
	is the unique solution to the Lyapunov equation
	\begin{align}
	\label{eq.lyapKX}
	(A \,-\, B K) X \;+\; X (A \,-\, B K)^T \, + \; \Omega  
	\; = \;
	0
	\end{align}
\end{subequations}
and $\Omega \DefinedAs \EX \, [ x(0)x^T(0) ]$. To ensure $f(K)=\infty$ for $K\notin\cSK$, we assume $\Omega\succ 0$. This assumption also guarantees  $K\in\cSK$ if and only if the solution $X$ to~\eqref{eq.lyapKX} is positive definite.

In problem~\eqref{eq.optProbK}, the matrix $K$ is the optimization variable, and ($A$, $B$, $Q\succ0$, $R\succ0$, $\Omega\succ0$) are the  problem parameters. This alternative formulation of the LQR problem has been studied for both continuous-time~\cite{andmoo90} and discrete-time systems~\cite{fazgekakmes18,bumesfazmes19} and it serves as a building block for several important control problems including optimal static-output feedback design~\cite{levath70}, optimal design of sparse feedback gain matrices~\cite{linfarjovTAC11al,farlinjovACC11,linfarjovTAC13admm,jovdhiEJC16}, and optimal sensor/actuator selection~\cite{polkhlshc13,dhijovluoCDC14,zarmohdhigeojovTAC20}.

For all stabilizing feedback gains $K\in\cSK$, the gradient of the objective function is determined by~\cite{levath70,linfarjovTAC11al}
\begin{align} 
\label{eq.nablaK}
\nabla f(K) \;=\; 2 (R \, K \, - \, B^T P(K) )X(K).
\end{align}
Here, $X(K)$ is given by~\eqref{eq.X} and 
	\begin{subequations}
	\be
	P(K)
	= 
	\int_{0}^{\infty} \mre^{(A \, - \, BK)^T t} \, (Q + K^TR\,K ) \, \mre^{(A \, - \, BK) t} \, \mrd t
	\label{eq.P}
	\ee
is the unique positive definite solution of 
	\be
	(A \,-\, B K)^T P \,+\, P (A \,-\, B K) 
	\,=\,
	-Q \,-\,K^TR\,K.   
	\label{eq.lyapKP}
	\ee
	\end{subequations}
	To simplify our presentation, for any $K\in\R^{m\times n}$, we define the closed-loop Lyapunov operator $\cA_K$: $\bbS^n \rightarrow \bbS^n$ as
\begin{subequations}\label{eq.LyapOper}
	\begin{align}\label{eq.LyapOperX}
	\cA_K(X) 
	&\;\DefinedAs\;
	(A\,-\,BK) X 
	\,+\,
	X(A\,-\,BK)^T.
	\end{align}
For $K\in\cSK$, both $\cA_K$ and its adjoint
	\begin{align}\label{eq.LyapOperP}
	\cA_K^*(P) 
	&\;=\;
	(A\,-\,BK)^TP 
	\,+\,
	P(A\,-\,BK)
	\end{align}
\end{subequations}
are invertible and $X(K)$, $P(K)$ are determined by
\[
X(K)
\, = \,
-\cA_K^{-1}(\Omega),
~~~
P(K)
\, = \,
-(\cA_K^{*})^{-1} (Q \, + \, K^TRK).
\]

In this paper, we first examine the global stability properties of the gradient-flow dynamics
\begin{align}
\label{eq.GDK-flow}
\tag{GF}
\dot{K}
\;=\; 
-\nabla f(K), 
~~
K(0) \, \in \, \cSK
\end{align}
associated with problem~\eqref{eq.optProbK} and its discretized variant, 
\begin{align}
\label{eq.GDK}
\tag{GD}
K^{k+1} 
\;\DefinedAs\; 
K^k \,-\, \alpha\,\nabla f(K^k), ~~ K^0 \,\in\,\cSK
\end{align}
where $\alpha>0$ is the stepsize. Next, we use this analysis as a building block to study the convergence of a search method based on random sampling~\cite{manguyrec18,rec19} for solving problem~\eqref{eq.optProbK}. As described in Algorithm~\ref{alg.GE}, at each iteration we form an empirical approximation $\widebar{\nabla} f (K)$ to the gradient of the objective function via simulation of system~\eqref{eq.linsys1} for randomly perturbed feedback gains $K\pm U_i$, $i=1,\ldots,N$, and update $K$ via,
\begin{align}
\label{eq.RS}
\tag{RS}
K^{k+1} 
\;\DefinedAs\; 
K^k \,-\, \alpha\,\widebar{\nabla} f(K^k), \quad K^0 \,\in\,\cSK.
\end{align} 
We note that the gradient estimation scheme in Algorithm~\ref{alg.GE} does not require knowledge of system matrices $A$ and $B$ in~\eqref{eq.linsys1} but only access to a simulation engine. 

\begin{algorithm}
	\caption{Two-point gradient estimation}
	\label{alg.GE}
	\begin{algorithmic}
		\REQUIRE Feedback gain $K\in\R^{m\times n}$, state and control weight matrices $Q$ and $R$, distribution $\cD$, smoothing constant $r$, simulation time $\tau$, number of random samples $N$.
		\FOR{$i \, = \, 1, \ldots, N$}
		\STATE -- Define perturbed feedback gains $K_{i,1} \DefinedAs K+ r U_i$ and $K_{i,2} \DefinedAs K- rU_i$, where $\vect({U_i})$ is a random vector uniformly  distributed on the sphere $\sqrt{mn} \, S^{mn-1}$.
		\STATE -- Sample an initial condition $x_{i}$ from distribution $\cD$.
		\STATE -- For $j \in \{ 1,2 \}$, simulate system~\eqref{eq.linsys1}  up to time $\tau$ with the feedback gain $K_{i,j}$ and initial condition $x_{i}$ to form
		\begin{center}
			$\ds \hat{f}_{i,j} \, = \, \ds{ \int_{0}^{\tau} (x^T (t) Q x(t) \, + \, u^T(t) R u (t)) \, \mrd t}$.
		\end{center}	
		\ENDFOR
		\ENSURE The gradient estimate 
		\begin{center}
			$\ds \widebar{\nabla} f(K) \, = \, \dfrac{1}{2r N} \, \sum_{i \, = \, 1}^{N} \left(\hat{f}_{i,1} \, - \, \hat{f}_{i,2}\right) U_i$.
		\end{center}
	\end{algorithmic}
\end{algorithm}

	\vspace*{-4ex}
\section{Main results}
\label{sec.MainResult}
Optimization problem~\eqref{eq.optProbK} is not convex~\cite{ack80}; see Appendix~\ref{app.non-convex-SK} for an example. The function $f(K)$, however, has two important properties: {\em uniqueness of the critical points and the compactness of sublevel sets\/}~\cite{toi85,rausac97}. Based on these, the LQR objective error $f(K)-f(K^\star)$ can be used as a maximal Lyapunov function (see~\cite{vanvid85} for a definition) to prove asymptotic stability of gradient-flow dynamics~\eqref{eq.GDK-flow} over the set of stabilizing feedback gains $\cSK$. However, this approach does not provide any guarantee on the rate of convergence and additional analysis is necessary to establish exponential stability; see Section~\ref{sec.non-convex-char} for details.  

\vspace*{-2ex}
\subsection{Known model}
\label{sec.known}

We first summarize our results for the case when the model is known. In spite of the nonconvex optimization landscape, we establish the exponential stability of gradient-flow dynamics~\eqref{eq.GDK-flow} for any stabilizing initial feedback gain $K(0)$. This result also provides an explicit bound on the rate of convergence to the LQR solution~$K^\star$.

\begin{mythm}
	\label{thm.main-cts}
	For any initial stabilizing feedback gain $K(0)\in\cSK$, the solution $K(t)$ to gradient-flow dynamics~\eqref{eq.GDK-flow} satisfies
	\begin{align*}
		f(K(t)) \,-\, f(K^\star)
		&\; \le \;  
		\mre^{-\rho\, t} \left(f(K(0)) \,-\, f(K^\star)\right)
		\\[-0.1cm]
		\norm{K(t) \,-\, K^\star}^2_F
		&\; \le \; 
		b \, \mre^{-\rho\, t} \, \norm{K(0) \,-\, K^\star}_F^2 
	\end{align*}
	where the convergence rate $\rho$ and constant $b$ depend on $K(0)$ and the parameters of the LQR problem~\eqref{eq.optProbK}. 
\end{mythm}


The proof of Theorem~\ref{thm.main-cts} along with explicit expressions for the convergence rate $\rho$ and constant $b$ are provided in Section~\ref{subsec.flowAnalysis}. Moreover, for a sufficiently small stepsize $\alpha$, we show that gradient descent method~\eqref{eq.GDK} also converges over $\cSK$ at a linear rate.

\begin{mythm}
	\label{thm.main-dsc}
	For any initial stabilizing feedback gain $K^0\in\cSK$, the iterates of gradient descent~\eqref{eq.GDK} satisfy
	\begin{align*}
	f(K^k) \,-\, f(K^\star)
	&\; \le \;
	\gamma^k \left(f(K^0) \,-\, f(K^\star)\right) 
		\\[-0.1cm]
	\norm{K^k \,-\, K^\star}^2_F
	&\; \le \;
	b \, \gamma^k \,\norm{K^0 \,-\, K^\star}_F^2 
	\end{align*}
	where the rate of convergence $\gamma$, stepsize $\alpha$, and constant $b$ depend on $K^0$ and the parameters of the LQR problem~\eqref{eq.optProbK}. 
\end{mythm}

\vspace*{-2ex}
\subsection{Unknown model}

We now turn our attention to the model-free setting. We use Theorem~\ref{thm.main-dsc} to carry out the convergence analysis of the random search method~\eqref{eq.RS} under the following assumption on the distribution of initial condition.
\begin{myass}\label{ass.initialCond}
	Let the distribution $\mathcal{D}$ of the initial conditions have i.i.d.\ zero-mean unit-variance entries with bounded sub-Gaussian norm, i.e.,~for a random vector $v\in\R^n$ distributed according to $\mathcal{D}$,
	$
	\mathbb{E}[v_i]=0$
	and
	$
	\norm{v_i}_{\psi_2}
	\le
	\kappa
	$,
	for some constant $\kappa$ and $i=1,\ldots,n$; see Appendix~\ref{app.probabilistic} for the definition of  $\norm{\cdot}_{\psi_2}$.
\end{myass}
	
Our main convergence result holds under Assumption~\ref{ass.initialCond}. Specifically, for a desired accuracy level $\epsilon>0$, in Theorem~\ref{thm.rndSearchInformal} we establish that iterates of~\eqref{eq.RS} with constant stepsize (that does not depend on $\epsilon$) reach accuracy level $\epsilon$ at a linear rate (i.e., in at most $O(\log \, (1/\epsilon))$ iterations) with high probability. Furthermore, the total number of function evaluations and the simulation time required to achieve an accuracy level $\epsilon$ are proportional to $\log \, (1/\epsilon)$. This significantly improves the existing results for discrete-time LQR~\cite{fazgekakmes18,malpanbhakhabarwai19} that require $O (1/\epsilon)$ function evaluations and $\mathrm{poly} (1/\epsilon)$ simulation time. 

\begin{mythm}[Informal]\label{thm.rndSearchInformal}
	 Let the initial condition $x_0\sim\cD$ of system~\eqref{eq.linsys1} obey Assumption~\ref{ass.initialCond}. Also let the simulation time $\tau$ and the number of samples $N$ in Algorithm~\ref{alg.GE} satisfy
	\begin{align*}
	\tau 
	\, \ge \, 
	\theta_1 \log \, (1/\epsilon) 
	~\text{and}~
	N
	\, \ge \,
	c \left(1 \, + \, \beta^4 \kappa^4\, \theta_1 \log^6 \! n \right) n
	\end{align*}
	for some  $\beta>0$ and  desired accuracy $\epsilon>0$. Then, we can choose the smoothing parameter $r< \theta_3 \sqrt{\epsilon}$ in Algorithm~\ref{alg.GE} and the constant stepsize $\alpha$ such that the random search method~\eqref{eq.RS} that starts from any initial stabilizing feedback gain $K^0\in\cSK$ achieves
$
	f(K^k)-f(K^\star)
	\le
	\epsilon
$
in at most
	\[
	k
	\;\le\;\theta_4
	\log \left( (f(K^0) \, - \, f(K^\star))/\epsilon \right)
	\]
iterations with probability not smaller than $1-c'k(n^{-\beta} + N^{-\beta} + N \mre^{-\frac{n}{8}} + \mre ^{-c' N})$. Here, the positive scalars $c$ and $c'$ are absolute constants and $\theta_1,\ldots,\theta_4>0$ depend on $K^0$ and the parameters of the LQR problem~\eqref{eq.optProbK}.
\end{mythm}
	
	The formal version of Theorem~\ref{thm.rndSearchInformal} along with a discussion of parameters $\theta_i$ and stepsize $\alpha$ is presented in Section~\ref{sec.rndSearch}.

	\vspace*{-2ex}
\section{Convex reparameterization}
\label{sec.convex-char}	

The main challenge in establishing the exponential stability of~\eqref{eq.GDK-flow} arises from nonconvexity of problem~\eqref{eq.optProbK}. Herein, we use a standard change of variables to reparameterize~\eqref{eq.optProbK} into a convex problem, for which we can provide exponential stability guarantees for gradient-flow dynamics. We then connect the gradient flow on this convex reparameterization to its nonconvex counterpart and establish the exponential stability of~\eqref{eq.GDK-flow}.

\vspace*{-2ex}
\subsection{Change of variables}

The stability of  the closed-loop system with the feedback gain $K\in\cSK$ in problem~\eqref{eq.optProbK} is equivalent to  the positive definiteness of the matrix $X(K)$ given by~\eqref{eq.X}. This condition  allows for a standard change of variables $K = Y X^{-1}$, for some $Y\in\R^{m\times n}$, to reformulate the LQR design as a convex optimization problem~\cite{ferbalboyelg92,dulpag00}. In particular, for any $K\in\cSK$ and the corresponding matrix $X$, we have
\begin{align*}
f(K) 
\;=\;
h(X,Y)
\; \DefinedAs \; 
\trace\, ( Q\,X + Y^T R\,Y X^{-1} )
\end{align*}
where 
$
h(X,Y)
$
is a jointly convex function of $(X,Y)$ for $X\succ0$. In the new variables, Lyapunov equation~\eqref{eq.lyapKX} takes the affine form
\begin{subequations}
\be
\label{eq.lyapunovXY}
\cA(X) \; - \; \cB(Y) \; + \; \Omega
\; = \;
0
\ee
where $\cA$ and $\cB$ are the linear maps
\be
\cA(X) 
\, \DefinedAs \, 
A\,X \,+\, X A^T,
~~~
\cB(Y) 
\, \DefinedAs \, 
B\,Y \,+\, Y^T B^T.
\ee
For an invertible map $\cA$, we can express the matrix $X$ as an affine function of $Y$
\be
	\label{eq.XofY}
X(Y) 
\, = \,
\cA^{-1} (\cB(Y) \,-\, \Omega)
\ee
\end{subequations}
and bring the LQR problem into the convex form
	\label{eq.lqrY}
	\begin{align}
	\underset{Y}{\minimize} \quad h(Y)
	\end{align}
	where
	\begin{align*}
	h(Y)
	\;\DefinedAs\;
	\left\{
	\ba{ll}
	h(X(Y),Y),
	& 
	Y\,\in\,\cSY
	\\[0cm]
	\infty,
	& 
	\text{otherwise}
	\ea
	\right.
	\end{align*}
and $\cSY\DefinedAs \{Y\in\bbR^{m\times n}\, | \, X(Y)\succ 0\}$ is the set of matrices $Y$ that correspond to stabilizing feedback gains $K = Y X^{-1}$. The set $\cSY$ is open and convex because it is defined via a positive definite condition imposed on the affine map $X(Y)$ in~\eqref{eq.XofY}.  This positive definite condition in $\cSY$ is equivalent to  the closed-loop matrix $A - B \, Y (X(Y))^{-1}$ being Hurwitz. 
\begin{myrem}
	Although our presentation assumes invertibility of $\cA$, this assumption comes without loss of generality. As shown in Appendix~\ref{app.K0}, all results  carry over to noninvertible $\cA$ with an alternative change of variables $A = \hat{A} + B K^0$, $K =  \hat{K}+K^0$, and $\hat{K} = \hat{Y}X^{-1}$, for some $K^0\in\cSK$. 
\end{myrem}

\vspace*{-2ex}
\subsection{Smoothness and strong convexity of $h(Y)$}

Our convergence analysis of gradient methods for problem~\eqref{eq.lqrY} relies on the $L$-smoothness and $\mu$-strong convexity of the function $h(Y)$ over its sublevel sets
$
\cSY(a)
\DefinedAs 
\{ Y \in \cSY \, | \, h(Y) \le a \}.
$
These two properties were recently established in~\cite{zarmohdhigeojovTAC20} where it was shown that over any sublevel set $\cSY(a)$, the second-order term  $\inner{\tY}{\nabla^2 h(Y; \tY)}$ in the Taylor series expansion of $h(Y+\tY)$ around $Y\in\cSY(a)$  can be upper and lower bounded  by quadratic forms $L \norm{\tY}_F^2$ and $\mu \norm{\tY}_F^2$ for some positive scalars $L$ and $\mu$. While an explicit form for the smoothness parameter $L$ along with an existence proof for the strong convexity modulus $\mu$ were presented in~\cite{zarmohdhigeojovTAC20}, in Proposition~\ref{prop.LMu} we establish an explicit expression for $\mu$ in terms of $a$ and parameters of the LQR problem. This allows us to provide bounds on the convergence rate for gradient methods. 	
\begin{myprop}
	\label{prop.LMu}
	Over any non-empty sublevel set $\cSY(a)$,  the function $h(Y)$ is  $L$-smooth and $\mu$-strongly convex with
	\begin{subequations}
		\begin{align}
		L
		&\;=\;
		\dfrac{2 a \norm{R}_2}{\nu} \left( 1 \,+\, \dfrac{a \norm{\cA^{-1} \cB}_2}{ \sqrt{\nu \lambda_{\min}(R)}} \right)^2
		\label{eq.L}
		\\[0.15cm]
		\label{eq.muStrong}
		\mu 
		&\;=\; 
		\dfrac{2 \lambda_{\min}(R) \lambda_{\min}(Q)}{a \left(1\,+\,a^2 \eta\right)^2}
		\end{align}
		where the constants
		\begin{align}
		\label{eq.nuprime}
		\eta
		&\;\DefinedAs\;
		\dfrac{\norm{\cB}_2}{\lambda_{\min}(Q) \lambda_{\min}(\Omega) \sqrt{\nu\,\lambda_{\min}(R)}}
		\\[0.15cm]
		\label{eq.nuu}
		\nu
		&\;\DefinedAs\;
		\dfrac{\lambda_{\min}^2(\Omega)}{4} \left(\dfrac{\norm{A}_2 }{\sqrt{\lambda_{\min}(Q)}} \,+\, \dfrac{\norm{B}_2}		
		{\sqrt{\lambda_{\min}(R)}} \right)^{-2}
		\end{align}
	\end{subequations}
	only depend on the problem parameters.
\end{myprop}
\begin{proof}
	See Appendix~\ref{app.convex-char}.
\end{proof}

\vspace*{-2ex}
\subsection{Gradient methods over $\cSY$}


The LQR problem can be solved by minimizing the convex function $h(Y)$ whose gradient is given by~\cite[Appendix C]{zarmohdhigeojovTAC20}
\begin{subequations}
	\label{eq.nablaY-W}
	\begin{align}
	\label{eq.nablaY}
	\nabla h(Y)
	\; = \;
	2\, R\,Y ( X (Y) )^{-1}
	\, - \,
	2\, B^T W (Y)
	\end{align} 
	where $W (Y)$ is the solution to
	\begin{align}
	\label{eq.W}
	A^T \, W \, +\, W A 
	\; = \;
	( X (Y) )^{-1} \, Y^{T} R\, Y ( X (Y) )^{-1} \, - \, Q.
	\end{align}
\end{subequations}
Using the strong convexity and smoothness properties of $h(Y)$ established in Proposition~\ref{prop.LMu}, we next show that the unique minimizer $Y^\star$ of the function $h(Y)$ is the exponentially stable equilibrium point of the gradient-flow dynamics over $\cSY$,
\begin{align}\label{eq.GDY-flow}
\tag{GFY}
\dot{Y} \;=\; -\nabla h(Y), \quad Y(0) \, \in \, \cSY.
\end{align}
	\vspace*{-4ex}	
\begin{myprop}
	For any $Y(0) \in \cSY$, the gradient-flow dynamics~\eqref{eq.GDY-flow} are exponentially stable, i.e., 
	\[
	\norm{Y(t) \, - \, Y^\star}_F^2 
	\; \le \; 
	(L/\mu) \, \mre^{-2\,\mu\,t} \, \norm{Y(0) \, - \, Y^\star}_F^2
	\]
	where $\mu$ and $L$ are the strong convexity and smoothness parameters of the function $h(Y)$ over the sublevel set $\cSY(h(Y(0)))$.
	\label{prop.Yexp} 
\end{myprop}

\begin{proof}
	The derivative of the Lyapunov function candidate
	$
	V(Y)
	\DefinedAs 
	h(Y) - h(Y^\star)
	$
	along the flow in~\eqref{eq.GDY-flow} satisfies
	\begin{align}
	\label{eq.comparison}
	\dot{V}
	\; = \;
	\inner{\nabla h(Y)}{\dot{Y}}
	\;=\; 
	-\norm{\nabla h(Y)}_F^2
	\;\le\; 
	-2\,\mu V.
	\end{align}
	Inequality~\eqref{eq.comparison} is a consequence of the strong convexity of $h(Y)$ and it yields~\cite[Lemma~3.4]{kha96}
	\begin{align}\label{eq.expConvLyapV}
	V(Y(t))
	\;\le\; 
	\mre^{-2\,\mu\,t} \, V(Y(0)).
	\end{align}
	Thus, for any $Y(0) \in \cSY$, $h(Y(t))$ converges exponentially to $h(Y^\star)$. Moreover, since $h(Y)$ is $\mu$-strongly convex and $L$-smooth, $V(Y)$ can be upper and lower bounded by quadratic functions and the exponential stability of~\eqref{eq.GDY-flow} over $\cSY$ follows from Lyapunov theory~\cite[Theorem~4.10]{kha96}.
\end{proof} 
In Section~\ref{sec.non-convex-char}, we use the above result to prove exponential/linear convergence of gradient flow/descent for the nonconvex optimization problem~\eqref{eq.optProbK}. Before we proceed, we note that similar convergence guarantees can be established for the gradient descent method with a sufficiently small stepsize~$\alpha$,
\begin{align}
\label{eq.GDY}
\tag{GY}
Y^{k+1}
\; \DefinedAs \; 
Y^k \; - \; \alpha\,\nabla h(Y^k), \quad Y^0 \, \in \, \cSY
\end{align}
Since the function $h(Y)$ is $L$-smooth over the sublevel set $\cSY(h(Y^{0}))$, for any $\alpha\in [ 0,1/L ]$ the iterates $Y^k$ remain within $\cSY(h(Y^{0}))$. This property in conjunction with the $\mu$-strong convexity of $h(Y)$ imply that $Y^k$ converges to the optimal solution $Y^\star$ at a linear rate of $\gamma = 1-\alpha \mu$.

\vspace*{-2ex}
\section{Control design with a known model }
\label{sec.non-convex-char}

The asymptotic stability of~\eqref{eq.GDK-flow} is a consequence of the following properties of the LQR objective function~\cite{toi85,rausac97}:
\begin{enumerate}
	\item The function $f(K)$ is twice continuously differentiable over its open domain $\cSK$ and $f(K)\rightarrow\infty$ as $K\rightarrow\infty$ and/or $K\rightarrow\partial\cSK$. 
	
	\item The optimal solution $K^\star$ is the unique equilibrium point over ${\cal S}_K$, i.e., $\nabla f(K) = 0$ if and only if $K=K^\star$.
\end{enumerate}
In particular, the derivative of the maximal Lyapunov function candidate 
$
V(K) \DefinedAs f(K) - f(K^\star)
$
along the trajectories of~\eqref{eq.GDK-flow} satisfies
\begin{align*}
\dot{V} 
\;=\;
\inner{\nabla f(K)}{\dot{K}}
\;=\;
- \norm{\nabla f (K)}_F^2
\;\le\; 
0
\end{align*}
where the inequality is strict for all $K\neq K^\star$. Thus, Lyapunov theory~\cite{vanvid85} implies that, starting from any stabilizing initial condition $K(0)$, the trajectories of~\eqref{eq.GDK-flow} remain within the sublevel set $\cSK(f(K(0)))$ and asymptotically converge to $K^\star$. 

Similar arguments were employed for the convergence analysis of the Anderson-Moore algorithm for output-feedback synthesis~\cite{toi85}. While~\cite{toi85} shows global asymptotic stability, it does not provide any information on the rate of convergence. In this section, we first demonstrate exponential stability of~\eqref{eq.GDK-flow} and prove Theorem~\ref{thm.main-cts}. Then, we  establish linear convergence of the gradient descent method~\eqref{eq.GDK} and prove Theorem~\ref{thm.main-dsc}.

	\vspace*{-2ex}
\subsection{Gradient-flow dynamics: proof of Theorem~\ref{thm.main-cts}}\label{subsec.flowAnalysis}

We start our proof of Theorem~\ref{thm.main-cts} by relating the convex and nonconvex formulations of the LQR objective function. Specifically, in Lemma~\ref{lem.comparison}, we establish a relation between the gradients $\nabla f(K)$ and $\nabla h(Y)$ over the sublevel sets of the objective function
$
\cSK(a)
\DefinedAs
\{ K \in \cSK \, | \, f(K) \le a \}.
$

\begin{mylem}\label{lem.comparison}
	For any stabilizing feedback gain $K\in\cSK(a)$ and $Y \DefinedAs K X(K)$, we have
	\begin{subequations}
		\begin{align}
		\label{eq.c1Ineq}
		\norm{\nabla f(K)}_F 
		&
		\;\ge\;
		c \, \norm{\nabla h(Y)}_F
		\end{align}
		where $X(K)$ is given by~\eqref{eq.X}, the constant $c$ is determined by
		\begin{align}\label{eq.cComparison}
		c 
		&\; = \;
		\dfrac{\nu\sqrt{\nu\,\lambda_{\min}(R)}}{2\, a^2 \,\norm{\cA^{-1}}_2 \,\norm{B}_2 \,+\, a\,\sqrt{\nu\,\lambda_{\min}(R)}}
		\end{align}
	\end{subequations}
	and the scalar $\nu$ given by Eq.~\eqref{eq.nuu} depends on the problem parameters.
\end{mylem}
\begin{proof}
	See Appendix~\ref{app.non-convex-char}.
\end{proof}

Using Lemma~\ref{lem.comparison} and the exponential stability of gradient-flow dynamics~\eqref{eq.GDY-flow} over $\cSY$,  established in Proposition~\ref{prop.Yexp}, we next show that~\eqref{eq.GDK-flow} is also exponentially stable. In particular, for any stabilizing $K\in\cSK(a)$, the derivative of 
$
V(K)
\DefinedAs
f(K) - f(K^\star)
$
along the gradient flow in~\eqref{eq.GDK-flow}
satisfies
\begin{align}
\dot{V}
\; = \,
-\norm{\nabla f(K)}_F^2
\; \le \,
-c^2 \, \norm{\nabla h(Y)}_F^2
\; \le \, 
-2\,\mu\, c^2\,V
\label{eq.VdotVGD}
\end{align}
where $Y = K X(K)$ and the constants $c$ and $\mu$ are provided in Lemma~\ref{lem.comparison} and Proposition~\ref{prop.LMu}, respectively. The first inequality in~\eqref{eq.VdotVGD} follows from~\eqref{eq.c1Ineq} and the second follows from $f(K)=h(Y)$ combined with $\norm{\nabla h(Y)}_F^2\ge2\mu V$ (which in turn is a consequence of the strong convexity of $h(Y)$ established in Proposition~\ref{prop.LMu}).

Now, since the sublevel set $\cSK(a)$ is invariant with respect to~\eqref{eq.GDK-flow}, following~\cite[Lemma~3.4]{kha96},  inequality~\eqref{eq.VdotVGD} guarantees
that system~\eqref{eq.GDK-flow} converges exponentially in the objective value with rate
$
\rho = 2 \mu c^2.
$
This concludes the proof of part~(a) in Theorem~\ref{thm.main-cts}. In order to prove part~(b), we use the following lemma which connects the errors in the objective value and the optimization variable. 
\begin{mylem}\label{lem.errorRelation}
	For any stabilizing feedback gain $K$, the objective function $f(K)$ in problem~\eqref{eq.optProbK} satisfies
	\begin{align*}
	f(K)
	\,-\,
	f(K^\star)
	\; = \;
	\trace \left((K-K^\star)^T R \,(K-K^\star)\, X(K) \right)
	\end{align*}
	where $K^\star$ is the optimal solution and $X(K)$ is given by~\eqref{eq.X}.
\end{mylem}
\begin{proof}
	See Appendix~\ref{app.non-convex-char}.
\end{proof}

From Lemma~\ref{lem.errorRelation} and part~(a) of Theorem~\ref{thm.main-cts}, we have
\beq
\ba{rcl}
\norm{K(t) \,-\, K^\star}^2_F
& \!\!\! \le \!\!\! &
\dfrac{f(K(t)) \,-\, f(K^\star)}{\lambda_{\min}(R)\,\lambda_{\min}(X(K(t)))}
\;\le\;
\mre^{-\rho\, t} 
\,
\dfrac{f(K(0)) \,-\, f(K^\star)}{\lambda_{\min}(R)\,\lambda_{\min}(X(K(t)))}
\;\le\;
b' \, \mre^{-\rho\, t} \, \norm{K(0) \,-\, K^\star}^2_F 
\ea
\non
\eeq
where 
$
b'
\DefinedAs
{\norm{R}_2\norm{X(K(0))}_2}/( \lambda_{\min}(R) \lambda_{\min}(X(K(t))) ).
$
Here, the first and third inequalities follow form basic properties of the matrix trace combined with  Lemma~\ref{lem.errorRelation} applied with $K=K(t)$ and $K=K(0)$, respectively. The second inequality follows from part~(a) of Theorem~\ref{thm.main-cts}.

Finally, to upper bound parameter $b'$, we use Lemma~\ref{lem.bounds} presented in Appendix~\ref{app.bounds} that provides the lower and upper bounds
$
{\nu}/{a}
\le
\lambda_{\min}(X(K))
$
and 
$\norm{X(K)}_2 \le a/\lambda_{\min}(Q)$ 
on the matrix $X(K)$ for any $K\in\cSK(a)$,
where the constant $\nu$ is given by~\eqref{eq.nuu}. Using these bounds and the invariance of $\cSK(a)$ with respect to~\eqref{eq.GDK-flow}, we obtain
\begin{align}\label{eq.bDef}
b'
\;\le\;
b
\;\DefinedAs\;
\dfrac{a^2\,\norm{R}_2}{\nu\,\lambda_{\min}(R)\,\lambda_{\min}(Q)}
\end{align}
which completes the proof of part~(b).
\begin{myrem}[Gradient domination]	Expression~\eqref{eq.VdotVGD} implies that the objective function $f(K)$ over any given sublevel set $\cSK(a)$ satisfies the Polyak-{\L}ojasiewicz (PL) condition~\cite{karnutsch16}
	\begin{align}\label{eq.GradDom}
	\norm{\nabla f(K)}_F^2 
	\;\ge\; 
	{2\,\mu_f} \left( f(K)\,-\,f(K^\star)\right)
	\end{align}
with parameter $ \mu_f\DefinedAs\mu\,c^2 $, where $\mu$ and $c$ are  functions of $a$ that are given by~\eqref{eq.muStrong} and~\eqref{eq.cComparison}, respectively. This condition is also known as gradient dominance and it was recently used to show convergence of gradient-based methods for {\em discrete-time\/} LQR problem~\cite{fazgekakmes18}.
\end{myrem}

\vspace*{-2ex}
\subsection{Geometric interpretation}

The solution $Y(t)$ to gradient-flow dynamics~\eqref{eq.GDY-flow} over the set $\cSY$ induces the trajectory 
\begin{align}\label{eq.Kind}
K_{\mathrm{ind}}(t) \; \DefinedAs \; Y(t) ( X (Y (t)) )^{-1}
\end{align} 
over the set of stabilizing feedback gains $\cSK$, where  the affine function $X(Y)$ is given by~\eqref{eq.XofY}. The induced trajectory $K_{\mathrm{ind}}(t)$ can be viewed as the solution to the differential equation
\begin{subequations}
	\label{eq.Kind-eq}
	\begin{align}
	\dot{K}
	\;=\;
	g(K)
	\end{align}
	where $g$: $\cSK\rightarrow\R^{m \times n}$ is given by
	\be
	\label{eq.gK}
	g(K)
	\, \DefinedAs \,
	\left(
	K \cA^{-1}(\cB(\nabla h(Y (K)))) 	- \nabla h(Y (K)) 
	\right)
	(X (K) )^{-1}.
	\ee
\end{subequations}
Here, the matrix $X = X(K)$ is given by~\eqref{eq.X} and $Y (K) = K X (K)$. System~\eqref{eq.Kind-eq}  is obtained by differentiating both sides of Eq.~\eqref{eq.Kind} with respect to time $t$ and applying the chain rule. Figure~\ref{fig.landscape} illustrates an induced  trajectory $K_{\mathrm{ind}}(t)$ and a trajectory $K(t)$ resulting from gradient-flow dynamics~\eqref{eq.GDK-flow} that starts from the same initial condition.

Moreover, using the definition of $h(Y)$, we have
\begin{align}\label{eq.geoHelper}
h(Y(t)) \; = \; f(K_{\mathrm{ind}}(t)).
\end{align}
Thus, the exponential decay of $h(Y(t))$ established in Proposition~\ref{prop.Yexp} implies that $f$ decays exponentially along the vector field $g$, i.e., for $K_{\mathrm{ind}}(0) \neq K^\star$, we have
\begin{align*}
\dfrac{f(K_{\mathrm{ind}}(t)) \,-\, f(K^\star)}{f(K_{\mathrm{ind}}(0)) \,-\, f(K^\star)} 
\;=\;
\dfrac{h(Y(t)) \,-\, h(Y^\star)}{h(Y(0)) \,-\, h(Y^\star)} 
\;\le\;
\mre^{-2\,\mu\, t}.
\end{align*}
This inequality follows from inequality~\eqref{eq.expConvLyapV}, where $\mu$ denotes the strong-convexity modulus of the function $h(Y)$ over the sublevel set $\cSY(h(Y(0)))$; see Proposition~\ref{prop.LMu}. Herein, we provide a geometric interpretation of the exponential decay of $f$ under the trajectories of~\eqref{eq.GDK-flow} that is based on the relation between the vector fields $g$ and $-\nabla f$. 

Differentiating both sides of Eq.~\eqref{eq.geoHelper} with respect to $t$ yields
\begin{align}\label{eq.geoint1}
\norm{\nabla h(Y)}^2 = \inner{-\nabla f(K)}{g(K)}.
\end{align}
Thus, for each $K \in \cSK$, the inner product between the vector fields $-\nabla f (K)$ and $g (K)$ is nonnegative. However, this is not sufficient to ensure exponential decay of $f$ along~\eqref{eq.GDK-flow}. To address this challenge, our proof utilizes inequality~\eqref{eq.c1Ineq} in Lemma~\ref{lem.comparison}. Based on~\eqref{eq.geoint1},~\eqref{eq.c1Ineq} can be equivalently restated as
\begin{align*}
\dfrac{\norm{-\nabla f(K)}_F}{\norm{\mathrm{\Pi}_{-\nabla f(K)}( g(K) )}_F}
\;=\;
\dfrac{\norm{\nabla f(K)}_F^2}{\inner{-\nabla f(K)}{g(K)} }
\; \ge \;
c^2
\end{align*}
where $\Pi_b(a)$ denotes the projection of $a$ onto $b$. Thus, Lemma~\ref{lem.comparison} ensures that the ratio between the norm of the vector field $-\nabla f (K)$ associated with gradient-flow dynamics~\eqref{eq.GDK-flow} and the norm of the projection of $g (K)$ onto $- \nabla f (K)$ is uniformly lower bounded by a positive constant. This lower bound is the key geometric feature that allows us to deduce exponential decay of $f$ along the vector field $-\nabla f$ from the exponential decay of the vector field $g$.

\begin{figure}
	\centering
	\begin{tabular}{c}			
		\includegraphics[width=.25\textwidth]{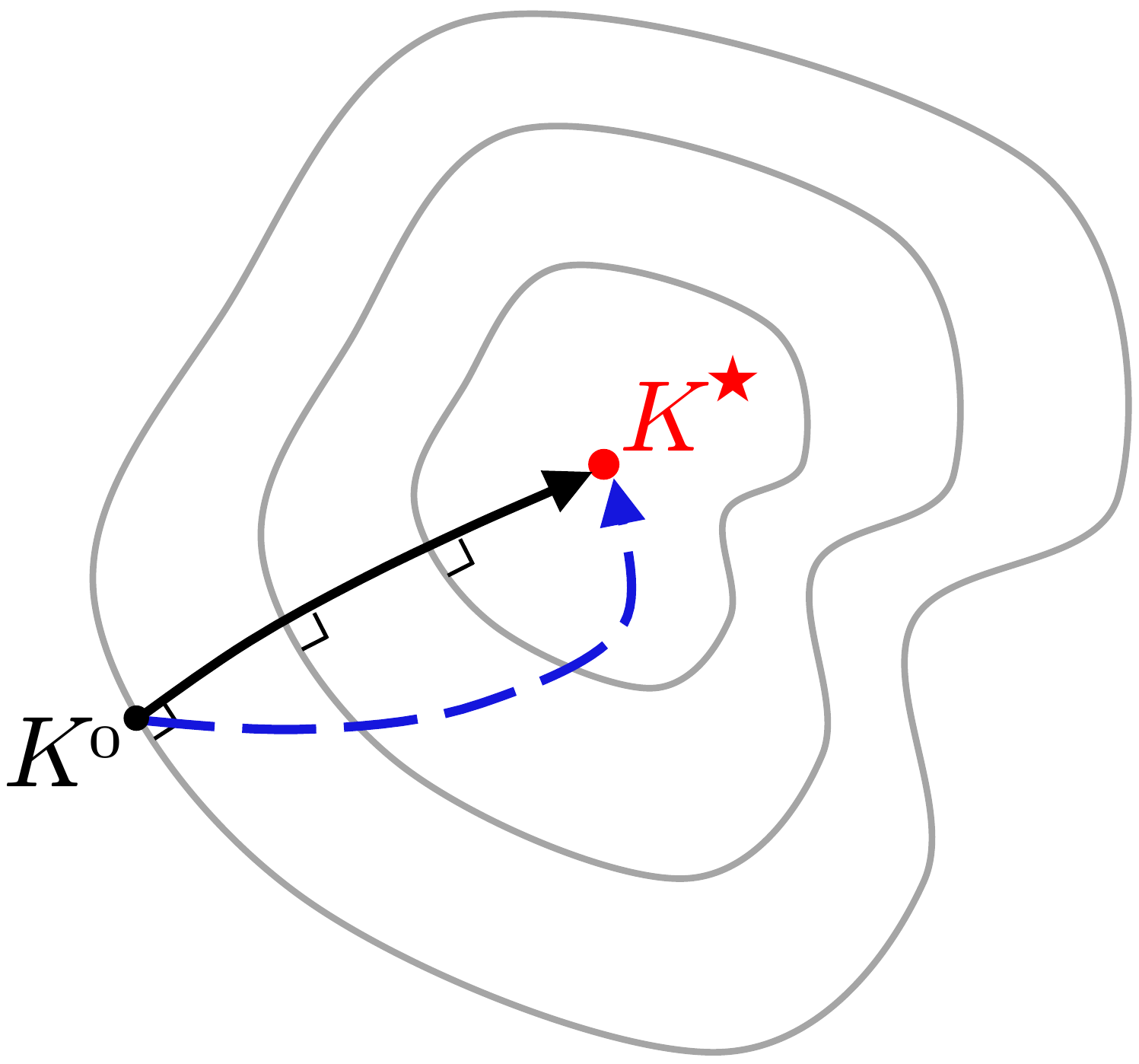}
	\end{tabular}
		\caption{Trajectories $K(t)$ of~\eqref{eq.GDK-flow} (solid black) and $K_{\mathrm{ind}}(t)$ resulting from Eq.~\eqref{eq.Kind} (dashed blue) along with the level sets of the function~$f(K)$.}
	\label{fig.landscape}
\end{figure}

\vspace*{-2ex}
\subsection{Gradient descent: proof of Theorem~\ref{thm.main-dsc}}\label{subsec.iterAnalysis}

Given the exponential stability of gradient-flow dynamics~\eqref{eq.GDK-flow} established in Theorem~\ref{thm.main-cts}, the convergence analysis of gradient descent~\eqref{eq.GDK} amounts to finding a suitable stepsize $\alpha$. Lemma~\ref{lem.Lf} provides a Lipschitz continuity parameter for $\nabla f(K)$, which facilitates finding such a stepsize. 

\begin{mylem}\label{lem.Lf}
	Over any non-empty sublevel set $\cSK(a)$, the  gradient $\nabla f(K)$ is Lipschitz continuous with parameter 
	\begin{align*}
	L_{f}
	\; \DefinedAs \;
	\dfrac{2 a\norm{R}_2}{\lambda_{\min}(Q)}
	\, + \,
	\dfrac{8 a^3 \norm{B}_2}{\lambda_{\min}^2(Q)\lambda_{\min}(\Omega)} 
	\big(
	\dfrac{\norm{B}_2}{\lambda_{\min}(\Omega)} 
	\, + \,
	\dfrac{\norm{R}_2}{\sqrt{\nu\lambda_{\min}(R)}}\!
	\big)
	\end{align*}
	where  $\nu$ given by~\eqref{eq.nuu} depends on the problem parameters.
\end{mylem}
\begin{proof}
	See Appendix~\ref{app.non-convex-char}.
\end{proof}

Let  $K_\alpha \DefinedAs K-\alpha\,\nabla f(K)$, $\alpha\ge0$ parameterize the half-line starting from $K\in\cSK(a)$ with $K\neq K^\star$  along $-\nabla f(K)$ and let us define the scalar $\beta_m\DefinedAs\max \beta$ such that $K_\alpha\in\cSK(a)$, for all $\alpha\in[0,\beta]$. The existence of $\beta_m$ follows from the compactness of $\cSK(a)$~\cite{toi85}.  We next show that  $\beta_m\ge 2/L_f$.  

For the sake of contradiction, suppose $\beta_m<2/L_f$. From the continuity of $f(K_\alpha)$ with respect to $\alpha$, it follows that $f(K_{\beta_s})=a$. Moreover, since $-\nabla f(K)$ is a descent direction of the function $f(K)$, we have $\beta_m>0$.  Thus, for $\alpha\in(0,\beta_m]$,
\begin{align*}
f(K_\alpha) \, - \, f(K)
\; \le \;
- \dfrac{\alpha (2 \, - \, L_f \alpha )}{2} \, \norm{\nabla f(K)}_F^2
\; < \; 0.
\end{align*} 
Here, the first inequality follows from the $L_f$-smoothness of $f(K)$ over $\cSK(a)$ (Descent Lemma~\cite[Eq.~(9.17)]{boyvan04}) and the second inequality follows from $\nabla f(K)\neq 0$ in conjunction with $\beta_m\in(0,2/L_f)$. This implies $f(K_{\beta_m}) < f(K)\le a$, which contradicts $f(K_{\beta_m})=a$. Thus, $\beta_m\ge2/L_f$. 

We can now use induction on $k$ to show that, for any stabilizing initial condition $K^0\in\cSK(a)$, the iterates of~\eqref{eq.GDK} with $\alpha\in[0,2/L_f]$ remain in $\cSK(a)$ and satisfy
\begin{align}\label{eq.gdConv1}
f(K^{k+1}) \,-\, f(K^k) 
\; \le \;
- \dfrac{\alpha (2 \, - \, L_f \alpha )}{2} \, \norm{\nabla f(K^k)}_F^2.
\end{align} 
Inequality~\eqref{eq.gdConv1} in conjunctions with the PL condition~\eqref{eq.GradDom} evaluated at $K^k$
guarantee linear convergence for gradient descent~\eqref{eq.GDK} with the rate
$
\gamma 
\le 
1- \alpha\, 
\mu_f
$
for all $\alpha\in (0,1/L_f]$, where $\mu_f$ is the PL parameter of the function $f(K)$. This completes the proof of part~(a) of Theorem~\ref{thm.main-dsc}. 

Using part~(a) and Lemma~\ref{lem.errorRelation}, we can make a similar argument to what we used for the proof of Theorem~\ref{thm.main-cts} to establish  part~(b) with constant $b$ in~\eqref{eq.bDef}. We omit the details for brevity.

	\begin{myrem}
		Using our results, it is straightforward to show linear convergence of
			$
			K^{k+1} 
			=
			K^k - \alpha H_1^k \nabla f(K^k) H_2^k
			$
			with	 
			$
			K^0 \in \cSK
			$
			and small enough stepsize, where $H_1^k$ and $H_2^k$ are uniformly upper and lower bounded positive definite matrices. In particular, the Kleinman iteration~\cite{kle68} is recovered for $\alpha = 0.5$, $H_1^k=R^{-1}$, and $H_2^k=(X(K^k))^{-1}$. Similarly, convergence of gradient descent may be improved by choosing $H_1^k=I$ and $H_2^k=(X(K^k))^{-1}$. In this case, the corresponding update direction provides the continuous-time variant of the so-called {\em natural gradient\/} for discrete-time systems~\cite{ama98}.
	\end{myrem}

	\vspace*{-2ex}
\section{Bias and correlation in gradient estimation}
\label{sec.BiasVariance}

In the model-free setting, we do not have access to the gradient $\nabla f(K)$ and the random search method~\eqref{eq.RS} relies on the gradient estimate $\widebar{\nabla} f(K)$ resulting from Algorithm~\ref{alg.GE}. According to~\cite{fazgekakmes18}, achieving $\| \widebar{\nabla} f(K) - \nabla f(K) \|_F \leq \epsilon$ may take $N=\Omega(1/\epsilon^4)$ samples using one-point gradient estimates. Our computational experiments (not included in this paper) also suggest that to achieve $\| \widebar{\nabla} f(K) - \nabla f(K) \|_F \leq \epsilon$, $N$ must scale as $\mathrm{poly} \, (1/\epsilon)$ even when a two-point gradient estimate is used. To avoid this poor sample complexity, in our proof we take an alternative route and give up on the objective of controlling the gradient estimation error. By exploiting the problem structure, we show that with a linear number of samples $N=\tilde{O}(n)$, where $n$ is the number of states, the estimate $\widebar{\nabla} f(K)$ concentrates with {\em high probability\/} when projected to the direction of $\nabla f(K)$. 

Our proof strategy allows us to significantly improve upon the existing literature both in terms of the required function evaluations and simulation time. Specifically, using the random search method~\eqref{eq.RS}, the total number of function evaluations required in our results to achieve an accuracy level $\epsilon$ is proportional to $\log \, (1/\epsilon)$ compared to at least $( 1/\epsilon^4 ) \log \, (1/\epsilon)$ in~\cite{fazgekakmes18} and $1/\epsilon$ in~\cite{malpanbhakhabarwai19}. Similarly, the simulation time that we require to achieve an accuracy level $\epsilon$ is proportional to~$\log \, (1/\epsilon)$; this is in contrast to $\mathrm{poly} \, (1/\epsilon)$ simulation times in~\cite{fazgekakmes18} and infinite simulation time in~\cite{malpanbhakhabarwai19}. 

Algorithm~\ref{alg.GE} produces a biased estimate  $\widebar{\nabla} f(K)$ of the gradient $\nabla f(K)$. Herein, we first introduce an unbiased estimate $\widehat{\nabla} f(K)$ of $\nabla f(K)$ and establish that the distance $\norm{\widehat{\nabla } f(K) - \widebar{\nabla } f(K)}_F$ can be readily controlled by choosing a large simulation time $\tau$ and an appropriate smoothing parameter $r$ in Algorithm~\ref{alg.GE}; we call this distance the estimation bias. Next, we show that with $N=\tilde{O}(n)$ samples, the unbiased estimate $\widehat{\nabla } f(K)$ becomes highly correlated with $\nabla f(K)$. We exploit this fact in our convergence analysis.

 	\vspace*{-2ex}
\subsection{Bias in gradient estimation due to finite simulation time}
\label{sec.Bias}

We first introduce an unbiased estimate of the gradient that is used to quantify the bias.
 For any $\tau\ge0$ and $x_0\in\R^n$, let
\[
f_{x_0,\tau}(K) 
\; \DefinedAs \;
\int_{0}^{\tau} \left(x^T (t) Q x(t) \, + \, u^T(t) R u (t)\right) \mrd t
\]
denote the $\tau$-truncated version of the LQR objective function associated with system~\eqref{eq.linsys1} with the initial condition $x(0)=x_0$ and  feedback law $u=-Kx$
for all $K\in\R^{m\times n}$.
Note that for any $K\in \cSK$ and $x(0)=x_0\in\R^n$, the infinite-horizon cost
\begin{subequations}
	\begin{align}\label{eq.f_x}
	f_{x_0}(K) 
	\;\DefinedAs\;
	f_{x_0,\infty}(K)
	\end{align}
	exists and it satisfies
	$
	f(K)
	=
	\EX_{x_0}[f_{x_0}(K)].
	$
	Furthermore, the gradient of $f_{x_0}(K)$ is given by~(cf.~\eqref{eq.nablaK})
	\begin{align} 
	\label{eq.nablaKv}
	\nabla f_{x_0}(K) 
	\;=\;
	2 \, (R K \, - \, B^T P(K) )X_{x_0}(K)
	\end{align}
\end{subequations}
where $X_{x_0}(K) = -\cA_K^{-1}(x_0x_0^T)$ is determined by the closed-loop Lyapunov operator in~\eqref{eq.LyapOper}
and $P(K)=-(\cA_K^*)^{-1}(Q+K^TRK)$. Note that the gradients $\nabla f(K)$ and $\nabla f_{x_0}(K)$ are linear in  $X(K)=-\cA_K^{-1}(\Omega)$  and $X_{x_0}(K)$, respectively. Thus, for any zero-mean random initial condition $x(0)=x_0$ with covariance $\EX[x_0x_0^T] = \Omega$, the linearity of the closed-loop Lyapunov operator $\cA_K$ implies
\begin{align*}
\EX_{x_0} [X_{x_0}(K)]
\;=\;
X(K),
~
\EX_{x_0} [\nabla f_{x_0}(K)]
\;=\;
\nabla f(K).
\end{align*}
Let us define the following three estimates of the gradient
\be
\ba{rcl}
\widebar{\nabla} f(K) 
& \!\!\! \DefinedAs \!\!\! &
\dfrac{1}{2r N} \, \ds{\sum_{i \, = \, 1}^{N}} \left(f_{x_i,\tau}(K+rU_i) - f_{x_i,\tau}(K-rU_i)\right) U_i
\\[0.35cm]
\widetilde{\nabla} f(K) 
& \!\!\! \DefinedAs \!\!\! &
\dfrac{1}{2r N} \, \ds{\sum_{i \, = \, 1}^{N}} \left(f_{x_i}(K+rU_i) - f_{x_i}(K-rU_i)\right) U_i
\\[0.35cm]
\widehat{\nabla} f(K)
& \!\!\! \DefinedAs \!\!\! &
\dfrac{1}{N} \, \ds{\sum_{i \, = \, 1}^N} \inner{\nabla f_{x_i}(K)}{U_i} U_i	
\ea
\label{eq.megaGrad}
\ee	
where $U_i\in\R^{m\times n}$ are i.i.d.\ random matrices with $\vect(U_i)$ uniformly distributed on the sphere $\sqrt{mn} \, S^{mn-1}$ and $x_i\in\R^n$ are i.i.d.\ initial conditions sampled from distribution $\cD$. Here, $\widetilde{\nabla} f(K)$ is the infinite-horizon version of the output $\widebar{\nabla} f(K)$ of Algorithm~\ref{alg.GE} and $\widehat{\nabla} f(K)$ provides an unbiased estimate of $\nabla f(K)$. To see this, note that by the independence of $U_i$ and $x_i$ we~have
\begin{multline*}
\EX_{x_i,U_i} \! \left[\vect(\widehat{\nabla} f(K)) \right]
\;=\;
\EX_{U_1} \! \left[\inner{\nabla f(K)}{U_1} \vect(U_1)\right]
\;=\;
\EX_{U_1}[\vect(U_1)\vect(U_1)^T]\vect(\nabla f(K))
\;=\;
\vect(\nabla f(K))
\end{multline*}
and thus $\EX[\widehat{\nabla}f(K)] = \nabla f(K)$. Here, we have utilized the fact that for the uniformly distributed random variable  $\vect(U_1)$ over the sphere $\sqrt{mn} \, S^{mn-1}$, 
$
\EX_{U_1} [\vect(U_1)\vect(U_1)^T] 
=
I.
$

\subsubsection{Local boundedness of the function $f(K)$}
\label{subsec.localBoundedness}
An important requirement for the gradient estimation scheme in Algorithm~\ref{alg.GE} is the stability of the perturbed closed-loop systems, i.e., ${K}\pm r U_i\in\cSK$; violating this condition leads to an exponential growth of the state and control signals. Moreover, this condition is necessary and sufficient for $\widetilde{\nabla} f(K)$ to be well defined. In Proposition~\ref{prop.localdisk}, we establish a radius within which any perturbation of $K\in\cSK$ remains stabilizing. 

\begin{myprop}\label{prop.localdisk}
	For any stabilizing feedback gain $K\in\cSK$, we have 
	$
	\{ \hat{K}  \in \R^{m\times n}\,|\, \norm{\hat{K} - K}_2 < \zeta\}
	\subset
	\cSK
	$
	where
	\begin{align*}
	\zeta 
	\;\DefinedAs\; {\lambda_{\min}(\Omega)}/ \! \left( 2\,\norm{B}_2 \,\norm{X(K)}_2 \right)
	\end{align*}
	and $X(K)$ is given by~\eqref{eq.X}.
\end{myprop}
\begin{proof}
	See Appendix~\ref{app.localBoundedness}.
\end{proof}
	
If we choose the parameter $r$ in Algorithm~\ref{alg.GE} to be smaller than $\zeta$, then the sample feedback gains $K\pm r U_i$ are all stabilizing. In this paper, we further require that the parameter $r$ is small enough so that $K\pm rU_i\in\cSK(2a)$ for all $K\in\cSK(a)$. Such upper bound on $r$ is provided in the next lemma.
\begin{mylem}\label{lem.cSK2a}
	For any $U\in\R^{m\times n}$ with $\norm{U}_F\le\sqrt{mn}$ and $K\in\cSK(a)$, 
	$
	K+r(a)U\in\cSK(2a)
	$  
	where 
	$
	r(a) \DefinedAs \tilde{c}/a
	$
	for some positive constant $\tilde{c}$ that depends on the problem data.  
\end{mylem}
\begin{proof}
	See Appendix~\ref{app.localBoundedness}.
\end{proof}
Note that for any $K\in\cSK(a)$, and $r\le r(a)$ in Lemma~\ref{lem.cSK2a}, $\widetilde{\nabla} f(K)$ is well defined because $K+rU_i\in\cSK(2a)$ for all $i$.

\subsubsection{Bounding the bias}
\label{subsec.thirdOrder}

Herein, we establish an upper bound on the difference between the output $\widebar{\nabla} f(K)$ of Algorithm~\ref{alg.GE} and the unbiased estimate $\widehat{\nabla}f(K)$ of the gradient $\nabla f(K)$. This is accomplished by bounding the difference between these two quantities and $\widetilde{\nabla}f(K)$ through the use of the triangle inequality
\begin{align}\label{eq.megaTriangle}
\norm{\widehat{\nabla} f(K) \,-\, \widebar{\nabla} f(K)}_F
\;\le\;
\norm{\widetilde{\nabla} f(K) \,-\, \widebar{\nabla} f(K)}_F
\;+\;
\norm{\widehat{\nabla}  f(K) \,-\, \widetilde{\nabla} f(K)}_F.
\end{align}
The first term on the right-hand side of~\eqref{eq.megaTriangle} arises from a bias caused by the finite simulation time in Algorithm~\ref{alg.GE}. The next proposition quantifies an upper bound on this term.
\begin{myprop}\label{prop.finiteTime}
	For any $K\in\cSK(a)$, the output of Algorithm~\ref{alg.GE} with parameter $r\le r(a)$ (given by~Lemma~\ref{lem.cSK2a}) satisfies
	\begin{align*}
	\norm{\widetilde{\nabla}f(K) - \widebar{\nabla}f(K)}_F
	\, \le \,
	\dfrac{\sqrt{m n} \max_i\norm{x_i}^2}{r} 
	\,
	\kappa_1(2a)
	\,
	\mre^{-\kappa_2(2a)\tau}
	\end{align*}	
	where  
	$\kappa_1(a)>0$ is a degree $5$ polynomial and $\kappa_2(a)>0$ is inversely proportional to $a$ and they are given by~\eqref{eq.kappa12}.
\end{myprop}
\begin{proof}
	See Appendix~\ref{app.finiteTime}.
\end{proof}

Although small values of $r$ may result in a large error $\norm{\widetilde{\nabla}f(K) - \widebar{\nabla}f(K)}_F$, the exponential dependence of the upper bound in Proposition~\ref{prop.finiteTime} on the simulation time $\tau$ implies that this error can be readily controlled  by increasing  $\tau$. In the next  proposition, we handle the second term in~\eqref{eq.megaTriangle}.
\begin{myprop}\label{prop.thirdOrder}
	For any $K\in\cSK(a)$ and $r\le r(a)$ (given by Lemma~\ref{lem.cSK2a}), we have
	\begin{align*}
	\norm{\widehat{\nabla}f(K) \, - \, \widetilde{\nabla} f(K)}_F
	\;\le\;
	\dfrac{( r m n)^2}{2} \, \ell(2a) \, \max_i\norm{x_i}^2
	\end{align*}
	where the function $\ell(a)>0$ is a degree 4 polynomial and it is given by~\eqref{eq.ell}.
\end{myprop}
\begin{proof}
	See Appendix~\ref{app.thirdOrder}.
\end{proof}
The third-derivatives of the functions $f_{x_i}(K)$ are utilized in the proof of Proposition~\ref{prop.thirdOrder}. It is also worth noting that unlike $\widebar{\nabla} f(k)$ and $\widetilde{\nabla}f(K)$, the unbiased gradient estimate $\widehat{\nabla} f(K)$ is independent of the parameter $r$.  Thus, Proposition~\ref{prop.thirdOrder} provides a quadratic upper bound on the estimation error in terms of $r$.

\vspace*{-2ex}
\subsection{Correlation between gradient and gradient estimate}

As mentioned earlier, one approach to analyzing convergence for the random search method in~\eqref{eq.RS} is to control the gradient estimation error $\widebar{\nabla} f(K) - \nabla f(K)$ by choosing a large number of samples $N$. For the one-point gradient estimation setting, this approach was taken in~\cite{fazgekakmes18} for the discrete-time LQR (and in~\cite{mohsoljovACC20} for the continuous-time LQR) and has led to an upper bound on the required number of samples for reaching $\epsilon$-accuracy that grows at least proportionally to~$1/\epsilon^4$.  Alternatively, our proof exploits the problem structure and shows that with a linear number of samples $N=\tilde{O}(n)$, where $n$ is the number of states, the gradient estimate $\widehat{\nabla} f(K)$ concentrates with {\em high probability\/} when projected to the direction of $\nabla f(K)$. In particular, in Propositions~\ref{prop.M1} and~\ref{prop.M2} we show that the following events occur with high probability for some positive scalars $\mu_1$, $\mu_2$, 
\begin{subequations} 
	\label{eq.megaEvents}
	\begin{align}\label{eq.megaEvent1}
	\sfM_1
	&\;\DefinedAs\;
	\left\{
	\inner{\widehat{\nabla} f(K)}{\nabla f(K)}
	\;\ge\;
	\mu_1\norm{\nabla f(K)}^2_F
	\right\}
	\\[0.cm]\label{eq.megaEvent2}
	\sfM_2
	&\;\DefinedAs\;
	\left\{
	\norm{\widehat{\nabla} f(K)}^2_F
	\;\le\;
	\mu_2\norm{\nabla f(K)}^2_F
	\right\}.
	\end{align}
\end{subequations}
To justify the definitions of these events, we first show that if they both take place then the unbiased estimate $\widehat{\nabla}f(K)$ can be used to decrease the objective error by a geometric factor.
\begin{myprop}\label{prop.linearConvG}[Approximate GD]
	If  the matrix $G\in\R^{m\times n}$ and the feedback gain $K\in\cSK(a)$ are such that
	\begin{subequations}\label{eq.gAs}
		\begin{align}\label{eq.gAs2}
		\inner{G}{\nabla f(K)}
		&\;\ge\;
		\mu_1
		\norm{\nabla f(K)}_F^2 
		\\[0.cm]\label{eq.gAs1}
		\norm{G}_F^2
		&\;\le\;
		\mu_2 
		\norm{\nabla f(K)}_F^2
		\end{align}
	\end{subequations}
	for some positive scalars $\mu_1$ and $\mu_2$, then  
	$
	K - \alpha  G  \in \cSK(a)
	$ for all $\alpha\in[0,\mu_1/(\mu_2L_f)]$, and
	\begin{align*}
	f(K-\alpha G) \, - \, f(K^\star)
	\;\le\;
	\gamma \left(f(K) \, - \, f(K^\star)\right)
	\end{align*}
	with 
	$
	\gamma
	=
	1-\mu_f \mu_1\alpha.
	$
	Here,  $L_f$ and $\mu_f$ are the smoothness and the PL parameters of the function $f$ over $\cSK(a)$.
\end{myprop}
\begin{proof}
	See Appendix~\ref{app.approxGD}.
\end{proof}
\begin{myrem}
	The fastest convergence rate guaranteed by Proposition~\ref{prop.linearConvG},
	$
	\gamma
	=
	1- {\mu_f \mu_1^2}/({L_f\mu_2}),
	$
	is achieved with the stepsize $\alpha = \mu_1/(\mu_2L_f)$. This rate bound is tight in the sense that if $G= c\nabla f(K)$, for some $c>0$, we recover the standard convergence rate $\gamma = 1 - \mu_f/L_f$ of gradient descent. 
\end{myrem}
We next quantify the probability of the events $\sfM_1$ and $\sfM_2$.  In our proofs, we exploit modern non-asymptotic statistical analysis of the concentration of random variables around their average. While in Appendix~\ref{app.probabilistic} we set notation and provide basic definitions of key concepts, we refer the reader to a recent book~\cite{ver18} for a comprehensive discussion. Herein, we use $c$, $c'$, $c''$, etc.\ to denote positive absolute constants.

\subsubsection{Handling $\sfM_1$}
\label{subsec.varianceMu1}

We first exploit the problem structure to confine the dependence of $\widehat{\nabla}f(K)$ on the random initial conditions $x_i$ into a zero-mean random vector. In particular, for any $K\in\cSK$ and $x_0\in\R^n$, 
\begin{align*}
\nabla f(K)
\;=\;
E X,
~~~
\nabla f_{x_0}(K)
\;=\;
E X_{x_0}
\end{align*} 
where $E \DefinedAs 2(R K - B^T P(K))\in\R^{m\times n}$ is a fixed matrix,  
$
X
=
-\cA_K^{-1}(\Omega)
$, and
$
X_{x_0}
=
-\cA_K^{-1}(x_0x_0^T).
$
This allows us to represent the unbiased estimate $\widehat{\nabla} f(K)$ of the gradient as 
\begin{subequations}
	\begin{align}\label{eq.splitting}
	\widehat{\nabla} f(K)
	& \;=\;
	\dfrac{1}{N}
	\sum_{i \, = \, 1}^{N}
	\inner{E X_{x_i}}{U_i} U_i
	\;=\; 
	\widehat{\nabla}_1 
	\;+\;
	\widehat{\nabla}_2
	\\[0.cm]
	\label{eq.nabla_1}
	\widehat{\nabla}_1 
	&\; = \;
	\dfrac{1}{N}
	\sum_{i \, = \, 1}^{N}
	\inner{E (X_{x_i}-X)}{U_i} U_i
	\\[0.cm]
	\label{eq.nabla_2}
	\widehat{\nabla}_2 
	& \; = \;
	\dfrac{1}{N}
	\sum_{i \, = \, 1}^{N}
	\inner{\nabla f(K)}{U_i} U_i.
	\end{align}
\end{subequations}
Note that $\widehat{\nabla}_2$ does not depend on the initial conditions $x_i$. Moreover, from $\EX[X_{x_i}] = X$ and the independence of $X_{x_i}$ and $U_i$, we have
$
\EX[\widehat{\nabla}_1] = 0
$
and 
$
\EX[\widehat{\nabla}_2 ] = \nabla f(K). 
$

In Lemma~\ref{lem.Solt1}, we show that $\inner{\widehat{\nabla}_1}{\nabla f(K)}$ can be made arbitrary small with a large number of samples $N$. This allows us to analyze the probability of the event $\sfM_1$ in~\eqref{eq.megaEvents}.
\begin{mylem}\label{lem.Solt1}
	Let $U_1,\ldots, U_N\in\R^{m\times n}$ be i.i.d.\ random matrices with each $\vect(U_i)$ uniformly distributed on the sphere $\sqrt{mn} \, S^{mn-1}$ and let $X_1, \ldots, X_N\in\R^{n\times n}$ be i.i.d.\ random matrices distributed according to $\cM(xx^T)$. Here, $\cM$ is a linear operator and $x\in\R^n$ is a random vector whose entries are i.i.d., zero-mean, unit-variance, sub-Gaussian random variables with sub-Gaussian norm less than $\kappa$. 
	For any fixed matrix $E\in\R^{m\times n}$ and positive scalars $\delta$ and $\beta$, if 
	\begin{align}\label{eq.sample}
	N 
	\, \ge \,
	C \, (\beta^2 \kappa^2/ \delta)^2
	\left( \norm{\cM^*}_2 \, + \, \norm{\cM^*}_S \right)^2 n \log^6 \! {n}
	\end{align}
	then, with probability not smaller than $1-C'N^{-\beta}-4N\mre^{-\tfrac{n}{8}}$, 
	\begin{align*}
	\abs{\dfrac{1}{N} \sum_{i \, = \, 1}^N \inner{E\left(X_i-X\right)}{ U_i} \inner{ E X}{ U_i}}
	\, \le \, 
	\delta \norm{E X}_F\norm{E}_F
	\end{align*}
	where $X\DefinedAs\EX[X_1]=\cM(I)$.
\end{mylem}
\begin{proof}
	See Appendix~\ref{app.varianceMu1}.
\end{proof}

In Lemma~\ref{lem.Solt2}, we show that $\inner{\widehat{\nabla}_2}{\nabla f(K)}$ concentrates with high probability around its average $\norm{\nabla f(K)}_F^2$. 
\begin{mylem}\label{lem.Solt2}
	Let  $U_1,\ldots, U_N\in\R^{m\times n}$ be i.i.d.\ random matrices with each $\vect(U_i)$ uniformly distributed on the sphere $\sqrt{mn} \, S^{mn-1}$. Then, for any $W\in\R^{m\times n}$ and $t \in (0, 1]$,
	\begin{align*}
	\bbP
	\left\{
	\dfrac{1}{N} 
	\, 
	\sum_{i \, = \, 1}^N \inner{W}{ U_i}^2
	\, < \, 
	(1 \, - \, t) \norm{W}_F^2
	\right\}
	\, \le \;
	2 \, \mre^{-c Nt^2}.
	\end{align*}
\end{mylem}
\begin{proof}
	See Appendix~\ref{app.varianceMu1}.
\end{proof}

In Proposition~\ref{prop.M1}, we use Lemmas~\ref{lem.Solt1} and~\ref{lem.Solt2} to address~$\sfM_1$.
\begin{myprop}\label{prop.M1}
	Under Assumption~\ref{ass.initialCond}, for any stabilizing feedback gain $K\in\cSK$ and positive scalar $\beta$, if
	\begin{align*}
	N 
	\, \ge \, 
	C_1 \, \frac{\beta^4 \kappa^4}{\lambda_{\min}^2 (X)}  
	\left(\norm{( \cA_K^*)^{-1}}_2 +  \norm{ ( \cA_K^*)^{-1}}_S\right)^2 
	{n} \log^6 \! {n}
	\end{align*}
	then the event $\sfM_1$ in~\eqref{eq.megaEvents} with $\mu_1\DefinedAs 1/4$ satisfies
	$
	\bbP(\sfM_1)
	\ge 
	1 - C_2N^{-\beta} - 4N\mre^{-\tfrac{n}{8}} - 2\mre^{-C_3 N}.
	$
\end{myprop}
\begin{proof}
	We use Lemma~\ref{lem.Solt1} with $\delta \DefinedAs \lambda_{\min}(X)/4$ to show that 
	\begin{subequations}
		\be
		\ba{rcl}
		\abs{\inner{\widehat{\nabla}_1}{ \nabla f(K)}}
		& \!\!\! \le \!\!\! &
		\delta \, \norm{E X}_F \, \norm{E}_F
\;\le\;
		\dfrac{1}{4} \, \norm{E X}_F^2
		\;=\;
		\dfrac{1}{4} \, \norm{\nabla f(K)}_F^2.
		\ea
		\label{eq.Proofnablatilde1a}
		\ee
		holds with probability not smaller than
		$1-C'N^{-\beta}-4N\mre^{-\tfrac{n}{8}}$.
		Furthermore, Lemma~\ref{lem.Solt2} with $t \DefinedAs {1}/{2}$ implies that
		\begin{align}\label{eq.Proofnablatilde1b}
		\inner{\widehat{\nabla}_2}{ \nabla f(K)}
		&\;\ge\;
		\dfrac{1}{2} \, \norm{\nabla f(K)}_F^2
		\end{align}
	\end{subequations}
	holds with probability not smaller than $1-2\mre^{-cN}$. Since $\widehat{\nabla} f(K) = \widehat{\nabla}_1 + \widehat{\nabla}_2$, we can use a union bound to combine~\eqref{eq.Proofnablatilde1a} and~\eqref{eq.Proofnablatilde1b}. This together with a triangle inequality completes the proof.
\end{proof}

\subsubsection{Handling $\sfM_2$}
\label{subsec.varianceMu2}

In Lemma~\ref{lem.Moh3}, we quantify a high probability upper bound on $\norm{\widehat{\nabla}_1}_F/\norm{\nabla f(K)}$. This lemma is analogous to Lemma~\ref{lem.Solt1} and it allows us to analyze the probability of the event $\sfM_2$ in~\eqref{eq.megaEvents}.
\begin{mylem}\label{lem.Moh3}
	Let $X_i$ and $U_i$ with $i=1,\ldots, N$ be random matrices defined in Lemma~\ref{lem.Solt1}, $X\DefinedAs\EX[X_1]$, and let $N\ge c_0 n$. Then, for any $E\in\R^{m\times n}$ and positive scalar $\beta$,   
	\begin{align*}
	\dfrac{1}{N} \, \norm{\sum_{i \, = \, 1}^N \inner{E\left(X_i-X\right)}{ U_i}  U_i}_F
	\;\le\;
	c_1\beta\, \kappa^2 (\norm{\cM^*}_2+\norm{\cM^*}_S)\norm{E}_F
	\sqrt{m n}
	\;
	\log n
	\end{align*}
	with probability not smaller than $1-c_2 (n^{-\beta} + N \mre^{-\frac{n}{8}})$.
\end{mylem}
\begin{proof}
	See Appendix~\ref{app.probabilistic}.
\end{proof}

In Lemma~\ref{lem.Moh2}, we quantify a high probability upper bound on $\norm{\widehat{\nabla}_2}_F/\norm{\nabla f(K)}$.
\begin{mylem}\label{lem.Moh2}
	Let $U_1, \ldots, U_N\in\R^{m\times n}$ be i.i.d.\ random matrices with $\vect(U_i)$ uniformly distributed on the sphere $\sqrt{mn} \, S^{mn-1}$ and let $N\ge C n$. Then, for any $W\in\R^{m\times n}$, 
	\begin{align*}
	\bbP 
	\bigg\{ 
	\dfrac{1}{N} \, \norm{\sum_{j \, = \, 1}^N \inner{W}{U_j} U_j}_F 
	\, > \,
	C'\sqrt{m}\norm{W}_F \bigg\}
	\;\le\;
	2N\mre^{-\frac{mn}{8}} \, + \, 2\mre^{-\hat{c}N}.
	\end{align*}
\end{mylem}
\begin{proof}
	See Appendix~\ref{app.probabilistic}.
\end{proof}

In Proposition~\ref{prop.M2}, we use Lemmas~\ref{lem.Moh3} and~\ref{lem.Moh2} to address~$\sfM_2$.
\begin{myprop}\label{prop.M2}
	Let Assumption~\ref{ass.initialCond} hold. Then, for any $K\in\cSK$, scalar $\beta>0$, and	 
	$ 
	N \ge C_4n,
	$
	the event $\sfM_2$ in~\eqref{eq.megaEvents} with 
	$
	\mu_2
	\DefinedAs
	C_5  \big( \beta \kappa^2 \, \dfrac{\norm{(\cA_K^*)^{-1}}_2+\norm{(\cA_K^*)^{-1}}_S}{\lambda_{\min}(X)} \sqrt{mn} \log n + \sqrt{m}\big)^2
	$
	satisfies
	\begin{align*}
	\bbP(\sfM_2)
	\;\ge\;
	1-C_6 (n^{-\beta} + N \mre^{-\frac{n}{8}} + \mre ^{-C_7 N}).
	\end{align*}
\end{myprop}
\begin{proof}
	We use Lemma~\ref{lem.Moh3} to show that, with probability at least $1-c_2 (n^{-\beta} \,+\, N \mre^{-\frac{n}{8}})$, $\widehat{\nabla}_1$ satisfies
		\begin{multline*} \label{eq.Proofnablatilde1c}
		\norm{\widehat{\nabla}_1}_F
		\;\le\;
		c_1\beta \kappa^2 (\norm{( \cA_K^*)^{-1}}_2+\norm{(\cA_K^*)^{-1}}_S)\norm{E}_F\,\sqrt{m n}\,\log n
		\; \le
		\\[0.cm]
		c_1\beta \kappa^2 \, \dfrac{\norm{( \cA_K^* )^{-1}}_2+\norm{ ( \cA_K^* )^{-1}}_S}{\lambda_{\min}(X)}
		\, \norm{\nabla f(K)}_F \, \sqrt{m n} \,\log n.
		\end{multline*}
		Furthermore, we can use Lemma~\ref{lem.Moh2} to show that, with probability not smaller than $1-2N\mre^{-\frac{mn}{8}}- 2\mre^{-\hat{c}N}$, $\widehat{\nabla}_2$ satisfies
		\[
		\label{eq.Proofnablatilde1d}
		\norm{\widehat{\nabla}_2}_F
		\;\le\;
		C'\sqrt{m}\norm{\nabla f(K)}_F.
		\]
	Now, since $\widehat{\nabla} f(K) = \widehat{\nabla}_1 + \widehat{\nabla}_2$, we can use a union bound to combine the last two inequalities. This together with a triangle inequality completes the proof.
\end{proof}

	\vspace*{-2ex}
\section{Model-free control design}
\label{sec.rndSearch}

In this section, we prove a more formal version of Theorem~\ref{thm.rndSearchInformal}. 
\begin{mythm}\label{thm.randomSearchFormal}
	Consider the random search method~\eqref{eq.RS} that uses the  gradient estimates of Algorithm~\ref{alg.GE} for finding the optimal solution $K^\star$ of LQR problem~\eqref{eq.optProbK}. Let the initial condition $x_0$ obey Assumption~\ref{ass.initialCond} and let the simulation time $\tau$, the smoothing constant $r$, and the number of samples $N$  satisfy 
	\begin{align}
	\label{eq.sampleComplexity}
	\tau 
	\, \ge \,
	\theta'(a) \log \dfrac{1}{r\epsilon},
	\quad r<\min\{r(a),\theta''(a)\sqrt{\epsilon}\},
	\quad
	N
	\,\ge\, 
	c_1 (1  \,+\,  \beta^4 \kappa^4 \, \theta (a) \log^6 \! {n})\, n
	\end{align}
	for some  $\beta>0$ and a desired accuracy $\epsilon>0$.
	Then,  for any initial condition $K^0\in\cSK(a)$,~\eqref{eq.RS} with the constant stepsize $\alpha \le 1/(32\mu_2(a)L_f)$ achieves
	$
	f(K^k)-f(K^\star)
	\le
	\epsilon
	$
	with probability not smaller than $1-kp-2kN\mre^{-n}$ in at most
	\[
	k
	\;\le\;
	\left.
	\left(\log \dfrac{ f(K^0) \, - \, f(K^\star)}{\epsilon} \right)
	\middle /
	\left(
	\log \dfrac{1}{1 \, - \, \mu_f(a) \alpha/8} 
	\right)
	\right.
	\]
	iterations. Here,  $p\DefinedAs c_2 (n^{-\beta} + N^{-\beta}+ N \mre^{-\frac{n}{8}} + \mre ^{-c_3 N})$, $
	\mu_2\DefinedAs
	c_4\left( \sqrt{m} + \beta \kappa^2 \theta(a) \sqrt{mn}\log n \right)^2$,
	$c_1,\ldots,c_4$ are positive absolute constants, $\mu_f$ and $L_f$ are the PL and smoothness parameters of the function $f$ over the sublevel set $\cSK(a)$,  $\theta$, $\theta'$, $\theta''$ are positive functions that depend only on the parameters of the LQR problem, and $r(a)$ is given by~Lemma~\ref{lem.cSK2a}.
\end{mythm} 

\begin{proof}
	The proof combines Propositions~\ref{prop.finiteTime},~\ref{prop.thirdOrder},~\ref{prop.linearConvG},~\ref{prop.M1},~and~\ref{prop.M2}. We first show that  for any $r\le r(a)$ and $\tau>0$,
	\begin{align}\label{eq.mainResulttemp1}
	\norm{\widebar{\nabla}f(K) \, - \, \widehat{\nabla}f(K)}_F
	\;\le\;
	\sigma
	\end{align}
	with probability not smaller than $1- 2N\mre^{-n}$, where
	\begin{align*}
	\sigma
	\DefinedAs
	c_5(\kappa^2+1) 
	\!
	\left(\dfrac{n\sqrt{m} }{r}\kappa_1(2a)\mre^{-\kappa_2(2a)\tau}
	+
	\dfrac{r^2m^2n^{\frac{5}{2}}}{2}\ell(2a)\right).
	\end{align*}
	Here, $r(a)$, $\kappa_i(a)$, and $\ell(a)$ are positive functions that are given by~Lemma~\ref{lem.cSK2a}, Eq.~\eqref{eq.kappa12}, and Eq.~\eqref{eq.ell}, respectively.
	
	Under Assumption~\ref{ass.initialCond}, the  vector $v\sim\cD$ satisfies~\cite[Eq.~(3.3)]{ver18},
	$
	\bbP\left\{
	\norm{v}
	\le
	c_5(\kappa^2+1)\sqrt{n}	
	\right\}
	\ge
	1- 2\mre^{-n}.
	$
	Thus, for the random initial conditions $x_1,\ldots,x_N\sim\cD$, we can apply the union bound (Boole's inequality) to obtain
	\begin{align}\label{eq.boundonx_0}
	\bbP\left\{
	\max_i \norm{x_i}
	\;\le\;
	c_5(\kappa^2+1)\sqrt{n}	
	\right\}
	\;\ge\;
	1\, - \, 2N\mre^{-n}.
	\end{align}
	Now, we combine Propositions~\ref{prop.finiteTime} and~\ref{prop.thirdOrder} to write 
	\vspace*{-1ex}
	\begin{align*}
	\norm{\widebar{\nabla}f(K) - \widehat{\nabla}f(K)}_F
	\;\le\;
	\left(\dfrac{\sqrt{m n} }{r} \, \kappa_1(2a)\mre^{-\kappa_2(2a)\tau}
	\,+\,
	\dfrac{( r m n)^2}{2} \, \ell(2a)\right)\max_i\norm{x_i}^2
	\, \le \,
	\sigma.
	\end{align*}
	The first inequality is obtained by  combining Propositions~\ref{prop.finiteTime} and~\ref{prop.thirdOrder} through the use of the triangle inequality, and the second inequality follows from~\eqref{eq.boundonx_0}. This completes the proof of~\eqref{eq.mainResulttemp1}.


	Let  $\theta(a)$ be a uniform upper bound on
	\begin{align*}
	\dfrac{ \norm{ ( \cA_K^* )^{-1}}_2 + \norm{( \cA_K^* )^{-1}}_S }{ \lambda_{\min}(X)}
	\;\le\;
	\theta(a)
	\end{align*}
	for all $K\in\cSK(a)$; see Appendix~\ref{app.invLyapBound} for a discussion on $\theta(a)$. Since, the number of samples satisfies~\eqref{eq.sampleComplexity}, for any given $K\in\cSK(a)$, we can combine Propositions~\ref{prop.M1} and~\ref{prop.M2} with a union bound to show that
	\vspace*{-2ex}
	\begin{subequations}
		\label{eq.mainResulttemp0}
		\begin{align}\label{eq.mainResulttemp01}
		\inner{\widehat{\nabla}f(K)}{\nabla f(K)}
		&\;\ge\;
		\mu_1
		\norm{\nabla f(K)}_F^2 
		\\[0.cm]\label{eq.mainResulttemp02}
		\norm{\widehat{\nabla}f(K)}_F^2
		&\;\le\;
		\mu_2 
		\norm{\nabla f(K)}_F^2
		\end{align}
	\end{subequations}
	holds with probability not smaller than $1-p$, where
	$\mu_1 = 1/4$, and $\mu_2$ and $p$ are determined in the statement of the theorem.

	Without loss of generality, let us assume that the initial error satisfies $f(K^0)-f(K^\star)>\epsilon$. We next show that
	\begin{subequations}\label{eq.mainResulttemp2}
		\begin{align}
		\inner{\widebar{\nabla}f(K^0)}{\nabla f(K^0)}
		&\;\ge\;
		\dfrac{\mu_1}{2}
		\norm{\nabla f(K^0)}_F^2 
		\\[0.1cm]
		\norm{\widebar{\nabla}f(K^0)}_F^2
		&\;\le\;
		4\mu_2 
		\norm{\nabla f(K^0)}_F^2
		\end{align}
	\end{subequations}
	holds with probability not smaller than $1-p-2N\mre^{-n}$.
	
	Since the function $f$ is gradient dominant over the sublevel set $\cSK(a)$  with parameter $\mu_f$, combining $f(K^0)-f(K^\star)>\epsilon$ and~\eqref{eq.GradDom} yields $
		\norm{\nabla f(K^0)}_F
		\ge
		\sqrt{2\mu_f\epsilon}.
		$ Also, let the positive scalars $\theta'(a)$ and $\theta''(a)$ be such that for any pair of $\tau$ and $r$ satisfying
		$
		\tau 
		\, \ge \,
		\theta'(a) \log (1/(r\epsilon))$ and $r<\min\{r(a),\theta''(a)\sqrt{\epsilon}\}$,  the upper bound $\sigma$ in~\eqref{eq.mainResulttemp1} becomes smaller than $
		\sigma \le \sqrt{2\mu_f\epsilon} \,\min \, \{ \mu_1/2, \sqrt{\mu_2} \}.$
		The choice of $\theta'$ and $\theta''$ with the above property is straightforward using the definition of $\sigma$. Combining $\norm{\nabla f(K^0)}_F
		\ge
		\sqrt{2\mu_f\epsilon}$ and $\sigma \le \sqrt{2\mu_f\epsilon} \,\min \, \{ \mu_1/2, \sqrt{\mu_2} \}$ yields
		\begin{align}\label{eq.Knotbound_sub}
		\sigma 
		\;\le\;
		\norm{\nabla f(K^0)}_F \,\min \, \{ \mu_1/2, \sqrt{\mu_2} \}.	
		\end{align}
	Using the union bound, we have
	\begin{align*}
	\inner{\widebar{\nabla}f(K^0)}{\nabla f(K^0)}
	&\;=\;\inner{\widehat{\nabla}f(K^0)}{\nabla f(K^0)}
	\;+\;
	\inner{\widebar{\nabla}f(K^0) - \widehat{\nabla}f(K^0)}{\nabla f(K^0)}
	\\[0.cm]
	&
	\;
	\overset{(a)}{\ge}\;
	\mu_1
	\norm{\nabla f(K^0)}_F^2
	\;-\;
	\norm{\widebar{\nabla}f(K^0) - \widehat{\nabla}f(K^0)}_F\norm{\nabla f(K^0)}_F
	\\[0.cm]
	&
	\;
	\overset{(b)}{\ge}\;
	\mu_1 \norm{\nabla f(K^0)}_F^2 \,-\, \sigma
	\norm{\nabla f(K^0)}_F
	\;\;
	\overset{(c)}{\ge}\;\;
	\dfrac{\mu_1}{2}
	\norm{\nabla f(K^0)}_F^2
	\end{align*}
	with probability not smaller than $1-p-2N\mre^{-n}$. Here, $(a)$ follows from combining~\eqref{eq.mainResulttemp01} and the Cauchy-Schwartz inequality, $(b)$ follows from~\eqref{eq.mainResulttemp1}, and $(c)$ follows from~\eqref{eq.Knotbound_sub}.  Moreover,
	\begin{multline*}
	\norm{\widebar{\nabla}f(K^0)}_F
	\;\overset{(a)}{\le}\;
	\norm{\widehat{\nabla}f(K^0)}_F + \norm{\widebar{\nabla}f(K^0) - \widehat{\nabla}f(K^0)}_F
	\;\overset{(b)}{\le}\;
	\sqrt{\mu_2}
	\norm{\nabla f(K^0)}_F
	\,+\,\sigma
	\;\overset{(c)}{\le}\;
	2\sqrt{\mu_2}
	\norm{\nabla f(K^0)}_F
	\end{multline*}
	where $(a)$ follows from the triangle inequality, $(b)$  from~\eqref{eq.mainResulttemp1}, and $(c)$  from~\eqref{eq.Knotbound_sub}.  This completes the proof of~\eqref{eq.mainResulttemp2}.
	
	Inequality~\eqref{eq.mainResulttemp2} allows us to apply Proposition~\ref{prop.linearConvG} and obtain with probability not smaller than $1-p-2N\mre^{-n}$ that for the stepsize
	$\alpha \le \mu_1/(8\mu_2L_f)$, we have 
	$
	K^1 \in \cSK(a)
	$ 
	and also
	$
	f(K^1)-f(K^\star)
	\le
	\gamma \left(f(K^0)-f(K^\star)\right),
	$
	with 
	$
		\gamma
		\;=\;
		1-\mu_f\mu_1\alpha/2
		$, where  $L_f$ is the smoothness parameter of the function $f$ over $\cSK(a)$. Finally, using the union bound, we can repeat this procedure via induction  to obtain that for some
	\[
	k\;\le\;
	\left. \left(\log \dfrac{f(K^0)-f(K^\star)}{\epsilon}\right)
	\middle/
	 \left( \log \dfrac{1}{\gamma} \right)
	\right.
	\]
	the error satisfies
	\begin{align*}
	f(K^k)-f(K^\star)
	\;\le\;
	\gamma^k \left(f(K^0)-f(K^\star)\right)
	\;\le\;
	\epsilon
	\end{align*} 
	with probability not smaller than $1-kp-2kN\mre^{-n}$.  
\end{proof}
\begin{myrem}
For the failure probability in Theorem~\ref{thm.randomSearchFormal} to be negligible,  the problem dimension $n$ needs to be large. Moreover, to account for the conflicting term $N\mre^{-n/8}$ in the failure probability, we can require a crude exponential bound $N\le \mre^{n/16}$ on the sample size. We also note that although Theorem~\ref{thm.randomSearchFormal} only guarantees convergence in the objective value, similar to the proof of Theorem~\ref{thm.main-cts}, we can use Lemma~\ref{lem.errorRelation} that relates the error in optimization variable, $K$, and the error in the objective function, $f(K)$, to obtain convergence guarantees in the optimization variable as well.    
\end{myrem}
\begin{myrem}
	Theorem~\ref{thm.randomSearchFormal} requires the lower bound on the simulation time $\tau$ in~\eqref{eq.sampleComplexity} to ensure that, for any desired accuracy $\epsilon$, the smoothing constant $r$ satisfies $r \ge  (1/\epsilon) \, \mre^{-\tau/\theta'(a)}$. As we demonstrate in the proof, this requirement accounts for the bias that arises from a finite value of $\tau$. Since this form of bias can be readily controlled by increasing $\tau$, the above lower bound on $r$ does not contradict the upper bound $r=O(\sqrt{\epsilon})$ required by Theorem~\ref{thm.randomSearchFormal}. Finally, we note that letting $r \rightarrow 0$ can cause large bias in the presence of other sources of inaccuracy in the function approximation process.
\end{myrem}

	\vspace*{-2ex}
\section{Computational experiments}
\label{sec.examples}

We consider a mass-spring-damper system with $s$ masses, where we set all mass, spring, and damping constants to unity. In state-space representation~\eqref{eq.linsys1}, the state  $x = [\,p^T\,v^T\,]^T$ contains the position and velocity vectors and the dynamic and input matrices are given by
\[
A
\;=\;
\tbt{\phantom{-}0}{\phantom{-}I}{-T}{-T},
~
B 
\;=\;
\tbo{0}{I}
\]
where $0$ and $I$ are ${s\times s}$ zero and identity matrices, and $T$ is a Toeplitz matrix with $2$ on the main diagonal and $-1$ on the first super and sub-diagonals.

	\vspace*{-2ex}
\subsection{Known model}

To  compare the performance of gradient descent methods~\eqref{eq.GDK} and~\eqref{eq.GDY} on $K$ and $Y$, we solve the LQR problem with $Q=I+100\,\mre_1 \mre_1^T$, $R=I+1000\,\mre_4 \mre_4^T$, and $\Omega = I$ for $s\in\{10,\,20\}$ masses (i.e., $n = 2s$ state variables), where $\mre_i$ is the $i$th unit vector in the standard basis of $\bbR^n$. 

Figure~\ref{fig.cc1} illustrates the convergence curves for both algorithms with a stepsize selected using a backtracking procedure that guarantees stability of the closed-loop system. Both algorithms were initialized with $Y^0 = K^0 = 0$. Even though Fig.~\ref{fig.cc1} suggests that gradient decent/flow on $\cS_K$  converges faster than that on $\cS_Y$, this observation does not hold in general.

\begin{figure}[h]
	\centering
	\begin{tabular}{cccc}
		\hspace{-.6cm}
		\subfigure[]{\label{fig.tridraglg}}
		&&\hspace{-.6cm}
		\subfigure[]{\label{fig.triengylg}}
		&
		\\[-.5cm]\hspace{-.6cm}
		\begin{tabular}{c}
			\vspace{.25cm}
			\normalsize{\rotatebox{90}{ $
					\dfrac{f(K^k)-f(K^\star)}{f(K^0)-f(K^\star)}
					$}}
		\end{tabular}
		&\hspace{-.7cm}
		\begin{tabular}{c}
			\includegraphics[width=0.35\textwidth]{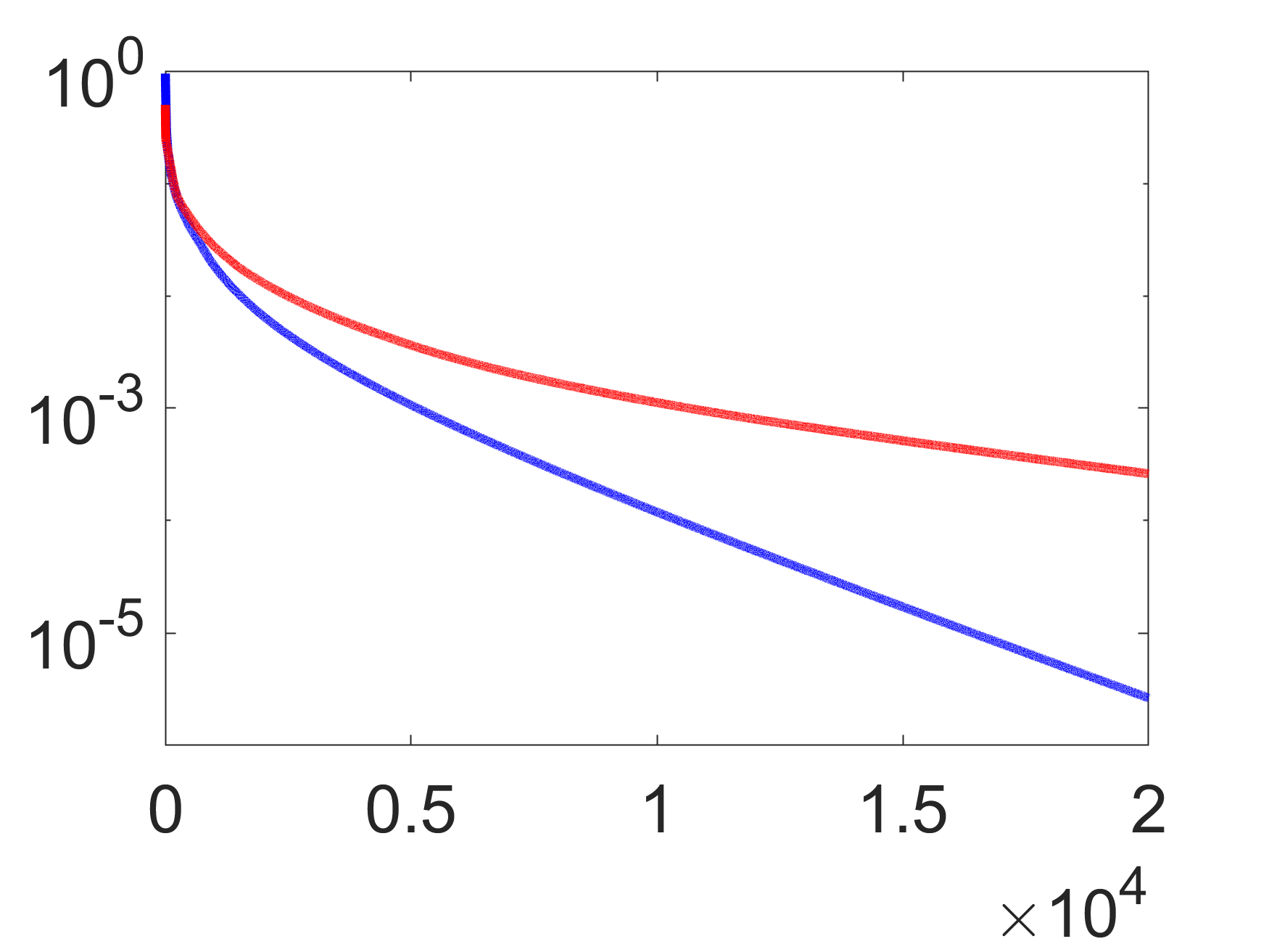}
			\\ { $k$}
		\end{tabular}
		\hspace{-.45cm}
		&\hspace{-.0cm}
		&\hspace{-.6cm}
		\begin{tabular}{c}
			\includegraphics[width=0.35\textwidth]{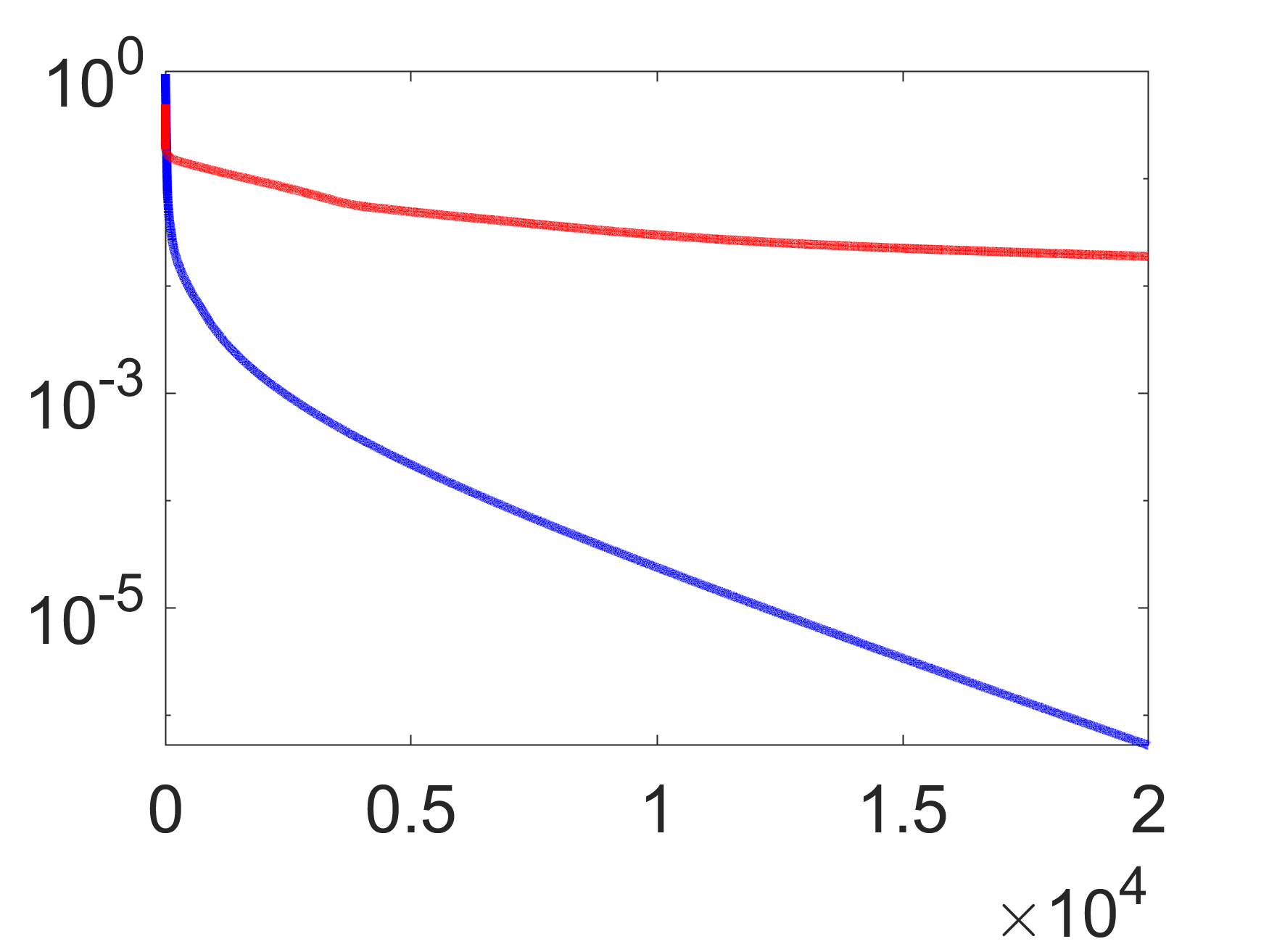}
			\\ {$k$}
		\end{tabular}
	\end{tabular}
	\vspace{-0.3cm}
	\caption{Convergence curves for gradient descent (blue) over the set $\cSK$, and gradient descent (red) over the set $\cSY$ with (a) $s = 10$ and (b) $s=20$ masses.}
	\label{fig.cc1}
\end{figure}

	\vspace*{-2ex}
\subsection{Unknown model}

To  illustrate our results on the accuracy of the  gradient estimation in Algorithm~\ref{alg.GE}  and the efficiency of our random search method, we consider the LQR problem with  $Q$ and $R$ equal to identity for $s=10$ masses (i.e., $n = 20$ state variables).  We also let the initial conditions $x_i$ in Algorithm~\ref{alg.GE} be standard normal and use  $N=n=2 s$ samples.

Figure~\ref{fig.tau} illustrates the dependence of the relative error $\norm{\widehat{\nabla}f(K) - \widebar{\nabla}f(K)}_F/\norm{\widehat{\nabla}f(K)}_F$ on the simulation time $\tau$ for $K=0$ and two values of the smoothing parameter $r=10^{-4}$ (blue) and $r=10^{-5}$ (red). We observe an exponential decrease in error for small values of $\tau$. In addition, the error does not pass a saturation level which is determined by $r$. We also see that, as $r$ decreases, this saturation level becomes smaller. These observations are in harmony with our theoretical developments; in particular, combining Propositions~\ref{prop.finiteTime} and~\ref{prop.thirdOrder} through the use of the triangle inequality yields
\begin{align*}
\norm{\widehat{\nabla}f(K) \, - \, \widebar{\nabla}f(K)}_F
\;\le\;
\left(\dfrac{\sqrt{m n}}{r} \, \kappa_1(2a) \, \mre^{-\kappa_2(2a)\tau}
\,+\,
\dfrac{r^2m^2n^2}{2} \, \ell(2a)\right) \max_i\norm{x_i}^2.
\end{align*}
This upper bound clearly captures the exponential dependence of the bias on the simulation time $\tau$ as well as the saturation level that depends quadratically on the smoothing parameter $r$.

In Fig.~\ref{fig.tauTrue}, we demonstrate the dependence of the total relative error $\norm{\nabla f(K) - \widebar{\nabla}f(K)}_F/\norm{\nabla f(K)}_F$ on the simulation time $\tau$ for two values of the smoothing parameter $r=10^{-4}$ (blue) and $r=10^{-5}$ (red), resulting from the use of $N=n$ samples. We observe that the distance between the approximate gradient and the true gradient is rather large. This is exactly why prior analysis of sample complexity and simulation time is subpar to our results. In contrast to the existing results which rely on the use of the estimation error shown in Fig.~\ref{fig.tauTrue}, our analysis shows that the simulated gradient $\widebar{\nabla}f(K)$ is close to the gradient estimate $\widehat{\nabla}f(K)$. While $\widehat{\nabla}f(K)$ is not close to the true gradient $\nabla f(K)$, it is highly correlated with it. This is sufficient for establishing convergence guarantees and it allows us to significantly improve upon existing results~\cite{fazgekakmes18,malpanbhakhabarwai19} in terms of sample complexity and simulation time reducing both to $O (\log \, (1/\epsilon))$.

Finally, Fig.~\ref{fig.RS} demonstrates linear convergence of the random search method~\eqref{eq.RS} with stepsize $\alpha=10^{-4}$, $r=10^{-5}$, and $\tau=200$ in Algorithm~\ref{alg.GE}, as established in Theorem~\ref{thm.randomSearchFormal}. In this experiment, we implemented Algorithm~\ref{alg.GE} using the $\mathrm{ode45}$ and $\mathrm{trapz}$ subroutines in MATLAB to numerically integrate the state/input penalties with the corresponding weight matrices $Q$ and $R$. However,  our theoretical results only account for an approximation error that arises from a finite simulation horizon.  Clearly,  employing empirical ODE solvers and numerical integration may introduce  additional errors in our gradient approximation that require further scrutiny.

\begin{figure*}[t]
	\centering
	\begin{tabular}{c@{\hspace{-0.2 cm}}c@{\hspace{0.2 cm}}c@{\hspace{-0.2 cm}}c c@{\hspace{0.2 cm}}c}
		&\subfigure[]{\label{fig.tau}}
		&&
		\subfigure[]{\label{fig.tauTrue}}
		&&
		\subfigure[]{\label{fig.RS}}
		\\[-.4cm]
		\begin{tabular}{c}
			\vspace{.25cm}
			\normalsize{\rotatebox{90}{ $
					\dfrac{\norm{\widehat{\nabla}f(K) - \widebar{\nabla}f(K)}_F}{\norm{\widehat{\nabla}f(K)}_F}
					$}}
		\end{tabular}
		&
		\begin{tabular}{c}
			\includegraphics[width=0.22\textwidth]{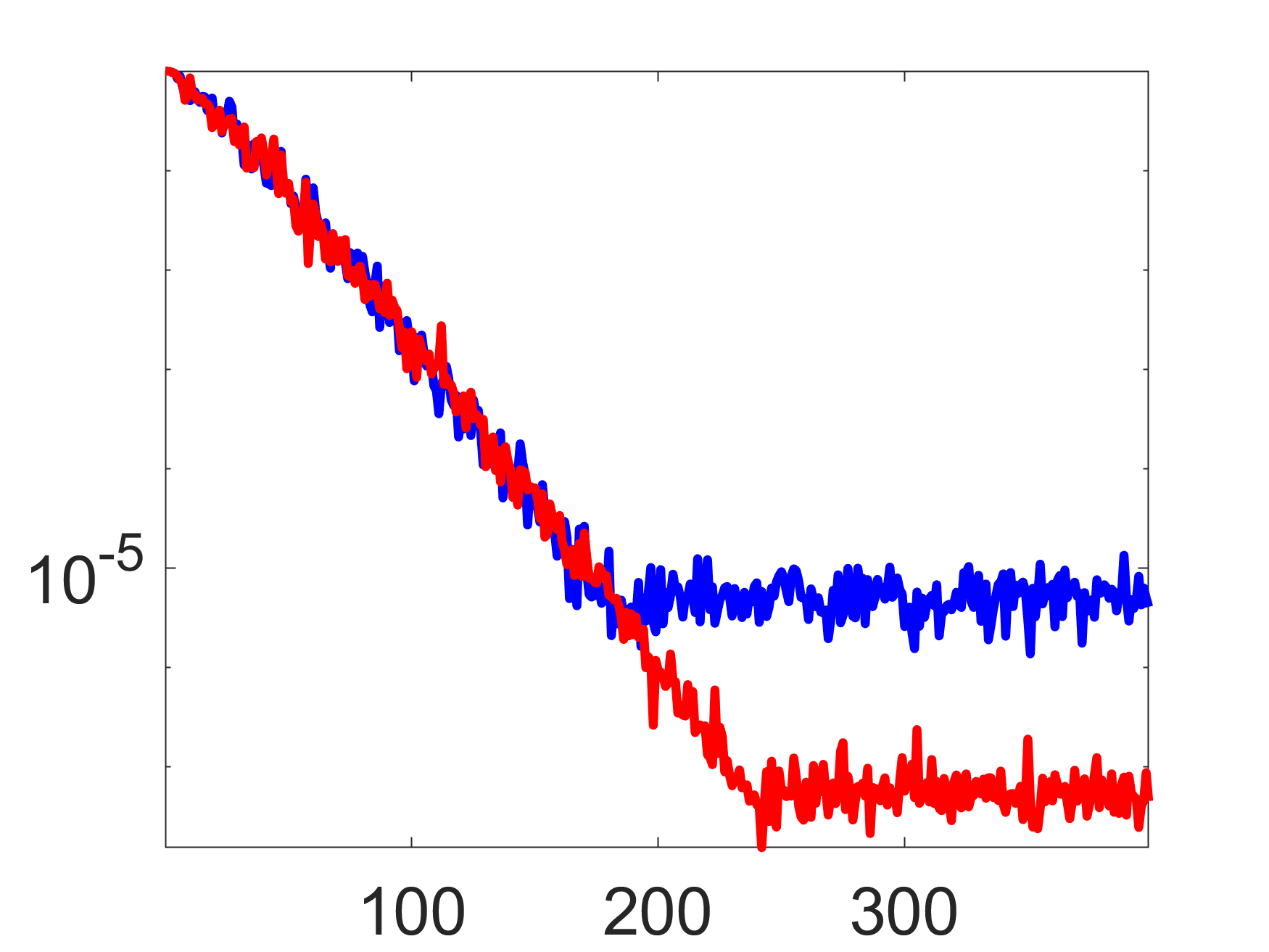}
			\\[-0.1 cm] { $\tau$}
		\end{tabular}
		&
		\begin{tabular}{c}
			\vspace{.25cm}
			\normalsize{\rotatebox{90}{ $
					\dfrac{\norm{\nabla f(K) - \widebar{\nabla}f(K)}_F}{\norm{\nabla f(K)}_F}
					$}}
		\end{tabular}
		&
		\begin{tabular}{c}
			\includegraphics[width=0.22\textwidth]{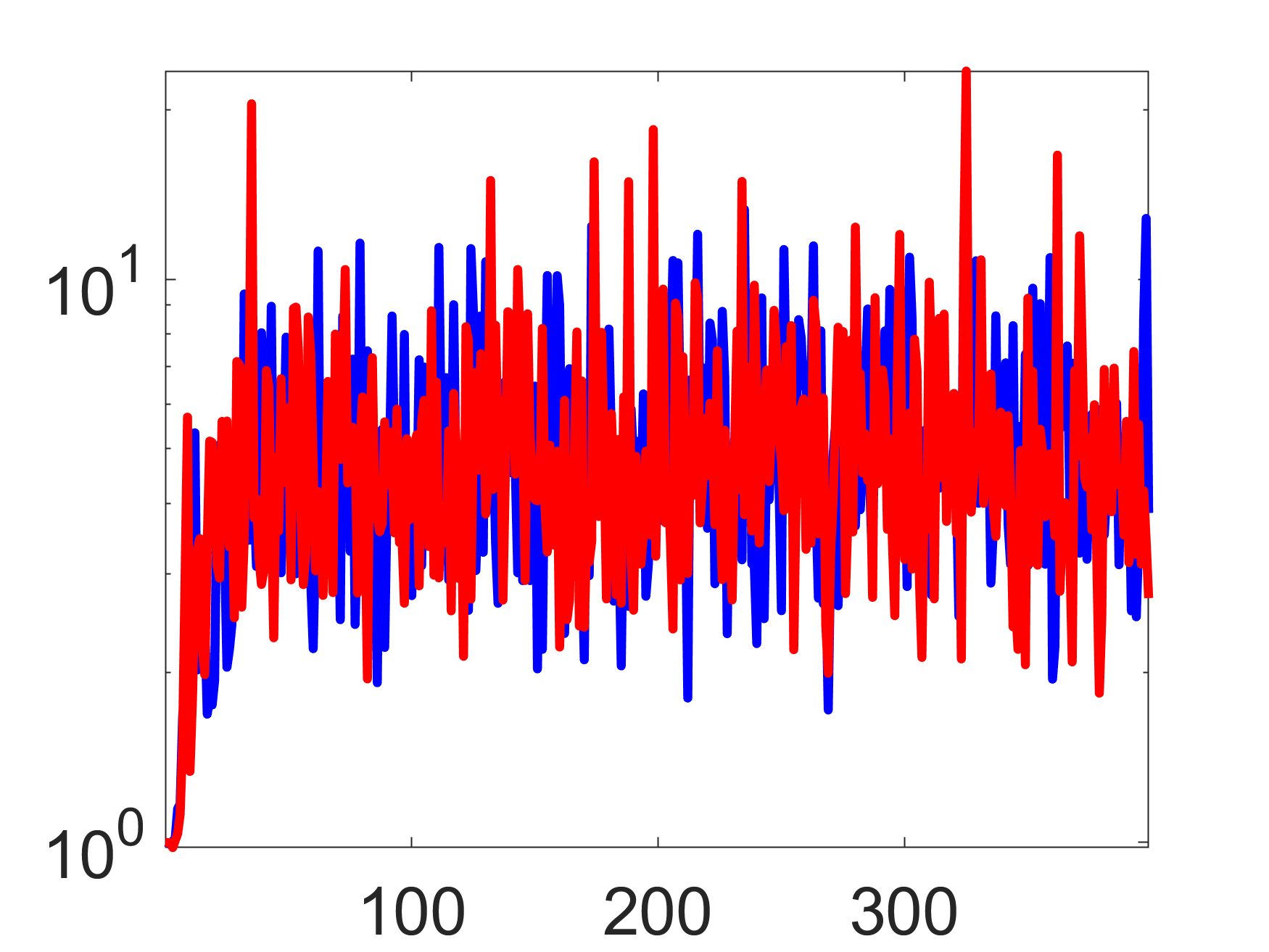}
			\\[-0.1 cm]  {$\tau$}
		\end{tabular}
		&
		\begin{tabular}{c}
			\vspace{.25cm}
			\normalsize{\rotatebox{90}{\small $
					\dfrac{f(K^k)-f(K^\star)}{f(K^0)-f(K^\star)}
					$}}
		\end{tabular}
		&
		\begin{tabular}{c}
			\includegraphics[width=0.22\textwidth]{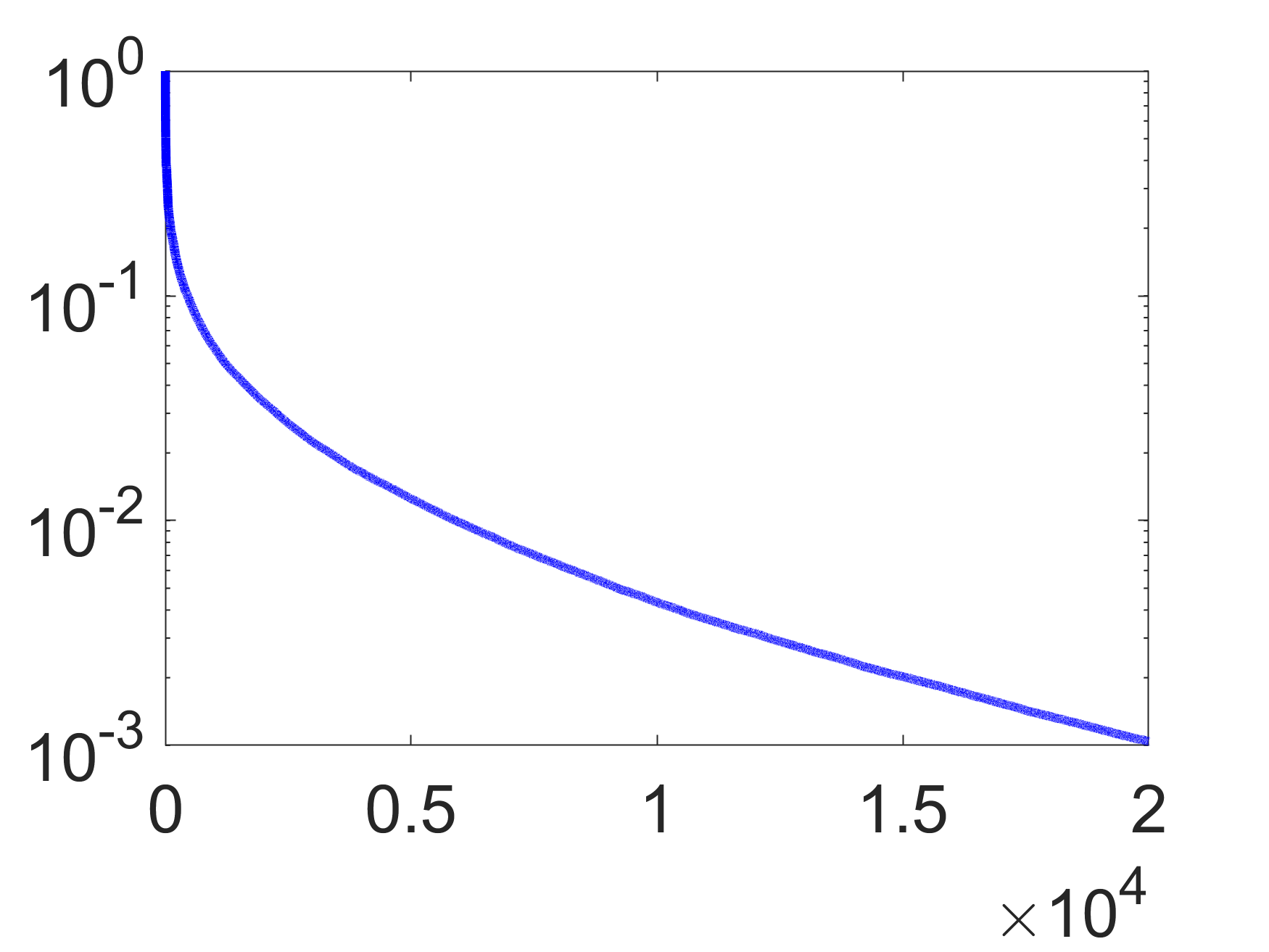}
			\\[-0.1 cm]  {$k$}
		\end{tabular}
		
	\end{tabular}
	\vspace{-0.5cm}
	\caption{(a) Bias in gradient estimation and (b) total error in gradient estimation as functions of the simulation time $\tau$. The blue and red curves correspond to two values of the smoothing parameter $r=10^{-4}$ and $r=10^{-5}$, respectively. (c) Convergence curve of the random search method~\eqref{eq.RS}.}
\end{figure*}

	\vspace*{-2ex}
\section{Concluding remarks}
\label{sec.conclusions}

We prove exponential/linear convergence of gradient flow/descent algorithms for solving the continuous-time LQR problem based on a nonconvex formulation that directly searches for the controller. A salient feature of our analysis is that we relate the gradient-flow dynamics associated with this nonconvex formulation to that of a convex reparameterization. This allows us to deduce convergence of the nonconvex approach from its convex counterpart. We also establish a bound on the sample complexity of the random search method for solving the continuous-time LQR problem that does not require the knowledge of system parameters. We have recently proved similar result for the discrete-time LQR problem~\cite{mohsoljovLCSS21}.

Our ongoing research directions include: (i) providing theoretical guarantees for the convergence of gradient-based methods for sparsity-promoting and structured control synthesis; and (ii) extension to nonlinear systems via successive linearization techniques.

\vspace*{-2ex}
\appendix

\vspace*{-2ex}
\subsection{Lack of convexity of  function $f$}
\label{app.non-convex-SK}

The function $f$ is nonconvex in general because its effective domain, namely, the set of stabilizing feedback gains $\cSK$ can be nonconvex. In particular, for $A=0$ and $B=-I$, the closed-loop $A$-matrix is given by $A-BK=K$. Now, let 
\begin{align}\label{eq.counterExample}
K_1 \;=\; \tbt{-1}{2-2\epsilon}{0}{-1},\qquad K_2 \;=\; \tbt{-1}{0}{2-2\epsilon}{-1},\qquad K_3 \;=\; \dfrac{K_1+K_2}{2} = \tbt{-1}{1-\epsilon}{1-\epsilon}{-1}
\end{align}
where  $0\le\epsilon\ll1$. It is straightforward to show that for $\epsilon>0$, the entire line-segment $\overline{K_1 K_2}$ lies in $\cSK$. However, if we let $\epsilon\rightarrow0$,  while the endpoints $K_{1}$ and $K_{2}$ converge to stabilizing gains, the middle point $K_3$ converges to the boundary of $\cSK$. Thus, $f(K_1)$ and $f(K_2)$ are bounded whereas $f(K_3)\rightarrow\infty$. This implies the existence of a point on the line-segment $\overline{K_1 K_2}$ for some $\epsilon\ll 1$ for which the function $f$ has negative curvature. For $\epsilon=0.1$, Fig.~\ref{fig.nonConvexExample} illustrates the value of the LQR objective function $f(K(\gamma))$ associated with the above example and the problem parameters $Q=R=\Omega=I$, where $K(\gamma)\DefinedAs \gamma {K_1}+(1-\gamma){K_2}$ is the line-segment $\overline{K_1 K_2}$.   We observe the negative curvature of $f$ around the middle point  $K_3$. Alternatively, we can  verify the negative curvature using the second-order term $\inner{J}{\nabla^2f(K);J}$ in the Taylor series expansion of $f(K+J)$ around $K$ given in Appendix~\ref{app.thirdOrder}. For the above example, letting  $J=(K_1-K_2)/\norm{K_1-K_2}$ yields the negative value $\inner{J}{\nabla^2f(K_3);J} = -135.27$.

\begin{figure}
	\centering
	\begin{tabular}{r@{\hspace{-0.5 cm}}l}
		\begin{tabular}{c}
			\vspace{.25cm}
			\normalsize{\rotatebox{90}{\small $
					f(K({\gamma}))
					$}}
		\end{tabular}	
	&\begin{tabular}{c}		
		\includegraphics[width=.25\textwidth]{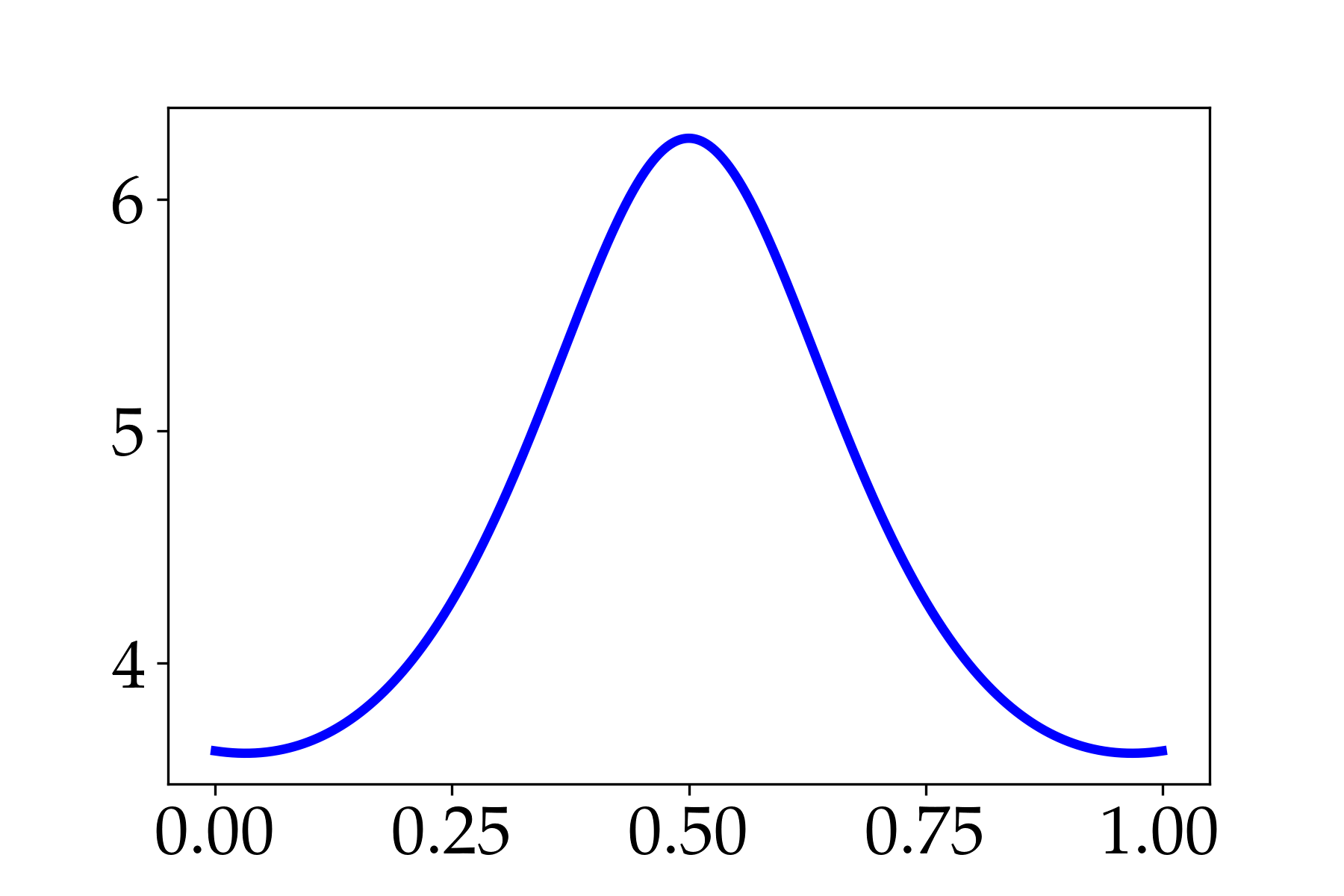}
		\\[-0.1 cm]
		{\small $\gamma$}
		
	\end{tabular}
\end{tabular}
	\caption{The LQR objective function $f(K(\gamma))$, where $K(\gamma)\DefinedAs \gamma {K_1}+(1-\gamma){K_2}$ is the line-segment between $K_1$ and $K_2$ in~\eqref{eq.counterExample} with $\epsilon=0.1$.}
	\label{fig.nonConvexExample}
\end{figure}
	
	\vspace*{-2ex}
\subsection{Invertibility of the linear map $\cA$}
\label{app.K0}

The invertibility of the map $\cA$ is equivalent to the matrices $A$ and $-A^T$ not having any common eigenvalues. If $\cA$ is non-invertible, we can use $K^0\in\cSK$ to introduce the change of variables $\hat{K}\DefinedAs K-K^0$ and $\hat{Y} \DefinedAs \hat{K} X$ and obtain
	$
	f(K) 
	= 
	\hat{h}(X,\hat{Y}) 
	\DefinedAs 
	\trace \, ( Q^0 X + X^{-1}\hat{Y}^TR\,\hat{Y} + 2 \hat{Y}^TR K^0 )
	$
for all $K\in\cSK$, where $Q^0\DefinedAs Q + (K^0)^TR\,K^0$. 
Moreover, $X$ and $\hat{Y}$ satisfy the affine relation
$
\cA_0(X) - \cB(\hat{Y}) + \Omega  = 0,
$
where $\cA_0(X)\DefinedAs (A - B K^0)X+X(A - B K^0)^T$. Since the matrix $A - B K^0$ is Hurwitz, the map $\cA_0$ is invertible. This allows us to write $X$ as an affine function of $\hat{Y}$, 
	$
	X(\hat{Y}) 
 	= 
	\cA_0^{-1}(\cB(\hat{Y}) - \Omega).
	$
Since the function $\hat{h}(\hat{Y})\DefinedAs\hat{h}(X(\hat{Y}),\hat{Y})$ has a similar form to $h(Y)$ except for the linear term $2 \, \trace \, (\hat{Y}^TR\,K^0)$, the smoothness and strong convexity of $h(Y)$ established in Proposition~\ref{prop.LMu} carry over to the function $\hat{h}(Y)$.

\vspace*{-2ex}
\subsection{Proof of Proposition~\ref{prop.LMu} }	
\label{app.convex-char}
The second-order term  in the Taylor series expansion of $h(Y+\tY)$ around $Y$ is given by~\cite[Lemma~2]{zarmohdhigeojovTAC20}
\begin{align}\label{eq.hessianY}
\inner {\tY}{\nabla^2 h(Y; \tY)}    
\; = \;
2\,\norm{R^{\frac{1}{2}} (\tY \,-\, K \tX ) X^{-\frac{1}{2}}}_F^2
\end{align}
where $\tX$ is the unique solution to 
	$\cA(\tX) = \cB(\tY).$ 
We show that this term is upper and lower bounded by $L\norm{\tY}_F^2$ and $\mu\norm{\tY}_F^2$, where $L$ and $\mu$ are given by~\eqref{eq.L} and \eqref{eq.muStrong}, respectively. The proof for the upper bound is borrowed from~\cite[Lemma~1]{zarmohdhigeojovTAC20}; we include it for completeness. We repeatedly use the bounds on the variables presented in Lemma~\ref{lem.bounds}; see Appendix~\ref{app.bounds}. 

\subsubsection*{Smoothness} For any $Y\in\cSY(a)$ and $\tilde{Y}$ with $\norm{\tilde{Y}}_F=1$, 
\begin{align*}
\inner {\tY}{\nabla^2h(Y; \tY)}    
&\;=\;
2 \norm{R^{\tfrac{1}{2}}(\tilde{Y} \,-\,K \tX)  X^{-\tfrac{1}{2}}}_F^2
\;\le\;
2 \norm{R}_2 \norm{X^{-1}}_2\norm{\tilde{Y} \,-\,K \cA^{-1} \cB(\tilde{Y})}_F^2
\\[0.cm]
\quad\quad\;
&\le\;
\dfrac{2 \,\norm{R}_2}{\lambda_{\min}(X)}  \left( \norm{\tY}_F \,+\, \norm{K}_2 \norm{\cA^{-1} \cB}_2 \norm{\tY}_F \right)^2 
\;\le\;
\dfrac{2 a \norm{R}_2}{\nu} \left( 1 \,+\, \dfrac{a \norm{\cA^{-1} \cB}_2}{\sqrt{\nu \lambda_{\min}(R)}} \right)^2 
\AsDefined \; 
L.
\end{align*}
Here, the first and second inequalities are obtained from the definition of the $2$-norm in conjunction with  the triangle inequality, and the third inequality follows from~\eqref{eq.uY} and~\eqref{eq.lX}. 
This completes the proof of smoothness.

\subsubsection*{Strong convexity} 
Using the positive definiteness of matrices $R$ and $X$, the second-order term~\eqref{eq.hessianY} can be lower bounded~by
\begin{align}\label{eq.quadformlower}
\inner {\tY}{\nabla^2 h(Y; \tY)}    
\; \ge \; 
\dfrac{2\lambda_{\min}(R) \norm{H}_F^2}{ \norm{X}_2}
\end{align}
where $H \DefinedAs\tY -K \tX$. Next, we show that
\begin{align}
\label{eq.lowerHtoTx}
\dfrac{\norm{H}_F}{\norm{\tX}_F}
\;\ge\;
\dfrac{\lambda_{\min}(\Omega)\lambda_{\min}(\Omega)}{a \,\norm{\cB}_2}.
\end{align}
We substitute $H + K \tX$ for $\tY$ in $\cA(\tX) = \cB(\tY)$ to obtain
\begin{align}
\label{eq.Wdef}
\Gamma
\;=\;
\cB(H)
\end{align}  
where 
$
\Gamma
\DefinedAs
\cA_K( \tX).
$
The closed-loop stability implies
$
\tX
=
\cA_K^{-1}(\Gamma)
$
and from Eq.~\eqref{eq.Wdef} we have
\begin{align}
\label{eq.Hlower}
\norm{H}_F
&
\;\ge\;
\dfrac{\norm{\Gamma}_F}{\norm{\cB}_2}.
\end{align}
This allows us to use Lemma~\ref{lem.ucAinv}, presented in Appendix~\ref{app.invLyapBound}, to write
	$
	a \norm{\Gamma}_F
	\ge 
	\lambda_{\min}(\Omega)\lambda_{\min}(Q) \norm{\tX}_F. 
	$
This inequality in conjunction with~\eqref{eq.Hlower} yield~\eqref{eq.lowerHtoTx}. Next, we derive an upper bound on $\norm{\tilde{Y}}_F$,
\begin{align}\label{eq.upperTY}
\norm{\tY}_F
\;=\;
\norm{H\,+\,K \tX}_F
\;\le\;
\norm{H}_F \,+\, \norm{K}_F \norm{\tX}_F
\;\le\;
\norm{H}_F \left( 1 \,+\,a^2 \eta \right)
\end{align}
where $\eta$ is given by~\eqref{eq.nuprime} and the second inequality follows from~\eqref{eq.uK} and~\eqref{eq.lowerHtoTx}.
Finally, inequalities~\eqref{eq.quadformlower} and~\eqref{eq.upperTY} yield
	\vspace*{-1ex}
\begin{align}
\dfrac{\inner {\tY}{\nabla^2 f(Y; \tY)} }{\norm{\tY}_F^2}
\;\ge\;
\dfrac{2 \lambda_{\min}(R) \norm{H}_F^2}{\norm{X}_2 \norm{\tY}_F^2}
\;\ge\;
\dfrac{2 \lambda_{\min}(R)}{\norm{X}_2(1 \, + \, a^2 \eta)^2}
\;\ge\;
\dfrac{2 \lambda_{\min}(R)\lambda_{\min}(Q)}{a (1\,+\,a^2 \eta)^2}
\; \AsDefined \;
\mu
\end{align}
where the last inequality follows from~\eqref{eq.uX}. 

	\vspace*{-2ex}
\subsection{Proofs for Section~\ref{sec.non-convex-char} }	
\label{app.non-convex-char}

\subsubsection*{Proof of Lemma~\ref{lem.comparison}}
The gradients are given by $\nabla f(K) = E X$ and $\nabla h(Y) = E + 2 B^T(P-W)$, where  $E\DefinedAs 2(R K - B^T P)$, $P$ is determined by~\eqref{eq.P}, and $W$ is the solution to~\eqref{eq.W}. Subtracting~\eqref{eq.W} from~\eqref{eq.lyapKP} yields
$
A^T (P-W) + (P-W)A 
	= - \tfrac{1}{2}\left(K^TE +  E^TK\right),
$
which in turn leads to
	\begin{align*}
	\norm{P-W}_F 
	\;\le\;
	\norm{\cA^{-1}}_2 \norm{K}_F \norm{E}_F
	\;\le\;
	\dfrac{ a \norm{\cA^{-1}}_2 \norm{E}_F}{\sqrt{\nu \lambda_{\min}(R)}}
	\end{align*}
where the second inequality follows from~\eqref{eq.uK} in Appendix~\ref{app.bounds}.
Thus, by applying the triangle inequality to $\nabla h(Y)$, we obtain
	\begin{align*}
\dfrac{\norm{\nabla h(Y)}_F }{ \norm{E}_F}
	\;\le\; 
	1 
	\,+\, 
	\dfrac{2 a \norm{\cA^{-1}}_2 \norm{B}_2}{\sqrt{\nu \lambda_{\min}(R)}}.
	\end{align*}
Moreover, using the lower bound~\eqref{eq.lX} on  $\lambda_{\min}(X)$, we have
	$
	\norm{\nabla f(K)}_F 
	= 
	\norm{E X}_F
	\ge 
	(\nu / a ) \norm{E}_F.
	$
Combining the last two inequalities completes the proof.
	
\subsubsection*{Proof of Lemma~\ref{lem.errorRelation}}
For any pair of stabilizing feedback gains $K$ and $\hK\DefinedAs K+\tK$, we have~\cite[Eq.~(2.10)]{toi85},
	$
	f(\hK) - f(K)
 	=
	\trace \big( \tK^T \big(R (K+\hK) - 2 B^T\hP\big) X \big),
	$
where $X=X(K)$ and $\hP=P(\hat{K})$ are given by~\eqref{eq.X} and~\eqref{eq.P}, respectively. 
Letting $\hat{K}=K^\star$ in this equation and using the optimality condition $B^T \hP =  R \hK$ completes the proof. 

\subsubsection*{Proof of Lemma~\ref{lem.Lf}}
We show that the second-order term  $\inner {\tK}{\nabla^2 f(K; \tK)}$ in the Taylor series expansion of $f(K+\tK)$ around $K$ is upper bounded by $L_f \norm{\tK}_F^2$ for all $K\in\cSK(a)$.
From~\cite[Eq.~(2.3)]{toimak87}, it follows 
\begin{align*}
\inner {\tK}{\nabla^2 f(K; \tK)}    
\;=\;
2\,\trace \, (\tK^T R \, \tK X - 2\tK^T B^T\tP X)
\end{align*}
where
$
\tP = ( \cA_K^*)^{-1}(C)
$
and  $C \DefinedAs \tK^T(B^TP - R K) + (B^TP - R K)^T\tK.$ Here, $X=X(K)$ and $P=P(K)$ are given by~\eqref{eq.X} and~\eqref{eq.P} respectively.
Thus, using basic properties of the matrix trace and the triangle inequality, we have 
\begin{align}\label{eq.tPtemp0}
\dfrac{\inner {\tK}{\nabla^2 f(K; \tK)}}{\norm{\tK}_F^2}
\,\le\,
2 \norm{X}_2\left(\norm{R}_2  +  \dfrac{2 \norm{B}_2 \norm{\tP}_F}{\norm{\tK}_F}\right).
\end{align}
Now, we use Lemma~\ref{lem.ucAinv} to upper bound the norm of $\tP$,
	$
	\norm{\tP}_F 
 	\le
	{a} \norm{C}_F / ( \lambda_{\min}(\Omega) \lambda_{\min}(Q) ).
	$
Moreover, from the definition of $C$, the triangle inequality, and the submultiplicative property of the $2$-norm, we have
	$
	\norm{C}_F
 	\le
	2\norm{\tK}_F(\norm{B}_2\norm{P}_2+\norm{R}_2\norm{K}_2)
	$
Combining the last two inequalities gives 
\begin{align*}
\dfrac{\norm{\tP}_F}{\norm{\tK}_F}
\;\le\;
\dfrac{2a}{\lambda_{\min}(\Omega)\lambda_{\min}(Q)} \, (\norm{B}_2\norm{P}_2+\norm{R}_2\norm{K}_2)
\end{align*}
which in conjunction with~\eqref{eq.tPtemp0} lead to
\begin{align*}
\dfrac{\inner {\tK}{\nabla^2 f(K; \tK)}}{\norm{\tK}_F^2}
\;\le\;
2\norm{X}_2\left(\norm{R}_2 \,+\,
\dfrac{4a}{\lambda_{\min}(\Omega)\lambda_{\min}(Q)} \, (\norm{B}^2_2\norm{P}_2+ \norm{B}_2\norm{R}_2\norm{K}_2)\right).
\end{align*}
Finally, we use the bounds provided in Appendix~\ref{app.bounds}  to obtain
\begin{align*}
\dfrac{\inner {\tK}{\nabla^2 f(K; \tK)}}{\norm{\tK}_F^2}
\;\le\; 
\dfrac{2a\norm{R}_2}{\lambda_{\min}(Q)} 
\;+\;
\dfrac{8 a^3 }{\lambda_{\min}^2(Q)\lambda_{\min}(\Omega)} 
\left(
\dfrac{\norm{B}_2^2}{\lambda_{\min}(\Omega)} 
 	\, + \,
\dfrac{\norm{B}_2\norm{R}_2}{\sqrt{\nu\lambda_{\min}(R)}}
\right)
\end{align*}
which completes the proof.

	\vspace*{-3ex}
\subsection{Proofs for Section~\ref{subsec.localBoundedness}}
\label{app.localBoundedness}
We first present two technical lemmas.

\begin{mylem}\label{lem.KYP_SG}
	Let $Z \succ 0$ and let the Hurwitz matrix $F$ satisfy
	\begin{align}\label{eq.SGtemp1}
	\tbt{\delta^2 I + F^T Z + Z F}{Z}{Z}{-I}
	\;\prec\;
	0.
	\end{align}
	Then $F + \delta \Delta$ is Hurwitz for all $\Delta$ with $\norm{\Delta}_2\le 1$. 
\end{mylem}
\begin{proof}
	The matrix $F+\delta\Delta$ is Hurwitz if and only if the linear map from $w$ to $x$ with the state-space realization
	$
	\{
	\dot{x} 
	= 
	F x + w + u,
	$
	$
	z
	=
	\delta x
	\}
	$
	in feedback with $ u = \Delta z$ is input-output stable. From the small-gain theorem~\cite[Theorem 8.2]{dulpag00}, this system is stable for all $\Delta$ in the unit ball if and only if the induced gain of the map $u \mapsto z$ with the state-space realization 
	$
	\{
	\dot{x} 
	=
	F x + u,
	$
	$
	z
	=
	\delta x
	\}
	$
is smaller than one. The KYP Lemma~\cite[Lemma 7.4]{dulpag00} implies that this norm condition is equivalent to~\eqref{eq.SGtemp1}. 
\end{proof}

\begin{mylem}\label{lem.localDisk}
	Let the matrices $F$, $X\succ 0$, and $\Omega\succ 0$ satisfy
	\begin{align}\label{eq.SGtemp2}
	F X \,+\, X F^T \, + \, \Omega 
	\;=\;
	0.
	\end{align} 
	Then the matrix $F + \Delta$ is Hurwitz for all $\Delta$ that satisfy
	$
	\norm{\Delta}_2
	< 
	{\lambda_{\min}(\Omega)}/( 2 \norm{X}_2 ).
	$
\end{mylem}
\begin{proof}
	From~\eqref{eq.SGtemp2}, we obtain that $F$ is Hurwitz and
	$
	F \hat{X} + \hat{X} F^T + I 
	\preceq
	0
	$
where $\hat{X} \DefinedAs X/\lambda_{\min}(\Omega)$.  Multiplication of this inequality
	from both sides by $\hat{X}^{-1}$ and division by $2$ yields
	$
	Z F + F^T Z + 2\,Z^2 
	\preceq 
	0
	$
	where $Z\DefinedAs (2\hat{X})^{-1}$. For any positive scalar 
	$
	\delta
	<
	\lambda_{\min}(Z) 
	=
	{\lambda_{\min}(\Omega)}/( 2 \norm{X}_2 )
	$
	the last matricial inequality implies
	$
	\delta^2 I 
	+	
	Z F 
	+
	F^T Z 
	+ 
	Z^2 
	\prec
	0.
	$
	The result  follows from Lemma~\ref{lem.KYP_SG} by observing that the last inequality is equivalent to~\eqref{eq.SGtemp1} via the use of Schur complement.
\end{proof}

\subsubsection*{Proof of Proposition~\ref{prop.localdisk}}
For any feedback gain $\hat{K}$ such that $\norm{\hat{K}-K}_2 <\zeta$, the closed-loop matrix $A-B\hat{K}$ satisfies
	$
	\norm{A-B\hat{K} - (A-BK)}_2
 	\le
	\norm{K-\hat{K}}_2 \norm{B}_2
	<
	\zeta \norm{B}_2.
	$
This bound on the distance between the closed-loop matrices $A-BK$ and $A-B\hat{K}$  allows us to apply Lemma~\ref{lem.localDisk} with $F\DefinedAs A-BK$ and $X\DefinedAs X(K)$ to complete the proof. 

We next present a technical lemma. 
\begin{mylem}\label{lem.perturbX}
	For any $K\in\cSK$ and $\hat{K}\in\R^{m\times n}$ such that 
	$\norm{\hat{K}-K}_2 
	< 
	\delta
	$, 
with
	\begin{align*}
	\delta
	\, \DefinedAs \,
	\dfrac{1}{4\,\norm{B}_F}\,
	\min\left\{
	\dfrac{\lambda_{\min}(\Omega)}{\trace \, (X(K))},
	\,
	\dfrac{\lambda_{\min}(Q)}{\trace \, (P(K))}
	\right\}
	\end{align*}	
	the feedback gain matrix $\hat{K}\in\cSK$, and 
	\begin{subequations}
		\begin{align}
		\norm{X(\hat{K})\,- X(K)}_F 
		&\;\le\;
		\epsilon_1 \norm{\hat{K} - K}_2
		\label{eq.perX}
		\\[0.cm]
		\norm{P(\hat{K})\,- P(K)}_F 
		&\;\le\;
		\epsilon_2 \norm{\hat{K} - K}_2
		\label{eq.perP}
		\\[0.cm]
		\norm{\nabla f(\hat{K})\,- \nabla f(K)}_F 
		&\;\le\;
		\epsilon_3 \norm{\hat{K} - K}_2
		\label{eq.perNab}
		\\[0.cm]
		|f(\hat{K})\,- f(K)|~~
		&\;\le\;
		\epsilon_4 \norm{\hat{K} - K}_2
		\label{eq.perf}
		\end{align}
	\end{subequations}
	where $X(K)$ and $P(K)$ are given by~\eqref{eq.X} and~\eqref{eq.P}, respectively. Furthermore, the parameters $\epsilon_i$ which only depend on $K$ and problem data are given by
		$
	\epsilon_1
	\DefinedAs
	{\norm{X(K)}_2}/{\delta},
	$
	$
	\epsilon_2
	\DefinedAs
	{2\,\trace(P)}
	(2\,\norm{P}_2\norm{B}_F
	+
	(\delta + 2\norm{K}_2)\norm{R}_F )/{\lambda_{\min}(Q)},
	$
	$
	\epsilon_4
	\DefinedAs
	\epsilon_2
	\norm{\Omega}_F,
	$
	$
	\epsilon_3
	\DefinedAs
	2  
	( \epsilon_1\norm{K}_2 + 2\norm{X(K)}_2 ) \norm{R}_F
	+
	2 \epsilon_1 
	( \norm{P(K)}_2 + 2\epsilon_2\norm{X(K)}_2 ) \norm{B}_F.
	$
\end{mylem}
\begin{proof}
	Note that $\delta\le \zeta$, where $\zeta$ is given in~Proposition~\ref{prop.localdisk}. Thus, we can use Proposition~\ref{prop.localdisk} to show that $\hat{K}\in\cSK$. We next prove~\eqref{eq.perX}. For $K$ and $\hat{K}\in\cSK$, we can represent $X = X(K)$ and $\hat{X} = X(\hat{K})$ as the positive definite solutions to
	\begin{subequations}
		\begin{align}
		\label{eq.lyapKtemp1}
		(A \,-\, B K) X \;+\; X (A \,-\, B K)^T +\; \Omega  
		&\; = \;
		0
		\\[-0.1cm]
		(A \,-\, B \hat{K}) \hat{X} \;+\; \hat{X} (A \,-\, B \hat{K})^T +\; \Omega  
		&\; = \;
		0.
		\label{eq.lyapKtemp2}
		\end{align}
	\end{subequations}
	Subtracting~\eqref{eq.lyapKtemp1} from~\eqref{eq.lyapKtemp2} and rearranging terms yield
	\begin{align*}
	(A \,-\, B K) \tX \;+\; \tX (A \,-\, B K)^T
	\;=\;
	B \tK \hat{X} +  \hat{X} ( B\tK)^T
	\end{align*}
	where $\tX \DefinedAs \hat{X}-X$ and $\tK \DefinedAs \hat{K} - K$. Now, we use Lemma~\ref{lem.normLyap}, presented in Appendix~\ref{app.invLyapBound}, with $F \DefinedAs A-BK$  to upper bound the norm of $\tX = \cF(-B \tK \hat{X} -  \hat{X} ( B\tK)^T)$,  where the linear map $\cF$ is defined in~\eqref{eq.cF}, as follows
	\begin{align}
	{\norm{\tX}_F}
	&\;\le\;
	\norm{\cF}_2  \norm{B \tK \hat{X} +  \hat{X} ( B\tK)^T}_F
	\;\le\;
	\dfrac{\trace(X)}{\lambda_{\min}(\Omega)}  \norm{B \tK \hat{X} +  \hat{X} ( B\tK)^T}_F
	\nonumber\\[0.cm]
	&\;\le\;
	\dfrac{2 \, \trace(X) \norm{B}_F\norm{ \tK}_2}{\lambda_{\min}(\Omega)}    \left(\norm{X}_2+\norm{\tilde{X}}_2\right) 
	\;\le\;
	\dfrac{2 \, \trace(X) \norm{B}_F\norm{ \tK}_2   \norm{X}_2}{\lambda_{\min}(\Omega)}   \, + \, \dfrac{1}{2} \, \norm{\tilde{X}}_F.
	\label{eq.perturbTemp1}
	\end{align}
	Here, the second inequality follows from Lemma~\ref{lem.normLyap}, the third inequality follows from a combination of the sub-multiplicative property of the Frobenius norm and the triangle inequality, and the  last inequality follows from $\norm{\tK}\le \delta$ and $\norm{\tX}_2\le \norm{\tX}_F$. 
	Rearranging the terms in~\eqref{eq.perturbTemp1} completes the proof of~\eqref{eq.perX}.  
	
	We next prove~\eqref{eq.perP}. Similar to the proof of~\eqref{eq.perX}, subtracting the Lyapunov equation~\eqref{eq.lyapKP} from that of $\hat{P}=P(\hat{K})$ yields
	$
	(A - B K)^T \tP + \tP (A - B K)
	=
	W
	$
where $\tP \DefinedAs \hat{P}-P$ and
	$
	W
	\DefinedAs
	(B\tK)^T \hat{P} +  \hat{P} B\tK
	-
	\tK^TR\,\tK 
	-
	\tK^TR\,K - K^TR\,\tK.
	$
	This allows us to use Lemma~\ref{lem.normLyap}, presented in Appendix~\ref{app.invLyapBound}, with $F \DefinedAs (A-BK)^T$  to upper bound the norm of $\tP = \cF(-W)$,  where the linear map $\cF$ is defined in~\eqref{eq.cF}, as follows
	\begin{align*}
	{\norm{\tP}_F}
	\;\le\;
	\norm{\cF}_2  \norm{W}_F
	\;\le\; 
	\dfrac{\trace(\cF(Q+K^TRK))}{\lambda_{\min}(Q+K^TRK)} \, \norm{W}_F
	\;=\;
	\dfrac{\trace(P)}{\lambda_{\min}(Q+K^TRK)}  \, \norm{W}_F
	\;\le\;
	\dfrac{\trace(P)}{\lambda_{\min}(Q)}  \, \norm{W}_F.
	\end{align*}
	Here, the second inequality follows from Lemma~\ref{lem.normLyap}. This inequality in conjunction with applying the triangle inequality to the definition of $W$ yield
	\begin{multline*}
	\norm{\tP}_F
	\;\le\;
	\dfrac{\trace(P)}{\lambda_{\min}(Q)}\left(
	\norm{(B\tK)^T \tP +  \tP\,  B\tK}F
	\;+\;
	\norm{(B\tK)^T P +  P\,  B\tK
		-
		\tK^TR\,\tK -\, \tK^TR\,K -\, K^TR\,\tK}_F\right)
	\,\le
	\\[0.cm]
	\dfrac{\norm{\tP}_F}{2}
	\;+\; \dfrac{\trace(P)}{\lambda_{\min}(Q)}
	\Big(
	2\norm{P}_2\norm{B}_F
	\,+\,
	(\delta+2\norm{K}_2)\norm{R}_F
	\Big)
	\norm{\tK}_2.
	\end{multline*}
	The second inequality is obtained by bounding the two terms on the left-hand side using basic properties of norm, where, for the first term, 
	$
	\norm{\tK}_2 \le \delta \le {\lambda_{\min}(Q)}/(4\norm{B}_F \, \trace \, (P(K)))
	$
and, for the second term, $\norm{\tK}_2\le\delta$.
	Rearranging the terms in above completes the proof of~\eqref{eq.perP}. 	
	
	We next prove~\eqref{eq.perNab}. It is straightforward to show that the gradient~\eqref{eq.nablaK} satisfies
	\begin{align*}
	\tilde{\nabla}
	\;\DefinedAs\; 
	\nabla f(\hat{K}) - \nabla f(K)
	\;=\;
	2 R ( \tK X + K \tX + \tK \tX )  
	\,-\, 	
	2 B^T (\tP X + P \tX + \tP \tX )
	\end{align*}
where $P\DefinedAs P(K)$ and $\tP\DefinedAs \hat{P} - P$. The triangle inequality in conjunction with $\norm{\tX}_F\le\epsilon_1\norm{\tK}_2$, $\norm{\tP}_F\le\epsilon_2\norm{\tK}_2$, and $\norm{\tK}_2<\delta$, yield
	$
	{\norm{\tilde{\nabla}}_F}/{\norm{\tK}_2}
	\le 
	2  
	\norm{R}_F \, (\norm{X}_2 + \epsilon_1(\norm{K}_2+\delta) ) 
	+
	2  
	\norm{B}_F \, (\epsilon_2\norm{X}_2 + \epsilon_1(\norm{P}_2+\epsilon_2\delta)).
	$
	Rearranging terms completes the proof of~\eqref{eq.perNab}.

	Finally, we prove~\eqref{eq.perf}. Using the definitions of  $f(K)$ in~\eqref{eq.f} and $P(K)$ in~\eqref{eq.P}, it is easy to verify that 
	$
	f(K) 
	=
	\trace \, (P(K) \Omega).
	$
	Application of the Cauchy-Schwartz inequality yields
	$
	| f(\hK) - f(K) |
	=
	| \trace \, (\tP \Omega )|
	\le
	\norm{\tP}_F\norm{\Omega}_F,
	$
which completes the proof. 
\end{proof}

\subsubsection*{Proof of Lemma~\ref{lem.cSK2a}}
For any $K\in\cSK(a)$, we can use the bounds provided in Appendix~\ref{app.bounds} to show that 
	$ 
	c_1 / a \le \delta
	$
and 
	$
	\epsilon_4
	\le
	c_2 a^2,
	$
where $\delta$ and $\epsilon_4$ are given in Lemma~\ref{lem.perturbX} and each $c_i$ is a positive constant that depends on the problem data. Now, Lemma~\ref{lem.perturbX} implies 
	$
	f(K+r(a)U)-f(K)
	\le
	\epsilon_4 r(a) \norm{U}_2
	\le
	a
	$
where
	$
	r(a)
	\DefinedAs
	\min \{ c_1,1/c_2 \} / ( a  \sqrt{m n} ).
	$	
This inequality together with $f(K) \le a$ complete the proof. 

	\vspace*{-2ex}
\subsection{Proof of Proposition~\ref{prop.finiteTime}}
\label{app.finiteTime}

We first present two technical lemmas.

\begin{mylem}\label{lem.matExpBound}
	Let the matrices $F$, $X\succ 0$, and $\Omega\succ 0$ satisfy
	$
	F X + X F^T + \Omega 
	= 
	0.
	$
	Then, for any $t\ge0$, 
	\[
	\norm{\mre^{Ft}}_2^2
	\;\le\;
	( \norm{X}_2 / \lambda_{\min}(X) ) 
	\,
	\mre^{- ( \lambda_{\min}(\Omega) / \norm{X}_2) \, t}.
	\]
\end{mylem}
\begin{proof}
	The function $V(x)\DefinedAs x^T X x$ is a Lyapunov function for 
	$
	\dot{x} = F^Tx
	$ 
because
	$
	\dot{V}(x)
	=
	{-x^T \Omega x}
	\le
	-c V(x),
	$
where $c\DefinedAs{\lambda_{\min}(\Omega)}/{\norm{X}_2}$. For any initial condition $x_0$, this inequality together with the comparison lemma~\cite[Lemma~3.4]{kha96} yield 
	$
	V(x(t)) 
	\le
	V(x_0)\, \mre^{-c t}.
	$
	Noting that $x^T(t)=x_0^T\mre^{Ft}$, we let $x_0$ be the normalized left singular vector associated with the maximum singular value of $\mre^{Ft}$ to obtain
	\begin{align*}
	\norm{\mre^{F t}}_2^2
	\;=\;
	\norm{x(t)}^2
	\;\le\;
	\dfrac{V(x(t))}{\lambda_{\min}(X)}
	\;\le\;
	\dfrac{V(x_0)}{\lambda_{\min}(X)} \, \mre^{-c t}
	\end{align*} 
which along with $V(x_0)\le \norm {X}_2$ complete the proof.
\end{proof}
	
Lemma~\ref{lem.finiteTruncation} establishes an exponentially decaying upper bound on the difference between $f_{x_0}(K)$ and $f_{x_0,\tau}(K)$ over any sublevel set $\cSK(a)$ of the LQR objective function $f(K)$.
\begin{mylem}\label{lem.finiteTruncation}
	For any $K\in\cSK(a)$ and $v\in\R^n$,
	$
	|f_{v}(K)-f_{v,\tau}(K)|
	\le
	\norm{v}^2 \kappa_1(a) \mre^{-\kappa_2(a)\tau},
	$
	where the positive functions $\kappa_1(a)$ and $\kappa_2(a)$, given by~\eqref{eq.kappa12}, depend on problem data.
\end{mylem}
\begin{proof}
	Since $x(t) = \mre^{(A-BK) t} v$ is the solution to~\eqref{eq.linsys1} with $u=-Kx$ and the initial condition $x(0)=v$, it is easy to verify that 
	$
		f_{v,\tau}(K)
		=
		\trace \left( (Q + K^TR\,K)\, X_{v,\tau}(K) \right)
		$
and
		$
		f_{v}(K)
		=
		\trace \left( (Q \,+\, K^TR\,K)\, X_{v}(K) \right),
		$
		where
		\begin{align*}
		X_{v,\tau}(K)
		\; \DefinedAs \; 
		\int_{0}^{\tau} \mre^{(A - BK) t}\,vv^T \mre^{(A - BK)^T t} \, \mrd t
		\end{align*}
	and $
	X_v
	\DefinedAs
	X_{v,\infty}.
	$
	Using the triangle inequality, we have
	\begin{align}
	\norm{X_{v}(K) - X_{v,\tau}(K)}_F
         \, \le \,
	\norm{v}^2\int_{\tau}^{\infty} \norm{\mre^{(A - BK) t}}_2^2 \; \mrd t.
	\label{eq.finiteTruncTemp1}
	\end{align}
Equation~\eqref{eq.lyapKX} allows us to use Lemma~\ref{lem.matExpBound} with $F\DefinedAs A-BK$, $X\DefinedAs X(K)$ to upper bound $\norm{\mre^{(A-B K)t}}_2$, 
	$
	{\lambda_{\min}(X)} \norm{\mre^{(A-B K)t}}_2^2
	\le
	{\norm{X}_2}\,\mre^{- ( \lambda_{\min}(\Omega) / \norm{X}_2 ) \, t}.
	$
	Integrating this inequality over $[\tau,\infty]$ in conjunction with~\eqref{eq.finiteTruncTemp1} yield
	\begin{align}\label{eq.gapXtau}
	\norm{X_{v}(K) - X_{v,\tau(K)}}_F
	&\;\le\;
	\norm{v}^2\kappa'_1\, \mre^{-\kappa'_2\tau}
	\end{align}
	where
	$
	\kappa'_1
	\DefinedAs
	{\norm{X(K)}_2^2}/( \lambda_{\min}(\Omega)\lambda_{\min}(X(K)) )
	$	
	and
	$
	\kappa'_2
	\DefinedAs
	{\lambda_{\min}(\Omega)} / {\norm{X(K)}_2}.
	$
Furthermore,
	\begin{multline*}
	|f_v(K)-f_{v,\tau}(K)|
	\;=\;
	\left|\trace \left( (Q + K^TR\,K)\, (X_v - X_{v,\tau}) \right)\right|
	\;\le\;
		\\[0.cm]
	(\norm{Q}_F+\norm{R}_2\norm{K}_F^2) \norm{X_v - X_{v,\tau}}_F
	\;\le\;
	\norm{v}^2(\norm{Q}_F+\norm{R}_2\norm{K}_F^2) \, \kappa'_1 \mre^{-\kappa'_2\tau}
	\end{multline*}
	where we use the Cauchy-Schwartz and triangle inequalities for the first inequality and~\eqref{eq.gapXtau} for the second inequality. Combining this result with the bounds on the variables provided in Lemma~\ref{lem.bounds} completes the proof with
	\begin{subequations}\label{eq.kappa12}
		\begin{align}\label{eq.kappa1}
		\kappa_1(a)
		&\;\DefinedAs\;
		\left(\norm{Q}_F+\dfrac{a^2\norm{R}_2}{\nu\lambda_{\min}(R)}\right)
		\dfrac{a^3}{\nu\lambda_{\min}(\Omega)\lambda_{\min}^2(Q)}
		\\[0.cm]\label{eq.kappa2}
		\kappa_2(a)
		&\;\DefinedAs\;
		{\lambda_{\min}(\Omega)\lambda_{\min}(Q)}/{a}
		\end{align}
	\end{subequations}
	where the constant $\nu$ is given by~\eqref{eq.nuu}.
\end{proof}

\subsubsection*{Proof of Proposition~\ref{prop.finiteTime}}
Since $K\in \cSK(a)$ and $r\le r(a)$, Lemma~\ref{lem.cSK2a} implies that $K\pm r U_i \in \cSK(2a)$. Thus, $f_{x_i}(K\pm r U_i)$ is well defined for $i=1,\ldots,N$, and 	
	\begin{align*}
	\widetilde{\nabla} f(K) - \widebar{\nabla} f(K)
	\,=\,
	\dfrac{1}{2 r N}	\Big ( \sum_{i} \big ( f_{x_i}(K+r U_i)  - f_{x_i,\tau} (K+ r U_i) \big ) U_i
	-
	\sum_{i} \big( f_{x_i}(K-r U_i) - f_{x_i,\tau} (K- r U_i) \big ) U_i\Big ).	
	\end{align*}
Furthermore, since $K\pm r U_i \in \cSK(2a)$, we can use triangle inequality and apply Lemma~\ref{lem.finiteTruncation}, $2N$ times, to bound each term individually and obtain
	\begin{align*}
	\norm {\widetilde{\nabla}f(K) -\widebar{\nabla}f(K)}_F
	\;\le\;
	( \sqrt{m n} / r ) \max_i \norm{x_i}^2 \kappa_1(2a)\mre^{-\kappa_2(2a) \tau}
	\end{align*}
where we used $\norm{U_i}_F = \sqrt{mn}$. This completes the proof.

	\vspace*{-2ex}
\subsection{Proof of Proposition~\ref{prop.thirdOrder}}
\label{app.thirdOrder}

We first establish bounds on the smoothness parameter of $\nabla f(K)$. For $J\in\R^{m\times n}$, $v\in\R^n$, and $f_v(K)$ given by~\eqref{eq.f_x}, 
let 
	$
	j_v(K)   
	\DefinedAs 
	\inner {J}{\nabla^2 f_v(K; J)},
	$
denote the second-order term in the Taylor series expansion of $f_v(K+J)$ around $K$. Following similar arguments as in~\cite[Eq.~(2.3)]{toimak87} leads to
	$
	j_v(K)
	= 
	2 \, \trace \, (J^T( R J - 2 B^TD ) X_v),
	$
where $X_v$ and $D$ are the solutions to
\begin{subequations}\label{eq.Gstemp_12}
	\begin{align}\label{eq.GStemp1}
	\cA_K(X_v) 
	&\;=\;
	-vv^T
	\\[0.cm]\label{eq.GStemp2}
	\cA_K^*(D)
	&\;=\;
	J^T(B^TP -R\,K) + (B^TP-R\,K)^TJ
	\end{align}
\end{subequations}
and $P$ is given by~\eqref{eq.P}. The following lemma provides an analytical expression for the gradient $\nabla j_v(K)$.

\begin{mylem}\label{lem.hesnabla}
	For any $v\in\R^n$ and $K\in\cSK$,
	$
	\nabla j_v(K)
	=
	4
	\big(
	B^TW_1X_v 
	+
	(R J - B^TD)W_2 
	+
	(R K -B^TP )W_3
	\big),
	$
where $W_i$ are the solutions to the linear equations
	\begin{subequations}\label{eq.GStemp_456}
		\begin{align}\label{eq.GStemp4}
		\cA_K^*(W_1)
		&\;=\;
		J^TR\,J-J^TB^TD - DBJ
		\\[0.cm]\label{eq.GStemp5}
		\cA_K(W_2)
		&\;=\;
		BJX_v + X_vJ^TB^T
		\\[0.cm]\label{eq.GStemp6}
		\cA_K(W_3)
		&\;=\;
		BJ\,W_2 + W_2J^TB^T.
		\end{align}
	\end{subequations}
\end{mylem}
\begin{proof}
We expand $j_v(K+\epsilon\tK)$ around $K$ and to obtain
	$
	j_v(K+\epsilon\tK)-j_v(K)
	 =
	2 \,\epsilon\,\trace \, ( J^T(R\,J-2B^TD)\tX_v) 
	- 
	4\,\epsilon\, \trace \, ( J^TB^T\tilde{D}X_v)
 	+
	o(\epsilon).
	$
Here, $o(\epsilon)$ denotes higher-order terms in $\epsilon$, whereas $\tX_v$, $\tilde{D}$, and $\tP$ are obtained by perturbing Eqs.~\eqref{eq.GStemp1},~\eqref{eq.GStemp2}, and~\eqref{eq.lyapKP}, respectively,
\begin{subequations}
	\begin{align}\label{eq.GStemp7}
	\!\!
	\cA_K(\tX_v) 
	& \;=\;
	B\tK X_v + X_v \tK^TB^T
	\\[0.cm]	\label{eq.GStemp8}
	\!\!
	\cA_K^*(\tilde{D})
	& \; = \;
	\tK^TB^TD + DB\tK \; +\;
	\cA_K^*(\tilde{D})
	\; =\;  
	J^T(B^T\tP-R\tK) + (B^T\tP - R\tK)^TJ
	\\
	\label{eq.GStemp9}
	\!\!
	\cA_K^*(\tP)
	&\;=\;
	 \tK^TB^TP + P B \tK - K^TR\tK - \tK^TRK.
	\end{align}
\end{subequations}
Applying the adjoint identity on Eqs.~\eqref{eq.GStemp7} and~\eqref{eq.GStemp8} yields
	$
	j_v(K+\epsilon\tK) - j_v(K)
	\approx
	2 \epsilon\,\trace \, ((B\tK X_v + X_v \tK^TB^T)W_1 )
	- 
	2 \epsilon\, \trace \, ( (\tK^TB^TD + DB\tK + J^T(B^T\tP-R\tK) 
	+
	(B^T\tP - R\tK)^TJ)W_2 )
	=
	4 \epsilon\,\trace \, (\tK^TB^TW_1X_v )
	-
	4 \epsilon\, \trace \, ( \tK^T(B^TD-RJ)W_2 )  - 4 \epsilon\,\trace \, (W_2J^TB^T\tP ),
	$
where we have neglected $o (\epsilon)$ terms, and $W_1$ and $W_2$ are given by~\eqref{eq.GStemp4} and~\eqref{eq.GStemp5}, respectively. Moreover, the adjoint identity applied to~\eqref{eq.GStemp9} allows us to simplify the last term~as,
	$
	2 \, \trace \, (W_2J^TB^T\tP )
	=
	\trace \, ((\tK^TB^TP + P B \tK - K^TR\tK - \tK^TRK) W_3 ),
	$
where $W_3$ is given by~\eqref{eq.GStemp6}.
Finally, this yields
	$
j(K+\epsilon\tK)-j(K)
	\approx
	4\epsilon
	\,  
	\trace \, (\tK^T ((RK -B^TP )W_3
	+ 
	B^TW_1X_v 
	+
	(RJ - B^TD)W_2 ) ).
	$
	\end{proof}

We next establish a bound on $\norm{\nabla j_v(K)}_F$.
\begin{mylem}\label{eq.hesnablaBound}
	Let $K, K'\in\R^{m\times n}$ be such that the line segment $K+t(K'-K)$ with $t\in[0,1]$ belongs to $\cSK(a)$ and let $J\in\R^{m\times n}$ and $v\in\R^n$ be fixed. Then, the function $j_v(K)$ satisfies
	$
	|j_v(K_1)-j_v(K_2)|
	\le
	\ell(a) \norm{J}_F^2 \norm{v}^2  \norm{K_1-K_2}_F,
	$
	where $l(a)$ is a positive function given by
	\begin{align}\label{eq.ell}
	\ell(a)
	\, \DefinedAs \,
	c a^2 \, + \, c'a^4
	\end{align}
	and $c$, $c'$ are positive scalars that depend only on problem data. 
\end{mylem}
\begin{proof}
	We show that the gradient $\nabla j_v(K)$ given by Lemma~\ref{lem.hesnabla} is upper bounded by $\norm{\nabla j_v(K)}_F\le \ell(a) \norm{J}_F^2 \norm{v}^2$. Applying Lemma~\ref{lem.ucAinv} on~\eqref{eq.Gstemp_12}, the bounds in Lemma~\ref{lem.bounds}, and the triangle inequality, we have
	$
	\norm{X_v}_F
	\le
	c_1 a \norm{v}^2
	$
and
	$
	\norm{D}_F
	\le
	c_2 a^2
	\norm{J}_F,
	$
where $c_1$ and $c_2$ are positive constants that depend on problem data. We can use the same technique to bound the norms of $W_i$ in Eq.~\eqref{eq.GStemp_456},
	$
	\norm{W_1}_F
	\le
	(c_3a + c_4 a^3 )\norm{J}^2_F,
	$
	$
	\norm{W_2}_F
	\le
	c_5 a^2 \norm{v}^2 \norm{J}_F,
	$
	$
	\norm{W_3}_F
	\le
	c_6 a^3 \norm{v}^2 \norm{J}^2_F,
	$
where $c_3,\ldots,c_6$ are positive constants that depend on problem data. Combining these bounds with the Cauchy-Schwartz and  triangle inequalities applied to $\nabla f_v(K)$  completes the proof.
\end{proof}

\subsubsection*{Proof of Proposition~\ref{prop.thirdOrder}}
Since $r\le r(a)$, Lemma~\ref{lem.cSK2a} implies that $K\pm sU\in\cSK(2a)$ for all $s\le r$. Also, the mean-value theorem implies that, for any $U\in\R^{m\times n}$ and $v\in\R^n$, 
	\begin{align*}
	f_{v}(K\pm rU) 
	\;=\;
	f_v(K)
	\; \pm \; 
	r\inner{\nabla f_v(K)}{U}
	\; + \;
	\dfrac{r^2}{2} \inner{U}{\nabla^2 f_v(K\pm\, s_{\pm}\,U;U)}
	\end{align*}
where $s_{\pm}\in[0,r]$ are constants that depend on $K$ and $U$.  Now, if $\norm{U}_F=\sqrt{mn}$, the above identity yields
\begin{align*}
	\dfrac{1}{2r}{\left(f_{v}(K+rU)- f_v(K-rU)\right)}
	-
	\inner{\nabla f_v(K)}{U} 
	\,
	&\;=\;
	\dfrac{r}{4} \, (  \inner{U}{\nabla^2 f_v(K+ s_{+} U;U)} - \inner{U}{\nabla^2 f_v(K- s_{-} U;U)}  )
	\\[0.15cm]
	&\;\le\;
	\dfrac{r}{4} \, (s_++s_-) \norm{U}_F^3 \, \ell(2a) \, \norm{v}^2
	\, \le \, 
	\dfrac{r^2}{2}mn\sqrt{mn} \, \ell(2a) \, \norm{v}^2
\end{align*}
where the first inequality follows from  Lemma~\ref{eq.hesnablaBound}. Combining this inequality with the triangle inquality applied to the definition of $\widehat{\nabla}f(K)-\widetilde{\nabla}f(K)$  completes the proof.

	\vspace*{-2ex}
\subsection{Proof of Proposition~\ref{prop.linearConvG}}
\label{app.approxGD}
From inequality~\eqref{eq.gAs2}, it follows that $G$ is a descent direction of the function $f(K)$. Thus, we can use the descent lemma~\cite[Eq.~(9.17)]{boyvan04} to show that  $K^+ \DefinedAs K - \alpha G$ satisfies
\begin{align}\label{eq.desLem}
	\!
	f(K^+) \, - \, f(K)
	\, \le \,
	( L_f \alpha^2 / 2 ) \, \norm{ G}_F^2
	\, - \, 
	\alpha \inner{\nabla f(K)}{G}
\end{align}
for any $\alpha$ for which the line segment between $K^+$ and $K$ lies in $\cSK(a)$. Using~\eqref{eq.gAs}, for any $\alpha \in [0, 2\mu_1/(\mu_2L_f) ]$, we have
\begin{align}\label{eq.desLemTemp1}
	( L_f\alpha^2 / 2) \, \norm{ G}_F^2 \, - \, \alpha\inner{\nabla f(K)}{G}
	\; \le\;
	( \alpha \, ( L_f \mu_2 \alpha -  2 \mu_1 ) / 2 ) \, \norm{\nabla f(K)}_F^2
	\;\le\;
	0
\end{align} 
and the right-hand side of inequality~\eqref{eq.desLem} is nonpositive for $\alpha\in[0,2\mu_1/(\mu_2L_f)]$. Thus, we can use the continuity of the function $f(K)$ along with inequalities~\eqref{eq.desLem} and~\eqref{eq.desLemTemp1} to conclude that $K^+\in\cSK(a)$ for all  $\alpha \in [0,2\mu_1/(\mu_2L_f)]$, and 
	$
	f(K^+) - f(K)
	\le
	( \alpha \, ( L_f \mu_2 \alpha -  2 \mu_1 ) / 2 ) \, \norm{\nabla f(K)}_F^2.
	$
Combining this inequality with the PL condition~\eqref{eq.GradDom}, it follows that, for any $\alpha\in[0,c_1/(c_2 L_f)]$, 
	\begin{align*}
	f(K^+) - f(K)
	\;\le\;
	- (\mu_1\alpha / 2 ) \, \norm{\nabla f(K)}_F^2
	\le
	- \mu_f \mu_1\alpha \, (f(K)-f(K^\star)).
	\end{align*}
Subtracting $f(K^\star)$ and rearranging terms complete the proof.

	\vspace*{-2ex}
\subsection{Proofs of Section~\ref{subsec.varianceMu1}}
\label{app.varianceMu1}
We first present two technical results. Lemma~\ref{lem.randomMatUniform} extends~{\cite[Theorem 3.2]{rudver13}} on the norm of Gaussian matrices  presented in Appendix~\ref{app.probabilistic} to random matrices with uniform distribution on the sphere $\sqrt{mn} \, S^{mn-1}$.
\begin{mylem}\label{lem.randomMatUniform}
	Let $E\in\R^{m\times n}$ be a fixed matrix and let $U\in\R^{m\times n}$ be a random matrix with $\vect(U)$ uniformly distributed on the sphere $\sqrt{mn} \, S^{mn-1}$. Then, for any $s \ge 1$ and $t \ge 1$, we have
	$\bbP(\sfB)\le 2\mre^{-s^2q-t^2n}
	+
	\mre^{-mn/8}
	$, 
	where
	$
	\sfB
	\DefinedAs
	\left\{\norm{E^TU}_2>c'\left(s\norm{E}_F+t\sqrt{n}\norm{E}_2\right)\right\},
	$
and  $q \DefinedAs \norm{E}_F^2/{\norm{E}_2^2} $ is the stable rank of ${E}$.
\end{mylem}
\begin{proof}
	 For a matrix  $G$ with i.i.d.\ standard normal entries, we have
	$
	\norm{E^TU}_2\sim\sqrt{mn}\norm{E^TG}_2/{\norm{G}_F}.
	$ 
	Let the constant $\kappa$ be the $\psi_2$-norm of the standard normal random variable and let us define two auxiliary events,
	$
	\sfC_1
	\DefinedAs
	\{
	\sqrt{mn}> 2\norm{G}_F
	\}
	$
and		
	$
	\sfC_0
	\DefinedAs
	\{
	\sqrt{mn} \, \norm{E^TG}_2>2c\kappa^2\norm{G}_F\left(s\norm{E}_F+t\sqrt{n}\norm{E}_2\right)
	\}.
	$	
For $c'\DefinedAs 2c\kappa^2$, we have
	$
	\bbP(\sfB)
	=
	\bbP(\sfC_0)
	\le
	\bbP(\sfC_1\cup \sfA)
	\le
	\bbP(\sfC_1) + \bbP(\sfA),
	$
where the event $\sfA$ is given by Lemma~\ref{lem.randomMatrix}. Here, the first inequality follows from $\sfC_0\subset \sfC_1\cup \sfA$ and the second follows from the union bound.
	Now, since $\norm{\cdot}_F$ is Lipschitz continuous with parameter $1$, from the concentration of Lipschitz functions of standard normal Gaussian vectors~\cite[Theorem 5.2.2]{ver18}, it follows that
	$\bbP(C_1)\le\mre^{- mn / 8}$. This in conjunction with Lemma~\ref{lem.randomMatrix} complete the proof.  
\end{proof}
\begin{mylem}\label{lem.randMatrixUnifFro}
	In the setting of Lemma~\ref{lem.randomMatUniform},  we have
	$
	\bbP
	\left\{\norm{{E}^T{U}}_F>2\sqrt{n} \, \norm{{E}}_F\right\}
	\le
	\mre^{-n/2}.
	$
\end{mylem}
\begin{proof}
	We begin by observing that
	$
	\norm{{E}^T{U}}_F
	=
	\norm{\vect({E}^T{U})}_F
	=
	\norm{\left({I}\otimes{E}^T\right)\vect(U)}_F,
	$
	where $\otimes$ denotes the Kronecker product.
	Thus, it is easy to verify that $\norm{{E}^T{U}}_F$ is a Lipschitz continuous function of ${U}$ with parameter $\norm{{I}\otimes{E}^T}_2=\norm{{E}}_2$. Now, from the concentration of Lipschitz functions of uniform random variables on the sphere $\sqrt{mn} \, S^{mn-1}$~\cite[Theorem 5.1.4]{ver18}, for all $t>0$, we have
	$
	\bbP
	\,
	\big\{\norm{{E}^T{U}}_F > \sqrt{\EX\left[\norm{{E}^T{U}}_F^2\right]} + t \big\}
	\le
	\mre^{-{t^2}/( 2\norm{E}_2^2)}.
	$
Now, since 
	$
	\EX \, [\norm{E^T U}^2_F ] 
	=
	\EX \, [ \norm{({I}\otimes{E}^T ) \, \vect(U)}_F^2 ]
	=
	\EX \, [\trace \, ((I\otimes E^T)\vect(U)\vect(U)^T(I\otimes E) ) ]
	=
	\trace \, ( (I\otimes E^T)(I\otimes E) )
	=
	n \norm{E}_F^2,
	$
we can rewrite the last inequality for $t=\sqrt{n}\norm{E}_F$ to obtain
	\[
	\bbP \, \{ \norm{E^TU}_F > 2\sqrt{n} \, \norm{E}_F \}
	\, \le \,
	\mre^{-n \norm{E}_F^2 / (2\norm{E}_2^2)}
	\, \le \,
	\mre^{- {n}/{2}}
	\]
	where the last inequality follows from
	$
	\norm{E}_F \ge \norm{E}_2.
	$
\end{proof}

\subsubsection*{Proof of Lemma~\ref{lem.Solt1}}
We define the auxiliary events
	\begin{align*}
\sfD_i
	\;\DefinedAs\;
	\{ 
	\norm{\cM^*(E^TU_i)}_2	\le c \sqrt{n} \, \norm{M^*}_S\norm{E}_F
	\}
	\cap
	\{
	\norm{\cM^*(E^TU_i)}_F
	\;\le\; 2\sqrt{n}
	\norm{M^*}_2\norm{E}_F
	\}
	\end{align*}
for $i=1,\ldots,N$. Since $\norm{\cM^*(E^TU_i)}_2\le \norm{\cM^*}_S \norm{E^TU_i}_2$ and $\norm{\cM^*(E^TU_i)}_F\le \norm{\cM^*}_2 \norm{E^TU_i}_F$, we have
	$
	\bbP(\sfD_i)
	\ge
	\{
	\norm{E^TU_i}_2
	\le c\sqrt{n}
	\norm{E}_F
	\}
	\cap
	\{
	\norm{E^TU_i}_F
	\le 2\sqrt{n}
	\norm{E}_F
	\}.
	$
Applying Lemmas~\ref{lem.randomMatUniform} and~\ref{lem.randMatrixUnifFro} to the right-hand side of the above events together with the union bound yield
	$
	\bbP(\sfD_i^c)
	\le
	2\mre^{-n}
	+
	\mre^{-{mn}/{8}}
	+
	\mre^{-{n}/{2}}
	\le
	4 \mre^{-{n}/{8}},
	$
where $\sfD_i^c$ is the complement of $\sfD_i$. This in turn implies 
\begin{align}\label{eq.mProof6}
\bbP(\sfD^c)
\;=\;
\bbP(\bigcup_{i \, = \, 1}^N \sfD_i^c)
\;\le\;
\sum_{i \, = \, 1}^N\bbP(\sfD_i^c)
\;\le\;
4 N \mre^{-\tfrac{n}{8}}
\end{align}
where  $\ds\sfD \DefinedAs \cap_{i} \, \sfD_i$. We can now use the conditioning identity to bound the failure probability,
	\begin{align}\nonumber
	\bbP\{|a|> b\}
	&\;=\;
	\bbP\left\{|a|> b\,\big|\, \sfD\right\}\bbP(\sfD) + \bbP\left\{|a|> b\,\big|\, \sfD^c\right\}\bbP(\sfD^c)
	\\[0.cm] \nonumber
	&\;\le\;
	\bbP\left\{|a|> b\,\big|\, \sfD\right\}\bbP(\sfD) +\bbP(\sfD^c)
	\\[0.cm] \nonumber
	&\;=\;
	\bbP\left\{|a\,\one_{\sfD}|\,> b\right\} \,+\, \bbP(\sfD^c)
	\\[0.cm] \label{eq.mProof1}
	&\;\le\;
	\bbP\left\{|a\,\one_{\sfD}|> b\right\} \,+\, 4N\mre^{-{n}/{8}}
	\end{align}
	where 
	$
	a
	\DefinedAs
	(1 / N )\sum_{i} \inner{E\left(X_i-X\right)}{ U_i} \inner{ E X}{ U_i},
	$
	$
	b
	\DefinedAs
	\delta \norm{E X}_F\norm{E}_F
	$, 
and $\one_{\sfD}$ is the indicator function of  $\sfD$. 
It is now easy to verify that 
	$
	\bbP\{|a\,\one_\sfD|> b\}
	\le
	\bbP\left\{ \abs{Y}> b \right\},
	$
where 
	$
	Y
	\DefinedAs
	( 1 / N )\sum_{i} Y_i,
	$
	$
	Y_i
	\DefinedAs
	\inner{E (X_i-X )}{ U_i} \inner{ E X}{ U_i} \one_{\sfD_i}.
	$
The rest of the proof uses the $\psi_{{1}/{2}}$-norm of $Y$ to establish an upper bound on
$
\bbP\left\{|Y|> b\right\}
$. 

Since $Y_i$ are linear in the zero-mean random variables $X_i-X$, we have 
$\EX[Y_i|U_i]=0$. Thus, the law of total expectation yields
	$
	\EX[Y_i]
 	=
	\EX\left[\EX[Y_i|U_i]\right]
	=	
	0.
	$

Therefore, Lemma~\ref{lem.talagrand} implies
	\begin{align}\label{eq.mProof7}
	\norm{Y }_{\psi_{{1}/{2}}}
	\le
	(c'/\sqrt{N}) (\log N ) \max_i \norm{Y_i}_{\psi_{{1}/{2}}}.
	\end{align}
	Now, using the standard properties of the $\psi_\alpha$-norm, we have
	\begin{align}
	\norm{Y_i}_{\psi_{{1}/{2}}}
	&\;\le\;
	c^{\prime\prime}
	\norm{\inner{E\left(X_i-X\right)}{ U_i}\one_{\sfD_i}}_{\psi_1} \norm{\inner{ E X}{ U_i}}_{\psi_1}
	\;\le\; 
	c^{\prime\prime\prime}
	\norm{\inner{E\left(X_i-X\right)}{ U_i}\one_{\sfD_i}}_{\psi_1} \norm{ E X}_F
	\label{eq.mProof3}
	\end{align}
where the second inequality follows from~\cite[Theorem~3.4.6]{ver18},
\begin{align}\label{eq.unifMarg}
\norm{\inner{ E X}{ U_i}}_{\psi_1} 
\;\le\;
\norm{\inner{ E X}{ U_i}}_{\psi_2}
\;\le\;
c_0\norm{ E X}_F.
\end{align}
We can now use 
	$
	\langle {E (X_i-X )},{ U_i} \rangle
	=
	\langle {X_i-X},{ E^TU_i} \rangle
	=
	\langle {\cM(x_ix_i^T)},{ E^TU_i} \rangle
	-
	\langle {\cM (I)},{ E^TU_i} \rangle
	= 
	x_i^T\cM^{*}(E^TU_i) x_i -
	\trace \, (\cM^{*}(E^TU_i))
	$
to bound the right-hand side of~\eqref{eq.mProof3}. This identity allows us to use the Hanson-Write inequality (Lemma~\ref{lem.HansonWright}) to upper bound the conditional probability
\begin{align*}
\bbP\left\{
\abs{\inner{E\left(X_i-X\right)}{ U_i}}
> t 
\;\big|\;
U_i
\right\}
\;\le\;
2\mre^{-\hat{c}\min\{\tfrac{t^2}{\kappa^4\norm{\cM^{*}(E^TU_i)}_F^2},\,\tfrac{t}{\kappa^2\norm{\cM^{*}(E^TU_i)}_2}\}}.
\end{align*}
Thus, we have 
\begin{align*}
\bbP\left\{
\abs{\inner{E\left(X_i-X\right)}{ U_i}\one_{\sfD_i}} 
>
t
\right\}
&\;=\;
\EX_{U_i}\left[
\one_{\sfD_i} 
\EX_{x_i}\left[
\one_{\{\abs{\inner{E\left(X_i-X\right)}{ U_i}}> t\}}
\right]
\right]
\\[0.cm]
&\;=\;
\EX_{U_i}\left[
\one_{\sfD_i} 
\bbP\left\{
\abs{\inner{E\left(X_i-X\right)}{ U_i}}
> t 
\;\big|\;
U_i
\right\}
\right]
\\[0.cm]
&\;\le\;
\EX_{U_i}\left[
\one_{\sfD_i} 
2\mre^{-\hat{c}\min\{\tfrac{t^2}{\kappa^4\norm{\cM^*(E^TU_i)}_F^2}\,\tfrac{t}{\kappa^2\norm{\cM^*(E^TU_i)}_2}\}}
\right]
\\[-0.1cm]
&\;\le\;
2\mre^{-\hat{c}\min\{\tfrac{t^2}{4n\kappa^4\norm{\cM^*}_2^2\norm{E}_F^2}\,\tfrac{t}{c\sqrt{n}\kappa^2\norm{\cM^*}_S\norm{E}_F}\}}
\end{align*}
where the definition of $\sfD_i$ was used to obtain the last inequality. The above tail bound implies~\cite[Lemma~11]{soljavlee19}
\begin{align}\label{eq.SolProoftemp1}
\norm{\inner{E\left(X_i-X\right)}{ U_i}\one_{\sfD_i}}_{\psi_{1}}
\;\le\;
\tilde{c}\kappa^2\sqrt{n}(\norm{\cM^*}_2+\norm{\cM^*}_S)\norm{E}_F.
\end{align} 
Using~\eqref{eq.sample}, it is easy to obtain the lower bound on the number of samples, 
	$
	N 
	\ge  
	C' \, (\beta^2 \kappa^2/ \delta)^2
	( \norm{\cM^*}_2 \, + \, \norm{\cM^*}_S )^2 \, n \, \log^6 \! {N} .
	$
We can now combine~\eqref{eq.mProof7},~\eqref{eq.SolProoftemp1} and~\eqref{eq.mProof3} to obtain
\begin{align*}
\norm{Y}_{\psi_{{1}/{2}}}
&\, \le \,
C'\kappa^2 \dfrac{\sqrt{n}\log N}{\sqrt{N}}
(\norm{\cM^*}_2+\norm{\cM^*}_S)\norm{E}_F \norm{EX}_F
	\; \le \;
\dfrac{\delta}{\beta^2\log^2 \! N} \, \norm{E}_F \norm{EX}_F
\end{align*}
where the last inequality follows from the above lower bound on $N$.
 Combining this inequality and~\eqref{eq.pollard} with
$t\DefinedAs\delta\norm{E}_F\norm{EX}_F/\norm{Y}_{\psi_{{1}/{2}}}$
yields
	$
	\bbP
	\{
	\abs{Y} > \delta \norm{E}_F \norm{EX}_F
	\}
	\le
	{1}/{N^\beta},
	$ 
which completes the proof. 

\subsubsection*{Proof of Lemma~\ref{lem.Solt2}}
The marginals of a uniform random variable have bounded sub-Gaussian norm (see inequality~\eqref{eq.unifMarg}). Thus,~\cite[Lemma 2.7.6]{ver18} implies
	$
	\norm{\inner{W}{U_i}^2}_{\psi_1}
	=
	\norm{\inner{W}{U_i}}_{\psi_2}^2
	\le	
	\hat{c}\norm{W}_F^2,
	$
which together with the triangle inequality yield
	$
	\norm{\inner{W}{U_i}^2-\norm{W}_F^2}_{\psi_1}
 	\le
	c'\norm{W}_F^2.
	$
Now since $\inner{W}{U_i}^2-\norm{W}_F^2$ are zero-mean and independent, we can apply the Bernstein inequality (Lemma~\ref{lem.Bern}) to obtain
\begin{align}\label{eq.unifRow}
\bbP
\left\{
\abs{\dfrac{1}{N}\sum_{i \, = \, 1}^N \inner{W}{ U_i}^2 -\norm{W}_F^2}
>
t \norm{W}_F^2
\right\}
	\, \le \;
2\mre^{-c N \! \min\{t^2,t\}}
\end{align}
which together with the triangle inequality complete the proof.

	\vspace*{-2ex}
\subsection{Proofs for Section~\ref{subsec.varianceMu2} and probabilistic toolbox}
\label{app.probabilistic}
We first present a technical lemma.
\begin{mylem}\label{lem.Moh1}
	Let $v_1,\ldots,v_N\in\R^{d}$ be i.i.d.\ random vectors uniformly distributed on the sphere $\sqrt{d} \, S^{d-1}$ and let $a\in\R^d$ be a fixed vector. Then, for any $t\ge 0$, we have
	\begin{align*}
	\bbP \left\{ \dfrac{1}{N} \norm{\sum_{j \, = \, 1}^{N} \inner{a}{v_j} v_j} \,>\, (c+c \dfrac{\sqrt{d}+t}{\sqrt{N}})\norm{a} \right\}
	\;\le\;
	2\mre^{-t^2}+N\mre^{-{d}/{8}}+ 2\mre^{-\hat{c}N}.
	\end{align*}
\end{mylem}
\begin{proof}
	It is easy to verify that
	$
	\sum_{j} \inner{a}{v_j} v_j
	=
	V v,
	$
	where $V \DefinedAs [v_1\,\cdots\, v_n]\in\R^{d\times N}$ is the random matrix with the $j$th column given by $v_j$ and $v \DefinedAs V^T a \in \R^N$. 
	Thus,   
	$
	\norm{\sum_{j} \inner{a}{v_j} v_j}
	=
	\norm{V v}
	\le
	\norm{V}_2 \norm{v}.
	$
	Now, let $G\in\R^{d\times N}$ be  a random matrix with i.i.d.\ standard normal Gaussian entries and let $\hat{G}\in\R^{d\times N}$ be a matrix obtained by normalizing the columns of $G$ as $\hat{G_j} \DefinedAs \sqrt{d} \, G_j/\norm{G_j}$, where $G_j$ and $\hat{G}_j$ are the $j$th columns of $G$ and $\hat{G}$, respectively. From the concentration of norm of Gaussian vectors~\cite[Theorem 5.2.2]{ver18}, we have $\norm{G_j}\ge \sqrt{d} / 2$ with probability not smaller than $1-\mre^{-{d}/{8}}$. This in conjunction with a union bound yield
	$
	\norm{\hat{G}}_2 \le 2\norm{G}_2
	$
	with probability not smaller than $1-N\mre^{-{d}/{8}}$.
	Furthermore, from the concentration of Gaussian matrices~\cite[Theorem~4.4.5]{ver18}, we have
	$
	\norm{G}_2 \le C(\sqrt{N} + \sqrt{d} + t)
	$
	with probability not smaller than $1-2\mre^{-t^2}$. 
By combining this inequality with the above upper bound on $\norm{\hat{G}}_2$, and using
	$V \sim \hat{G}$ in conjunction with a union bound, we obtain
	\begin{align}\label{eq.normBoundUniformTemp2}
	\norm{V}_2 \; \le \; 2C(\sqrt{N} \, + \, \sqrt{d} \, + \, t)
	\end{align} 
	with probability not smaller than $1-2\mre^{-t^2}-N\mre^{-{d}/{8}}$.
	Moreover, using~\eqref{eq.unifRow} in the proof of Lemma~\ref{lem.Solt2}, gives
	$
	\norm{v} \le C'\sqrt{N}\norm{a}
	$
	with probability not smaller than $1-2\mre^{-\hat{c}N}$. Combining this inequality with~\eqref{eq.normBoundUniformTemp2} and employing a union bound complete the proof. 
\end{proof}

\subsubsection*{Proof of Lemma~\ref{lem.Moh3}}
We begin by noting that
\begin{align}\label{eq.mProof4}
\norm{\sum_{i=1}^N \inner{E\left(X_i-X\right)}{ U_i}  U_i}_F
\;=\;
\norm{Uu}
\;\le\;
\norm{U}_2\norm{u}
\end{align}
where $U\in\R^{mn\times N}$ is a matrix with the $i$th column $\vect({U_i})$ and $u\in\R^N$ is a vector with the $i$th entry $\inner{E\left(X_i-X\right)}{ U_i}$. Using~\eqref{eq.normBoundUniformTemp2} in the proof of Lemma~\ref{lem.Moh1}, for $s\ge 0$, we have
\begin{align}\label{eq.mProof5}
\norm{U}_2 \; \le \; c(\sqrt{N} \, + \, \sqrt{mn} \, + \, s)
\end{align} 
with probability not smaller than $1-2\mre^{-s^2}-N\mre^{-{mn}/{8}}$. To bound the norm of $u$, we use similar arguments as in the proof of Lemma~\ref{lem.Solt1}. In particular, let $\sfD_i$ be defined as above and let $\ds\sfD \DefinedAs \cap_{i} \, \sfD_i$. Then for any $b\ge 0$, 
	\begin{align} 
	\label{eq.mProof7-8}
	\bbP\{\norm{u}> b\}
	&\;\le\;
	\bbP\left\{\norm{u \one_{\sfD}}> b\right\} \,+\, 4N\mre^{-{n}/{8}}
	\end{align}
	where $\one_{\sfD}$ is the indicator function of  $\sfD$; cf.~\eqref{eq.mProof1}.
	Moreover, it is straightforward to verify that 
	$
	\norm{u\one_{\sfD}}
	\le
	\norm{z},
	$
where the entries of $z\in\R^N$ are given $z_i=u_i\one_{\sfD_i}$. Since
	$
	\norm{\norm{z}^2}_{\psi_{1/2}}
	=
	\norm{\sum_{i} z_i^2 }_{\psi_{1/2}},
	$
we have
\begin{align*}
\norm{\sum_{i \, = \, 1}^N z_i^2 }_{\psi_{1/2}}
&\;\overset{(a)}{\le}\;
\norm{\sum_{i \, = \, 1}^N z_i^2 \, - \, \EX[z_i^2] }_{\psi_{1/2}}
	\, + \,
N\norm{\EX[z_1^2]}_{\psi_{1/2}}
\\[-0.1cm]
&\;\overset{(b)}{\le}\;
\bar{c}_1\norm{z_1^2}_{\psi_{1/2}} \sqrt{N} \log N  
	\, + \,
\bar{c}_2 N \norm{z_1}_{\psi_{1}}^2
\\[0.cm]
&\;\overset{(c)}{\le}\;
\bar{c}_3 N   \norm{z_1}_{\psi_{1}}^2
\\[0.cm]
&\;\overset{(d)}{\le}\;
\bar{c}_4 N \kappa^4 n (\norm{\cM^*}_2+\norm{\cM^*}_S)^2\norm{E}_F^2.
\end{align*}
Here, $(a)$ follows from the triangle inequality, $(b)$ follows from combination of Lemma~\ref{lem.talagrand}, applied to the first term, and 
$
\EX[z_1^2]
\le
\tilde{c}_0\norm{z_1}_{\psi_{1}}^2
$~(e.g., see~\cite[Proposition~2.7.1]{ver18})
applied to the second term, $(c)$ follows from $
\norm{z_1^2}_{\psi_{1/2}}
\le
\tilde{c}_1
\norm{z_1}_{\psi_{1}}^2
$, and $(d)$ follows from~\eqref{eq.SolProoftemp1}.
This allows us to use~\eqref{eq.pollard} with $\xi = \norm{z}^2$ and $t = r^2$ to obtain
	$
	\bbP
	\{
	\norm{z} 
	> 
	r \sqrt{n N }\kappa^2 (\norm{\cM^*}_2+\norm{\cM^*}_S)\norm{E}_F
	\}
	\le
	\bar{c}_5 \mre^{-r},
	$
for all $r>0$.
Combining this inequality with~\eqref{eq.mProof7-8} yield
\begin{align*}
	\bbP	\left\{
	\norm{u} 
	\, > \,
	r \sqrt{n N }\kappa^2 (\norm{\cM^*}_2 \, + \, \norm{\cM^*}_S)\norm{E}_F
	\right\}
	\, \le \; 
	\bar{c}_5\, \mre^{-r}
	\, + \,
	4N\mre^{-{n}/{8}}.
\end{align*}
Finally, substituting $r=\beta\log n$ in the last inequality and letting $s=\sqrt{mn}$ in~\eqref{eq.mProof5} yield
	\begin{multline*}
	\bbP
	\left\{
	\dfrac{1}{N} \norm{\sum_{i} \inner{E\left(X_i-X\right)}{ U_i}  U_i}_F
	>
	c_1\beta\sqrt{m n}\log n \kappa^2 (\norm{\cM^*}_2+\norm{\cM^*}_S)\norm{E}_F
	\right\}
	\;\le\;
	\\[0.cm]
	c_0 n^{-\beta} \,+\, 2\mre^{-mn} + N\mre^{-{mn}/{8}} + 4N\mre^{-{n}/{8}}
	\;\le\;
	c_2 (n^{-\beta} + N \mre^{-{n}/{8}})
	\end{multline*}
where we used inequality~\eqref{eq.mProof4}, $N\ge c_0 n$, and applied the union bound. This completes the proof. 


\subsubsection*{Proof of Lemma~\ref{lem.Moh2}}
This result is obtained by applying Lemma~\ref{lem.Moh1} to the vectors $\vect(U_i)$ and setting $t=\sqrt{mn}$.

\subsubsection*{Probabilistic toolbox}

In this subsection, we summarize known technical results which are useful in establishing bounds on the correlation between the gradient estimate and the true gradient. Herein, we use $c$, $c'$, and $c_i$ to denote positive absolute constants.   
For any positive scalar $\alpha$, the $\psi_\alpha$-norm of a random variable $\xi$ is given by~\cite[Section~4.1]{ledtal13},
	$
	\norm{\xi}_{\psi_\alpha} 
	\DefinedAs
	\inf_t 
	\, 
	\{ t>0 \, | \EX \, [\psi_\alpha ({|\xi|}/{t} ) ] \le 1\},
	$
where $\psi_\alpha(x)\DefinedAs \mre^{x^\alpha}-1$ (linear near the origin when $0<\alpha<1$ in order for $\psi_\alpha$ to be convex) is an Orlicz function. Finiteness of the $\psi_\alpha$-norm implies the tail bound
\begin{align}\label{eq.pollard}
\bbP\left\{
|\xi| \, > \, t \norm{\xi}_{\psi_\alpha}
\right\}
\;\le\;
c_\alpha \mre^{-t^\alpha}
~
\mbox{for all}~t \, \ge \, 0
\end{align}
where $c_\alpha$ is an absolute constant that depends on $\alpha$; e.g., see~\cite[Section~2.3]{pol15} for a proof. The random variable $\xi$ is called sub-Gaussian if its distribution is dominated by that of a normal random variable. This condition is equivalent to
$
\norm{\xi}_{\psi_2}
<
\infty. 
$
The random variable $\xi$ is sub-exponential if $
\norm{\xi}_{\psi_1}
<
\infty. 
$
It is also well-known that for any random variables $\xi$ and $\xi'$ and any positive scalar $\alpha$, 
	$
	\norm{\xi\,\xi'}_{\psi_{\alpha}}
	\le
	\hat{c}_{\alpha}
	\norm{\xi}_{\psi_{2\alpha}}\norm{\xi'}_{\psi_{2\alpha}}
	$
and the above inequality becomes equality with $c_\alpha=1$ if $\alpha\ge1$.
\begin{mylem}[Bernstein inequality~{\cite[Corollary 2.8.3]{ver18}}]\label{lem.Bern}
	Let $\xi_1,\ldots,\xi_N$ be independent, zero-mean, sub-exponential random variables with $\kappa\ge\norm{\xi_i}_{\psi_1}$. Then, for any scalar $t\ge0$, 
	$
	\bbP
	\{
	\abs{ (1/N) \sum_{i} \xi_i}
	>t
	\}
	\le
	2 \mre^{-c N\min\{{t^2}/{\kappa^2},{t}/{\kappa}\}}.
	$
\end{mylem}
\begin{mylem}[Hanson-Wright inequality~{\cite[Theorem 1.1]{rudver13}}]
	\label{lem.HansonWright}
	Let $A\in\R^{N\times N}$ be a fixed matrix and let $x\in\R^N$ be a random vector with independent entries that satisfy $\EX[x_i]=0$, $\EX[x_{i}^2]=1$, and $\norm{x_i}_{\psi_2}\le \kappa$. Then, for any nonnegative scalar $t$, we have
	$
	\bbP
	\{
	\abs {x^T A\, x - \EX[x^TA\,x]}
	>
	t
	\}
	\le
	2 \mre^{- c \min \{t^2/(\kappa^4\norm{A}_F^2), t/( \kappa^2\norm{A}_2) \}}.
	$
\end{mylem}
\begin{mylem}[Norms of random matrices~{\cite[Theorem 3.2]{rudver13}}]
	\label{lem.randomMatrix} 
	Let $E\in\R^{m\times n}$ be a fixed matrix and let $G\in\R^{m\times n}$ be a random matrix with independent entries that satisfy $\EX[G_{ij}]=0$, $\EX[G_{ij}^2]=1$, and $\norm{G_{ij}}_{\psi_2}\le \kappa$. Then, for any scalars $s,t\ge 1$, $\bbP(\sfA)\le 2\mre^{-s^2q-t^2n} $, where $q \DefinedAs \norm{E}_F^2/{\norm{E}_2^2} $ is the stable rank of ${E}$ and
	$
	\sfA
	\DefinedAs
	\{ \norm{E^TG}_2>c\kappa^2\left(s\norm{E}_F+t\sqrt{n}\norm{E}_2\right) \}.
	$
\end{mylem}
	
The next lemma provides us with an upper bound on the $\psi_{\alpha}$-norm of sum of random variables that is by Talagrand. It is a straightforward consequence of combining~\cite[Theorem~6.21]{ledtal13} and~\cite[Lemma~2.2.2]{vaawel96}; see e.g.~\cite[Theorem~8.4]{mawig15} for a formal argument.
\begin{mylem}\label{lem.talagrand}
	For any scalar $\alpha\in(0,1]$, there exists a constant $C_{\alpha}$ such that for any sequence of independent random variables $\xi_1,\ldots,\xi_N$ we have
	$
	\norm{\sum_i \xi_i - \EX\big[\sum_i \xi_i\big]}_{\psi_{\alpha}}
	\le
	C_{\alpha} ( \max_i \norm{\xi_i}_{\psi_{\alpha}} ) \sqrt{N}\log N .
	$
\end{mylem}

	\vspace*{-2ex}
\subsection{Bounds on optimization variables}\label{app.bounds}

Building on~\cite{toi85}, in Lemma~\ref{lem.bounds} we provide useful bounds on $K$, $X=X(K)$, $P = P(K)$, and $Y = K X (K)$.
\begin{mylem}\label{lem.bounds}
	Over the sublevel set $\cSK(a)$ of the LQR objective function $f(K)$, we have
	\begin{subequations}
		\label{eq.bounds}
		\begin{align}
		\label{eq.uX}
		\trace \, (X) 
		&
		\; \le \; 
		{a}/{\lambda_{\min}(Q)}
		\\[-0.1cm]
		\label{eq.uY}
		\norm{Y}_F
		&
		\; \le \; 
		{a}/{ \sqrt{\lambda_{\min}(R) \lambda_{\min}(Q)}}
		\\[-0.1cm]
		\label{eq.lX}
		{\nu}/{a}
		&
		\; \le \; 
		\lambda_{\min}(X)
		\\[-0.1cm]
		\label{eq.uK}
		\norm{K}_F
		&
		\; \le \; 
		{a}/{\sqrt{\nu \lambda_{\min}(R)}}
		\\[-0.1cm]
		\label{eq.uP}
		\trace \, (P)
		&
		\;\le\;
		{a}/{\lambda_{\min}(\Omega)}
		\end{align}
	\end{subequations}
	where the constant $\nu$ is given by~\eqref{eq.nuu}.
\end{mylem}
\subsubsection*{Proof}
For $K\in\cSK(a)$, we have
	\be
	\trace \, (Q X + Y^T R\, Y X^{-1} ) \; \le \; a
	\label{eq.f-sub-a}
	\ee 
which along with $\trace \, (Q X ) \geq \lambda_{\min}(Q) \norm{X^{1/2}}_F^2$ yield~\eqref{eq.uX}. To establish~\eqref{eq.uY}, we combine~\eqref{eq.f-sub-a} with 
	\[
	\trace \, (R\, Y X^{-1} Y^T ) \; \geq \; \lambda_{\min} (R) \norm{YX^{-1/2}}_F^2
	\] 
to obtain 
	$
	\norm{YX^{-1/2}}_F^2 
	\le 
	a/\lambda_{\min}(R).
	$ 
Thus,
$
\norm{Y}_F^2
\le
a \norm{X}_2/\lambda_{\min}(R).
$
This inequality along with~\eqref{eq.uX} give~\eqref{eq.uY}. To show~\eqref{eq.lX}, let $v$ be the normalized eigenvector corresponding to the smallest eigenvalue of $X$. Multiplication of Eq.~\eqref{eq.lyapunovXY} from the left and the right by $v^T$ and $v$, respectively, gives
	$
	v^T (DX^{{1}/{2}} + X^{{1}/{2}}D^T ) \,v 
	= 
	\sqrt{\lambda_{\min}(X)} 
	\, 
	v^T(D + D^T)\,v 
	=
	- v^T \Omega\, v,
	$	
where $D\DefinedAs A X^{1/2} - B Y X^{-1/2}$. Thus,
\begin{align}
\label{eq.lambda-min-XX}
\lambda_{\min}(X)
\;=\;
\dfrac{(v^T \Omega \, v)^2}{(v^T(D + D^T)\,v)^2}
\;\ge\;
\dfrac{\lambda_{\min}^2 (\Omega)}{4\,\norm{D}_2^2}
\end{align}
where we applied the Cauchy-Schwarz inequality on the denominator.  Using the triangle inequality and submultiplicative property of the $2$-norm, we can upper bound $\norm{D}_2$,
	\be
	\ba{rcl}
	\norm{D}_2 
	& \!\!\! \le \!\!\! &
	\norm{A}_2 \norm{X^{1/2}}_2 + \norm{B}_2 \norm{Y X^{-1/2}}_2
	\; \le \;
	\sqrt{a} \, (\norm{A}_2 /\sqrt{\lambda_{\min}(Q)} + \norm{B}_2/\sqrt{\lambda_{\min}(R)} )
	\ea
	\label{eq.bound-DD}
	\ee
where the last inequality follows from~\eqref{eq.uX} and the upper bound on 
	$
	\norm{YX^{-1/2}}_F^2. 
	$  	
Inequality~\eqref{eq.lX}, with $\nu$ given by~\eqref{eq.nuu}, follows from combining~\eqref{eq.lambda-min-XX} and~\eqref{eq.bound-DD}. To show~\eqref{eq.uK}, we use the upper bound on
	$
	\norm{YX^{-1/2}}_F^2, 
	$  
which is equivalent to $\norm{K X^{{1}/{2}}}_F^2 \le a/\lambda_{\min}(R)$, to obtain
	$
	\norm{K}^2_F 
	\le
	{a}/{\lambda_{\min}(R)\lambda_{\min}(X)}
	\le
	{a^2}/(\nu \lambda_{\min}(R)).
	$
Here, the second inequality follows from~\eqref{eq.lX}. Finally, to prove~\eqref{eq.uP}, note that the definitions of $f(K)$ in~\eqref{eq.f} and $P$ in~\eqref{eq.P} imply
	$
	f(K) 
	=
	\trace \, (P \, \Omega).
	$
Thus, from $f(K)\le a$, we have
	$
	\trace \, (P)
	\le
	{a}/{\lambda_{\min}(\Omega)},
	$
which completes the proof. 

	\vspace*{-2ex}
\subsection{A bound on the norm of the inverse Lyapunov operator}
\label{app.invLyapBound}

Lemma~\ref{lem.normLyap} provides an upper bound on the norm of the inverse Lyapunov operator for stable LTI systems. 
\begin{mylem}\label{lem.normLyap}
	For any Hurwitz matrix  $F\in\R^{n\times n}$, the linear map $\cF$: $\bbS^n \rightarrow \bbS^n$ 
	\begin{align}\label{eq.cF}
	\cF(W)
	\;\DefinedAs\;
	\int_0^{\infty} \mre^{F t} \, W\,\mre^{F^Tt} \,\mrd t
	\end{align}
	is well defined and, for any $\Omega \succ 0$,
	\begin{align}
	\label{eq.normLyap}
	\norm{\cF}_2
	\;\le\; 
	\trace\,(\cF(I))
	\;\le\;
	{\trace\,(\cF(\Omega))}/{\lambda_{\min}(\Omega)}.
	\end{align}
\end{mylem}
\begin{proof}
	Using the triangle inequality and the sub-multiplicative property of the Frobenius norm, we can write
	\begin{align}
	\norm{\cF(W)}_F
	\;\le\;
	\int_0^{\infty} \norm{\mre^{F\,t} \, W\,\mre^{F^Tt}}_F\, \mrd t 
	\;\le\;
	\norm{W}_F \int_0^{\infty} \norm{\mre^{F\,t}}_F^2 \,  \mrd t
	\;=\;
	\norm{W}_F \, \trace\left(\cF(I)\right).
	\end{align}
	Thus,
	$
	\norm{\cF}_2
	=
	\max_{\norm{W}_F \, = \, 1}
	\norm{\cF(W)}_F
	\le
	\trace\left(\cF(I)\right),
	$
which proves the first inequality in~\eqref{eq.normLyap}. To show the second inequality, we use the monotonicity of the linear map $\mathcal{F}$, i.e., for any symmetric matrices $W_1$ and $W_2$ with $W_1\preceq W_2$, we have
	$\cF(W_1)\preceq\cF(W_2)$. In particular, $\lambda_{\min}(\Omega) I \preceq \Omega $ implies $\lambda_{\min}(\Omega) \cF(I)\preceq\cF(\Omega)$ which yields 
	$
	\lambda_{\min}(\Omega)\,\trace(\cF(I))
	\le
	\trace(\cF(\Omega))
	$
and completes the proof.
\end{proof}

We next use Lemma~\ref{lem.normLyap} to establish a bound on the norm of the inverse of the closed-loop Lyapunov operator $\cA_K$ over the sublevel sets of the LQR objective function $f(K)$.
\begin{mylem}\label{lem.ucAinv}
	For any $K\in\cSK(a)$, the closed-loop Lyapunov operators $\cA_K$ given by~\eqref{eq.LyapOper} satisfies
	$
	\norm{\cA_K^{-1}}_2
	=
	\norm{( \cA_K^*)^{-1}}_2
	\le
	{a}/{\lambda_{\min}(\Omega)\lambda_{\min}(Q)}.
	$
\end{mylem}
\begin{proof}
	Applying Lemma~\ref{lem.normLyap} with $F=A-BK$ yields
	$
	\norm{\cA_K^{-1}}_2
	=
	\norm{( \cA_K^*)^{-1}}_2
	\le
	{\trace(X)}/{\lambda_{\min}(\Omega)}.
	$
	Combining this inequality with~\eqref{eq.uX} completes the proof. 
\end{proof}

\subsubsection*{Parameter $\theta(a)$ in Theorem~\ref{thm.randomSearchFormal}}
As  discussed in the proof, over any sublevel set $\cSK(a)$ of the function $f (K)$, we require the function $\theta$ in Theorem~\ref{thm.randomSearchFormal} to satisfy
	$
	( \norm{(\cA_K^*)^{-1}}_2+\norm{(\cA_K^*)^{-1}}_S)/\lambda_{\min}(X)
	\le
\theta(a)
	$
for all $K\in\cSK(a)$. Clearly, Lemma~\ref{lem.ucAinv} in conjunction with Lemma~\ref{lem.bounds} can be used to obtain
	$
	\norm{( \cA_K^* )^{-1}}_2
	\le
	{a}/( \lambda_{\min}(Q)\lambda_{\min}\Omega )
	$
and	
	$
	\lambda^{-1}_{\min}(X)
	\le
	a/\nu,
	$
where $\nu$ is given by~\eqref{eq.nuu}. The existence of $\theta(a)$, follows from the fact that there is a scalar $M(n)>0$ such that $\norm{\cA}_S\le M \norm{\cA}_2$ for all linear operators $\cA$: $\bbS^n\rightarrow \bbS^n $.

{\renewcommand{\baselinestretch}{1.35}

}

\end{document}